\documentclass[11pt]{amsart}
\usepackage{amsmath,amsthm,amssymb}
\usepackage[latin1]{inputenc}
\usepackage[T1]{fontenc}
\usepackage[all]{xy}
\usepackage{mathrsfs}

\numberwithin{equation}{section}
\newtheorem{thm}[equation]{Theorem}
\newtheorem*{thma*}{Theorem A}
\newtheorem*{thmb*}{Theorem B}
\newtheorem*{thmc*}{Theorem C}
\newtheorem*{thmd*}{Theorem D}
\newtheorem{cor}[equation]{Corollary}
\newtheorem{lem}[equation]{Lemma}
\newtheorem{prop}[equation]{Proposition}

\theoremstyle{definition}

\newtheorem{defn}[equation]{Definition}

\def\R{\mathbb R}

\def\Z{\mathbb Z}
\def\A{\mathbb A}
\def\Q{\mathbb Q}
\def\C{\mathbb C}

\def\K{\mathbb K}
\def\bM{\mathbb M}

\def\g{\mathfrak g}
\def\gl{\mathfrak{gl}}
\def\so{\mathfrak{so}}

\def\s{\mathfrak s}

\def\O{\mathcal O}

\def\k{\mathfrak k}

\def\c{\mathfrak c}

\def\<{\langle}
\def\>{\rangle}

\def\M{\mathcal{M}}
\def\tM{\widetilde{\mathcal{M}}}

\def\cE{\mathcal{E}}
\def\cT{\mathcal{T}}
\def\cH{\mathcal{H}}

\def\GL{{\rm GL}}

\def\SO{{\rm SO}}

\def\Orth{{\rm O}}
\def\Unit{{\rm U}}
\def\G{\mathcal{G}}
\def\T{\mathcal{T}}

\def\Hom{{\rm Hom}}

\def\Whit{\mathcal{W}}
\def\z{\mathfrak{z}}

\def\autc{{\rm Aut}(\C)}
\def\Sym{{\rm Sym}}
\def\Coh{{\rm Coh}} 
\def\ul{\underline}

\usepackage{bm}
\usepackage{upgreek}
\makeatletter
\newcommand{\bfgreek}[1]{\bm{\@nameuse{up#1}}}
\makeatother

\title[Critical values of $L$-functions]{Critical values of 
Rankin--Selberg $L$-functions for $\GL_n \times \GL_{n-1}$ and the symmetric cube $L$-functions for $\GL_2$}
\author{\bf A. Raghuram}
\date{\today}
\address{A. Raghuram, Indian Institute of Science Education and Research, Dr.Homi Bhabha Road, Pashan, Pune 411008, INDIA.}
\email{raghuram@iiserpune.ac.in}

\thanks{This work is partially supported by the National Science Foundation (NSF), award number DMS-0856113, and an Alexander von Humboldt Research Fellowship.}

\begin{document}

\maketitle

\smallskip 
 

\section{\bf Introduction and statements of results}
In a previous article \cite{raghuram-imrn} an algebraicity result for the central critical value for $L$-functions for 
$\GL_n \times \GL_{n-1}$ over $\Q$  was proved assuming the validity of a nonvanishing hypothesis involving archimedean integrals. The purpose of this article is to generalize \cite[Thm.\,1.1]{raghuram-imrn} for all critical values for $L$-functions for $\GL_n \times \GL_{n-1}$ {\it over any number field $F$} while using the period relations of 
\cite{raghuram-shahidi-imrn} and some additional inputs as will be explained below. Thanks to a recent preprint of Binyong Sun \cite{sun}, the nonvanishing hypothesis has now been proved, and so one may claim that the results 
of this article are unconditional. 
Using such results for $\GL_3 \times \GL_2$, new unconditional algebraicity result for the special values of symmetric cube $L$-functions for $\GL_2$ over $F$ have been proved. 
Previously, algebraicity results for the critical values of symmetric cube $L$-functions for $\GL_2$ have been known only in special cases: see Garrett--Harris \cite{garrett-harris}, Kim--Shahidi \cite{kim-shahidi-israel},   
Grobner--Raghuram \cite{grobner-raghuram-ajm}, and 
Januszewski \cite{januszewski}.

\subsection{\bf $L$-functions for $\GL_n \times \GL_{n-1}$}
Suppose $\A_F$ is the ring of ad\`eles of $F.$ Let $\Pi$ be a regular algebraic cuspidal automorphic representation of ${\rm GL}_n(\A_F).$ Such a representation contributes to the cuspidal cohomology of 
$G_n := {\rm Res}_{F/\Q}(\GL_n / F)$ with coefficients in a sheaf $\tM_\mu$ attached to an algebraic irreducible representation
$\M_\mu$ of $G_n$ with highest weight $\mu$. This information will be denoted as $\Pi \in \Coh(G_n, \mu)$. Similarly, let $\Sigma$ be a regular algebraic cuspidal automorphic representation of ${\rm GL}_{n-1}(\A_F),$ 
and let $\Sigma \in \Coh(G_{n-1}, \lambda).$  
Consider the Rankin--Selberg $L$-function $L(s, \Pi \times \Sigma)$ attached to such a pair of cohomological representations $(\Pi,\Sigma)$.  Algebraicity results for the critical values of 
$L(s, \Pi \times \Sigma)$ are proved under a compatibility condition on the weights $\mu$ and $\lambda.$

\smallskip

Take a representation $\Pi \in \Coh(G_n, \mu^{\sf v})$ as above; henceforth, as in \cite{raghuram-imrn}, working with the dual weight $\mu^{\sf v}$ is only for  convenience.  Let $\Pi = \Pi_\infty \otimes \Pi_f$ be the usual decomposition of $\Pi$ into its archimedean part $\Pi_\infty$ and 
its finite part $\Pi_f$. The rationality field of $\Pi$ is denoted $\Q(\Pi)$; it is a number field.
For a given weight $\mu$, the representation $\M_\mu^{\sf v}$ is defined over a number field $\Q(\mu)$, and by Clozel~\cite{clozel}, it is known that cuspidal cohomology has a $\Q(\mu)$-structure;  
hence the realization of $\Pi_f$ as a Hecke-summand in cuspidal cohomology in lowest possible degree has a  
$\Q(\Pi)$-structure. The choice of lowest possible degree--as will be explained below--is absolutely crucial for this paper. On the other hand, the Whittaker model $\Whit(\Pi_f)$ of the finite part of the representation admits a $\Q(\Pi)$-structure. By comparing these two $\Q(\Pi)$-structures, 
in a previous article with Shahidi \cite{raghuram-shahidi-imrn}, certain 
periods $p^{\epsilon}(\Pi) \in \C^\times$ were defined and studied; here 
$\epsilon = (\epsilon_v)_{v \in S_r}$ is a collection of signs indexed by the set $S_r$ of real places of $F$. These signs can be arbitrary if $n$ is even, and are canonically determined by $\Pi_\infty$ if $n$ is odd; if the number of real places is $r_1$, then there are $2^{r_1}$ such periods of $\Pi$ if $n$ is even and only one period if $n$ is odd. 
For any $\sigma \in \autc$, one knows that ${}^\sigma\Pi \in \Coh(G_n, {}^\sigma\!\mu^{\sf v})$; indeed, one defines  
periods $p^{\epsilon}({}^\sigma\Pi)$ simultaneously for all the conjugates of $\Pi.$ The collection of periods 
$\{p^{\epsilon}({}^\sigma\Pi) : \sigma \in \autc\}$ is well-defined as an element of 
$(\Q(\Pi) \otimes \C)^\times/\Q(\Pi)^\times.$ When $F$ is totally real or a totally imaginary quadratic extension of a totally real field, then the constituents of the representation at infinity of 
${}^\sigma\Pi$ are, up to signs,  a permutation of the constituents of $\Pi_\infty$ (see \cite[Prop.\,3.2]{gan-raghuram}); however, in the case of a general number field this poses some additional difficulties involving a careful analysis at infinity. This issue is related to the set of possible weights that can support cuspidal cohomology for $\GL_n/F$. We call them strongly pure weights; see  \ref{sec:purity}. Henceforth, let $\mu \in X^+_{00}(T_n)$ stand for a dominant integral strongly pure weight and consider 
$\Pi \in \Coh(G_n, \mu^{\sf v}).$ The first main theorem of this article is the following:

\begin{thm}
\label{thm:main}
Let $\mu \in X^+_{00}(T_n)$ and $\Pi \in {\rm Coh}(G_n,\mu^{\sf v})$ for $n \geq 2.$ 
Similarly, let $\lambda \in X^+_{00}(T_{n-1})$ and $\Sigma \in {\rm Coh}(G_{n-1}, \lambda^{\sf v})$. 
Assume that there is an integer $j$ such that 
$\M_{\lambda + j} = \M_{\lambda} \otimes {\rm det}^j$ appears in the restriction to $G_{n-1}$ of $\M_{\mu^{\sf v}}$, i.e., 
\begin{equation}
\label{eqn:comp}
\left\{ j \in \Z\ : \ {\rm Hom}_{G_{n-1}}(\M_{\mu^{\sf v}}, \M_{\lambda + j}) \neq (0) \right\} \mbox{is a non-empty set}. 
\end{equation}
Let $s=\tfrac12 + m \in \tfrac12 + \Z$ be critical for $L_f(s, \Pi \times \Sigma)$ which is the finite part of the 
Rankin--Selberg $L$-function attached to the pair $(\Pi,\Sigma)$. Then, there exist signs 
$\epsilon = (\epsilon_v)_{v \in S_r}$ and $\eta = (\eta_v)_{v \in S_r}$ with $\eta_v = (-1)^n\epsilon_v$ attached to the 
data $(\Pi_\infty,\Sigma_\infty)$ as in \ref{sec:signs}, such that: 

\begin{enumerate}

\item if $n$ is even, then there exists $p^{\epsilon,\eta}_{\infty}(\mu+m,\lambda) \in \C^\times$, 
such that for any $\sigma \in {\rm Aut}({\mathbb C})$ we have 

$$
\sigma\left(\frac{L_f(\tfrac12 + m,\Pi \times \Sigma)}
{p^{\epsilon_m\epsilon}(\Pi)  \,  p^{\eta}(\Sigma)  \,  
\G(\omega_{\Sigma_f})  \, p_\infty^{\epsilon,\eta}(\mu+m,\lambda)}\right)
\  = \  
\frac{L_f(\tfrac12 + m, {}^{\sigma}\Pi \times {}^{\sigma}\Sigma)}
{p^{\epsilon_m{} \epsilon}({}^{\sigma}\Pi) \, p^{\eta}({}^{\sigma}\Sigma) \, 
\G(\omega_{{}^{\sigma}\Sigma_f}) \,
p_{\infty}^{\epsilon, \eta}({}^{\sigma}\!\mu+m, {}^{\sigma}\!\lambda)}, 
$$

\noindent
where $\epsilon_m = (-1)^m$ and $\G(\omega_{\Sigma_f})$ is the Gauss sum attached to the central character of $\Sigma$. In particular, we have   
$$
L_f(\tfrac12 + m,\Pi \times \Sigma)\  
\sim_{{\mathbb Q}(\Pi,\Sigma)} \ 
p^{\epsilon_m \epsilon}(\Pi)\, p^{\eta}(\Sigma)\, \G(\omega_{\Sigma_f})\, p_{\infty}^{\epsilon,\eta}(\mu+m,\lambda), 
$$
where, by $\sim_{{\mathbb Q}(\Pi,\Sigma)}$, we mean up to an element of the number field $\Q(\Pi,\Sigma)$ which is the
compositum of the rationality fields ${\mathbb Q}(\Pi)$ and ${\mathbb Q}(\Sigma)$ of $\Pi$ and $\Sigma$ respectively.

\medskip

\item If $n$ is odd, then there exists $p^{\epsilon,\eta}_{\infty}(\mu,\lambda+m) \in \C^\times$, 
such that  for any $\sigma \in {\rm Aut}({\mathbb C})$ we have

$$
\sigma\left(\frac{L_f(\tfrac12 + m,\Pi \times \Sigma)}
{p^{\epsilon}(\Pi)  \,  p^{\epsilon_m\eta}(\Sigma)  \,  
\G(\omega_{\Sigma_f})  \, p_\infty^{\epsilon,\eta}(\mu,\lambda+m)}\right)
\  = \  
\frac{L_f(\tfrac12 + m, {}^{\sigma}\Pi \times {}^{\sigma}\Sigma)}
{p^{\epsilon}({}^{\sigma}\Pi) \, 
p^{\epsilon_m \eta}({}^{\sigma}\Sigma) \, 
\G(\omega_{{}^{\sigma}\Sigma_f}) \,
p_{\infty}^{\epsilon, \eta}({}^{\sigma}\!\mu, {}^{\sigma}\!\lambda + m)}. 
$$

\noindent
In particular, 
$$
L_f(\tfrac12 + m,\Pi \times \Sigma)\  
\sim_{{\mathbb Q}(\Pi,\Sigma)} \ 
p^{\epsilon}(\Pi)\, p^{\epsilon_m  \eta}(\Sigma)\, \G(\omega_{\Sigma_f})\, p_{\infty}(\mu,\lambda+m). 
$$
\end{enumerate}
\end{thm}

\smallskip

Various special cases have been known: 
For $n=2$ and $F$ totally real, which is the classical Hilbert modular situation, the above theorem is due to 
Shiumra~\cite[Thm.\,4.3]{shimura-duke}, and for a general $F$, it is due to Hida~\cite[Thm.\,I (i)]{hida-duke}. 
For $F = \Q$ and $n \geq 2$, some variation of the above theorem has appeared in Kazhdan--Mazur--Schmidt \cite{kms}, Mahnkopf~\cite{mahnkopf-jussieu}, Kasten--Schmidt \cite{kasten-schmidt} and  
Raghuram~\cite{raghuram-imrn}.  Recently, Grobner--Harris \cite{grobner-harris} proved a similar result when $F$ is an imaginary quadratic field.  

\smallskip

The basic idea of the proof of Theorem~\ref{thm:main} involves interpreting the $L$-value as a Rankin--Selberg integral involving two carefully chosen cusp forms; see Prop.\,\ref{prop:rankin-selberg}. This integral is then interpreted as a map in cohomology stemming from Poincar\'e duality; see the diagram in \ref{sec:main-idea} and then see the main identity in 
Thm.\,\ref{thm:main-identity}. The reader is also referred to 
the introduction of \cite{raghuram-imrn} where this idea is explained in greater detail. 

\smallskip

We now address an important additional ingredient needed for the above generalization from $F=\Q$ to any number field. Kasten and Schmidt observed in \cite[Thm.\,2.3]{kasten-schmidt} that when $F=\Q$, if the weights $\mu$ and 
$\lambda$ satisfy (\ref{eqn:comp}), then for any $m \in \Z$ the point 
$\tfrac12 + m$ is critical for $L(s, \Pi \times \Sigma)$ 
if and only if $\M_\mu^{\sf v}$ contains $\M_{\lambda + m}.$ This is a purely local statement involving cohomological representations of $\GL_n(\R)$ and standard branching laws. What seems initially surprising, but rather natural after the fact, is that an identical statement (see Thm.\,\ref{thm:critical-compatible}) holds true for 
cohomological representations of $\GL_n(\C)$  although the proof turns out to be combinatorially more challenging than for $\GL_n(\R)$; the main steps in the proof of 
Thm.\,\ref{thm:critical-compatible} are contained in 
Lem.\,\ref{lem:1-2-crit}, Cor.\,\ref{cor:>implies-crit} and Prop.\,\ref{prop:crit-glnc}. (The reader should also look at 
Grobner--Harris \cite[Lem.\,4.7]{grobner-harris}.)  This particular observation, that we can deal with both real and complex cases on the same footing, was at the genesis of this article.

\medskip

\subsection{\bf Relation with motivic periods and motivic $L$-functions}

Let $M$ be a pure motive over $\Q$ with coefficients in a number field $\Q(M).$ Suppose $M$ is critical, then 
a celebrated conjecture of Deligne~\cite[Conj.\,2.8]{deligne} relates the critical values of its $L$-function $L(s,M)$ to certain periods that arise out of a comparison of the Betti and de~Rham realizations of the motive. One expects a cohomological cuspidal automorphic representation $\Pi$ to correspond to a motive $M(\Pi)$; one of the properties of this correspondence is that the standard $L$-function $L(s, \Pi)$ is the motivic $L$-function $L(s, M(\Pi))$ up to a shift in the $s$-variable; see Clozel \cite[Sect.\,4]{clozel}. With the current state of technology, it seems impossible to compare our periods $p^\epsilon(\Pi)$ with Deligne's periods $c^{\pm}(M(\Pi))$. However, one may ask if the periods of a tensor product motive decompose in a way suggested by our theorem. Bhagwat~\cite{bhagwat}  
has given a description of Deligne's periods $c^\pm$ for the tensor product $M \otimes M'$, where $M$ and $M'$ are two pure motives over $\Q$ all of whose nonzero Hodge numbers are one, in terms of the periods $c^\pm$ {\it and} some other finer invariants attached to $M$ and $M'$ by Yoshida~\cite{Yo}.  These period relations suggest that 
$p^\epsilon(\Pi)$ and $c^\pm(M(\Pi))$ are two very different kind of periods, and a direct comparison with Deligne's conjecture seems out of reach; however, see Grobner--Harris~\cite{grobner-harris}. 
Be that as it may, one can still claim that Thm.\,\ref{thm:main} is compatible with Deligne's conjecture by considering 
the behavior of $L$-values under twisting by characters. Blasius \cite{blasius} and Panchishkin \cite{panchishkin} have independently studied the behavior of Deligne's periods upon twisting the motive by a Dirichlet character (more generally by Artin motives). Using Deligne's conjecture, they predict the behavior of critical values of motivic $L$-functions upon twisting by Dirichlet characters. For a critical motive over $\Q$, assumed to be simple and of rank $2r$, this prediction looks like $L(m, M \otimes \chi_f)  \sim_{\Q(M,\chi)}  L(m,M) \G(\chi_f)^r.$
Applying this to the tensor product motive $M(\Pi) \otimes M(\Sigma)$, which has rank $n(n-1),$ we have the following compatibility with Deligne's conjecture:

\begin{cor}
\label{cor:blasius}
Let $\Pi \in {\rm Coh}(G_n,\mu^{\sf v})$ and $\Sigma \in {\rm Coh}(G_{n-1}, \lambda^{\sf v}).$ Assume that the compatibility condition (\ref{eqn:comp}) holds. For any critical point $s=\tfrac12 + m$ of $L_f(s, \Pi \times \Sigma)$ and for any character $\chi :  F^\times\backslash \A_F^\times \to \C^\times$ of finite order, and any $\sigma \in {\rm Aut}({\mathbb C})$ we have
$$
\sigma\left(\frac{L_f(\tfrac12 + m,\Pi \times \Sigma \otimes \chi)}
{\G(\chi)^{n(n-1)/2}L_f(\tfrac12 + m,\Pi \times \Sigma)}\right)  \ = \ 
\frac{L_f(\tfrac12 + m, {}^\sigma\Pi \times {}^\sigma\Sigma \otimes {}^\sigma\!\chi)}
{\G({}^\sigma\!\chi)^{n(n-1)/2}L_f(\tfrac12 + m, {}^\sigma\Pi \times {}^\sigma\Sigma)}. 
$$
\end{cor}

In the above corollary, the compatibility condition (\ref{eqn:comp}) has an important bearing, because, if not, then 
already for $\GL_2 \times \GL_1$ over an imaginary quadratic field, twisting by a finite order character can change the critical $L$-value by a quantity not lying in any field that is cyclotomic over the coefficient fields involved. (See, for example, Hida~\cite[Thm.\,I, (ii)]{hida-duke} when the set $A$ therein is not empty.) This phenomenon is further elaborated upon in \ref{sec:subtlety}.

\medskip

It should be noted that Theorem~\ref{thm:main} also gives some evidence toward a conjecture due to Gross \cite[Conjecture 2.7(ii)]{deligne} which says that the order of vanishing of a motivic $L$-function at a critical point is independent of which conjugate of the motive we are looking at, i.e., if $M$ is critical, then ${\rm ord}_{s=0} L(s, \sigma, M)$ is independent of the embedding $\sigma : \Q(M) \to \C$. We are unable to say anything about the order of vanishing, however, it follows from our theorem that the property of vanishing is indeed independent of which particular conjugate of the representation we consider.

\begin{cor}
\label{cor:gross}
Let $\Pi \in {\rm Coh}(G_n,\mu^{\sf v})$ and $\Sigma \in {\rm Coh}(G_{n-1}, \lambda^{\sf v}).$ Assume that the compatibility condition (\ref{eqn:comp}) holds. For any critical point $\tfrac12 + m$ of $L_f(s, \Pi \times \Sigma)$ and 
$\sigma\in {\rm Aut}(\C)$ we have 
$$
L(\tfrac 12+m, \Pi \times \Sigma) = 0  \ \iff \  L(\tfrac 12+m,{}^\sigma\Pi \times {}^\sigma\Sigma) = 0.
$$
\end{cor}

\medskip

As in my paper with Wee Teck Gan \cite{gan-raghuram}, Thm.\,\ref{thm:main} implies an arithmeticity result for 
$\GL_{n-1}$-periods of automorphic representations of $\GL_n \times \GL_{n-1}.$ See \cite[Thm.\,7.1]{gan-raghuram}. 
As explained in the introduction of that paper, this consequence is analogous to Gross's conjecture but for automorphic periods. We state this consequence as the following 

\begin{cor}\label{cor:arithmeticity}
Let $\Pi \in {\rm Coh}(G_n,\mu^{\sf v})$ and $\Sigma \in {\rm Coh}(G_{n-1}, \lambda^{\sf v})$ be as in Thm.\,\ref{thm:main}. 
Suppose $\mu$ and $\lambda$ satisfy (\ref{eqn:comp}). Consider the representation $\Pi \otimes \Sigma$ of 
$(G_n \times G_{n-1})(\A)$. Let $\Delta G_{n-1}$ be the image of the diagonal embedding of $G_{n-1}$ in 
$G_n \times G_{n-1}.$ Then 
$$
\Pi \otimes \Sigma \ {\rm is} \ \Delta G_{n-1} \ {\rm distinguished} \implies 
{}^\sigma\Pi \otimes {}^\sigma\Sigma \ {\rm is}\  \Delta G_{n-1} \ {\rm distinguished} \ \forall \sigma \in \autc.
$$
\end{cor}

\medskip

\subsection{\bf Symmetric power $L$-functions for $\GL_2$}  

Let $\pi$ be a cohomological cuspidal automorphic representation of $\GL_2$ over $F$. For any $r \geq 1$, let 
${\rm Sym}^r(\pi)$ denote the Langlands transfer of $\pi$ corresponding to the homomorphism 
$\Sym^r : \GL_2(\C) \to \GL_{r+1}(\C)$ of $L$-groups; ${\rm Sym}^r(\pi)$ 
is conjecturally an isobaric automorphic representation of $\GL_{r+1}$ over $F$. Such a transfer is known to exist for a general $\pi$ for $r \leq 4$ by the works of  Gelbart--Jacquet, Kim, and Kim--Shahidi, and for all $r$ if 
$\pi$ is dihedral.   
We define the $r$-th symmetric power $L$-function as the standard $L$-function of the $r$-th symmetric transfer of $\pi$, i.e., $L(s, \Sym^r, \pi) = L(s, \Sym^r(\pi))$; for more details and references, see 
\cite[Sect.\,3.1]{raghuram-shahidi-aim}, \cite[Sect.\,5.1.1]{raghuram-imrn} and Sect.\,\ref{sec:symm-powers} below. 
(For some new instances of symmetric power functoriality, such as for $r=6$ and $r=8$, albeit with some restrictions on $F$, see Clozel--Thorne~\cite{clozel-thorne}.)

Given a weight $\mu \in X^*_{00}(T_2)$ and  an integer $r \geq 1$, 
we define its symmetric power transfer $\Sym^r(\mu) \in X^*_{00}(T_{r+1})$ 
in Sect.\,\ref{sec:defn-symr-mu}. In Thm.\,\ref{thm:sym-coh} we prove 
$$
\pi \in \Coh(G_2, \mu^{\sf v}) \ \mbox{and if $\Sym^r(\pi)$ is cuspidal} 
\ \implies \ 
\Sym^r(\pi) \in \Coh(G_{r+1}, \Sym^r(\mu)^{\sf v}). 
$$
This says that symmetric power transfers preserves the property of being cohomological,  
generalizing  our previous result \cite[Thm.\,5.5]{raghuram-shahidi-aim} and lending further evidence to the discussion in \cite[Sect.\,5.2]{raghuram-shahidi-aim} relating functoriality and the property of being cohomological. 
Similarly, we can define a weight ${\rm det}(\mu)$ so that $\omega_\pi \in \Coh(G_1, {\rm det}(\mu)^{\sf v})$; 
see Sect.\,\ref{sec:explicate-gl2-gl1}.

In \cite{raghuram-imrn} Rankin--Selberg theory for $\GL_{n} \times \GL_{n-1}$ over $\Q$ was used to 
to get results on special values of odd symmetric power $L$-functions attached to modular forms. We completely work out this idea in the case of symmetric cube $L$-functions for cusp forms for $\GL_2$ over any number field. The starting point is the factorization: 
$$
L_f(s, \Sym^2(\pi) \times \pi \otimes \xi )
 \ =  \ 
L_f(s, \Sym^3(\pi) \otimes \xi)  \cdot L_f(s, \pi \otimes\omega_{\pi} \xi), 
$$
where $\xi$ is a Hecke character of $F$ of finite order. 
We know the special values of the left hand side (resp., the second factor on the right hand side) 
by the $\GL_3 \times \GL_2$ case (resp., $\GL_2 \times \GL_1$ case) 
of Thm.\,\ref{thm:main}; we deduce a result for the special values of symmetric cube $L$-functions. Take a 
half-integer $\tfrac12+m$ which is critical for $L_f(s, \Sym^3(\pi) \otimes \xi)$ and 
$L(s, \pi \otimes\omega_{\pi}\xi)$, then it is also critical for 
$L_f(s, \Sym^2(\pi) \times \pi \otimes \xi)$; see Sect.\,\ref{sec:critical-sym3} for a description of the set of critical points. 
We write 
$$
L_f(\tfrac12+m, \Sym^3(\pi) \otimes \xi)  \ = \  
\frac{L_f(\tfrac12+m, \Sym^2(\pi) \times \pi \otimes \xi )}{L_f(\tfrac12+m, \pi \otimes\omega_{\pi}\xi)}, 
$$
provided $L(\tfrac12+m, \pi \otimes\omega_{\pi}\xi) \neq 0$. 
For all critical points, except possibly for the central critical point, the denominator is indeed nonzero. 
If ${\sf w} = {\sf w}(\mu)$ is the purity weight of $\mu$, then $\pi = \pi^\circ \otimes |\ |^{{\sf w}/2}$ for a unitary cuspidal automorphic representation $\pi^\circ$, and   
$L(\tfrac12+m, \pi \otimes\omega_{\pi}\xi) =  L(\tfrac12+m + 3{\sf w}/2, \pi^\circ \otimes\omega_{\pi^\circ}\xi).$ Hence, 
$L(\tfrac12+m, \pi \otimes\omega_{\pi}\xi) = 0$ can happen only when $m = -3{\sf w}/2$. (This is the center of symmetry in the unitary case; note the necessary condition that ${\sf w}$ should be even.) 
In this particular case, we use the main theorem of Rohrlich~\cite{rohrlich} and introduce another twisting character making the $L$-value nonzero, i.e., after relabeling if necessary, we may take $\xi$ to be such that 
$L((1- 3{\sf w})/2, \pi \otimes\omega_{\pi}\xi) \neq0.$ 
Putting these together gives us the following

\begin{thm}\label{thm:sym3}
Let $\pi \in \Coh(G_2, \mu^{\sf v})$ for $\mu \in X^*_{00}(T_2).$ Assume that the pairs of weights 
$(\Sym^2(\mu), \mu)$ and $(\mu, {\rm det}(\mu))$ satisfy (\ref{eqn:comp}). 
Let $\xi : F^\times \backslash \A_F^\times \to \C^\times$ 
be a character of finite order; if ${\sf w}$ is even, then take $\xi$ such that 
$L((1- 3{\sf w})/2, \pi \otimes\omega_{\pi} \xi) \neq 0.$ Suppose $\tfrac12+m$ is critical for 
$L_f(s, \Sym^3(\pi) \otimes \xi)$ and $L(s, \pi \otimes\omega_{\pi} \xi).$ Then we have: 
$$
L_f(\tfrac12 + m, \, \Sym^3(\pi) \otimes \xi) 
\ \sim \ 
p^{\epsilon_+}(\Sym^2(\pi)) \ 
\frac{p^{-\epsilon_m \epsilon_\xi}(\pi)}{p^{\epsilon_m \epsilon_\xi}(\pi)} \ 
\G(\xi)^2 \ 
\frac{p_\infty^{\epsilon_+, \, \epsilon_-}(\Sym^2(\mu), \mu+m)}
{p_\infty^{\epsilon_\xi, \epsilon_\xi}(\mu + m, {\rm det}(\mu))}, 
$$
where, by $\sim$, we mean up to an element of $\Q(\pi,\xi)$; $\epsilon_+$ is the sign which is $+$ everywhere; $\epsilon_- = -\epsilon_+$, and $\epsilon_\xi$ is the signature of $\xi$ as in Sect.\,\ref{sec:algebraic-hecke}. 
Furthermore, the ratio of the left hand side by the right hand 
side is equivariant under $\autc.$ 
\end{thm}
The proof of the above theorem is given in Sect.\,\ref{sec:proof-thm-sym}. Moreover, 
exactly as in \cite[Sect.\,5]{raghuram-imrn}, we can get analogous results  for higher odd symmetric power $L$-functions. For fifth and seventh symmetric power $L$-functions we would get partial results, and assuming Langlands's functoriality, we would get conditional results for all odd symmetric power $L$-functions. We omit the details as any interested reader can proceed as in Sect.\,\ref{sec:proof-thm-sym}.

\bigskip

\noindent
{\Small
{\it Note to the reader:} Although we have tried to make this article self-contained, anyone who wishes to verify details, will need to keep copies of  \cite{raghuram-imrn}, \cite{raghuram-shahidi-aim} and 
\cite{raghuram-shahidi-imrn} by his/her side.}

\smallskip

\noindent
{\Small
{\it Acknowledgements:} It is a pleasure to thank Michael Harris and Fabian Januszewski for their comments on a preliminary version of this article. I thank the referee for pointing out a piquant phenomenon concerning critical points 
that manifests itself when the base field has a complex place; in an earlier version of this manuscript this subtle issue was glossed over, but it is now addressed in \ref{sec:subtlety}.  
}

\bigskip

\section{\bf Arithmetic properties of Rankin--Selberg $L$-functions}

\subsection{\bf Some notation and preliminaries}

\subsubsection{\bf The base field} 
Let $F$ be a number field of degree $d=[F:\Q]$ with ring of integers $\O$. For any place $v$ we write $F_v$ for the topological completion of $F$ at $v$. Let $S_\infty$ be the set of archimedean places of $F$. Let $S_\infty := S_r \cup S_c$, where $S_r$ (resp., $S_c$) is the set of real (resp., complex) places. Let $\cE_F = {\rm Hom}(F,\C)$ be the set of all embeddings of $F$ as a field into $\C$. There is a canonical surjective map $\cE_F \to S_\infty$, which is a bijection on the real embeddings and real places, and identifies a pair of complex conjugate embeddings $\{\iota_v, \bar{\iota}_v\}$ with the  complex place $v$. For each $v \in S_r$, we fix an isomorphism 
$F_v \cong \R$ which is canonical. Similarly for $v \in S_c$, we fix $F_v \cong \C$ given by (say) $\iota_v$; this choice is not canonical. 
Let $r_1 = |S_r|$ and $r_2 = |S_c|$; hence $d = r_1 + 2r_2.$ 
If $v \notin S_{\infty}$, we let $\O_v$ be the ring of integers of $F_v$, and $\wp_v$ it's  unique maximal ideal. 
Moreover, $\A_F$ denotes the ring of ad\`eles of $F$ and $\A_{F,f}$ its finite part. The group of id\`eles of $F$ will be denoted $\A_F^\times$ and similarly, $\A_{F,f}^\times$ is the group of finite id\`eles. We will drop the subscript $F$ when talking about 
$\Q$. Hence, $\A$ is $\A_\Q$, etc. 
We use the local and global normalized absolute values and denote each of them by $|\cdot|$. 
Further, $\mathfrak{D}_F$ stands for the absolute different of $F$, i.e., $\mathfrak{D}_F^{-1} = \{x \in F : Tr_{F/\Q}(x \O) \subset \Z\}$.

\subsubsection{\bf The groups $G_n \supset B_n \supset T_n \supset Z_n \supset S_n$} 
The algebraic group ${\rm GL}_n/F$ will be denoted as $\underline{G}_n$, and we put 
$G_n = R_{F/\Q}(\underline{G}_n)$. An $F$-group will be denoted by 
an underline and the corresponding $\Q$-group via Weil restriction of scalars will be denoted without the underline; hence for any $\Q$-algebra $A$ the 
group of $A$-points of $G_n$ is $G_n(A) = \underline{G}_n(A \otimes_\Q F)$. 
Let $\underline{B}_n = \underline{T}_n \underline{U}_n$ stand for the standard Borel subgroup of $\underline{G}_n$ of all upper triangular matrices, where $\underline{U}_n$ is the unipotent radical of $\underline{B}_n$, and $\underline{T}_n$ the diagonal torus. The center of $\underline{G}_n$ will be denoted by $\underline{Z}_n$. These groups define the corresponding $\Q$-groups $G_n \supset B_n = T_nU_n \supset Z_n$. Observe that $Z_n$ is not $\Q$-split, and we let $S_n$ be the maximal $\Q$-split torus in $Z_n$; we have 
$S_n \cong {\mathbb G}_m$ over $\Q.$

\subsubsection{\bf The groups at infinity}
\label{sec:group-infinity}
Note that 
$$
G_{n,\infty} := G_n(\R) = \underline{G}_n(F \otimes_\Q \R) = \prod_{v \in S_{\infty}} \GL_n(F_v) \cong \prod_{v \in S_r} \GL_n(\R) \times \prod_{v \in S_c} \GL_n(\C).
$$ 
Let $\K$ denote either $\R$ or $\C$. We have $Z_n(\R) =  \prod_{v \in S_r} \R^\times \times \prod_{v \in S_c} \C^\times,$ where each copy of $\K^\times$ consists of nonzero scalar matrices in the 
corresponding copy of $\GL_n(\K).$  
The subgroup $S_n(\R)$ of $Z_n(\R)$ consists of $\R^\times$ diagonally embedded in $\prod_{v \in S_r} \R^\times \times \prod_{v \in S_c} \C^\times.$
Let $C_{n,\infty} =  \prod_{v \in S_r} \Orth(n) \times \prod_{v \in S_c} \Unit(n)$ be the maximal compact subgroup of $G_n(\R)$, and let 
\begin{eqnarray*}
K_{n,\infty} & =  & S_n({\mathbb R}) C_{n,\infty} \\
& \cong & \R^\times  \left(\prod_{v \in S_r} \Orth(n) \times \prod_{v \in S_c} \Unit(n)\right) \\
& \cong &  \R_+^\times  \left(\prod_{v \in S_r} \Orth(n) \times \prod_{v \in S_c} \Unit(n)\right) \ \cong \ S_n({\mathbb R})^0C_{n,\infty}.
\end{eqnarray*}
Let $K_{n,\infty}^0$ be the topological connected component of $K_{n,\infty}$. Hence 
$$
K_{n,\infty}^0 =  S_n({\mathbb R})^0 C_{n,\infty}^0 \ \cong \  \R_+^\times \left(\prod_{v \in S_r} {\rm SO}(n) \times \prod_{v \in S_c} \Unit(n)\right). 
$$
 For any topological group $\mathfrak{G}$ we will let $\pi_0(\mathfrak{G}) := \mathfrak{G}/\mathfrak{G}^0$ stand for the group of connected components. 
Inclusion of connected components induces the equality $\pi_0(K_{n,\infty}) = \pi_0(G_n(\R))$. Observe also that 
$\pi_0(K_{n,\infty}) \cong  \prod_{v \in S_r} \{\pm 1 \} \cong: \prod_{v \in S_r} \{\pm \}$. The matrix 
$\delta_n = {\rm diag}(-1,1,\dots,1)$ represents the nontrivial element in $\Orth(n)/\SO(n)$, and if $n$ is odd, the scalar matrix $-1_n$ 
also represents this nontrivial element. We identify $\pi_0(G_n(\R))$ inside $G_n(\R)$ via the $\delta_n$'s.  
The character group of $\pi_0(K_{n,\infty})$ is denoted $\widehat{\pi_0(K_{n,\infty})}$.

\subsubsection{\bf Lie algebras}
\label{sec:lie-algebras}
The general notational principle we follow is that for a real Lie group $G$, we denote its Lie algebra by $\g^0$ and the complexified Lie algebra by $\g,$ i.e., 
$\g = \g^0 \otimes_\R \C.$ Hence, if $G$ is the Lie group $\GL_n(\R)$ then $\g^0 = \gl_n(\R)$ and $\g = \gl_n(\C)$. 
On the other hand, if $G$ stands for the real Lie group $\GL_n(\C)$ then $\g^0 = \gl_n(\C)$ as a $\R$-Lie algebra,  
and $\g = \gl_n(\C) \otimes_\R \C.$  With this notational scheme, we have $\g_n$, 
$\mathfrak{b}_n$, $\mathfrak{t}_n$ and $\k_n$ denoting the complexifed Lie algebras of $G_n(\R)$, 
$B_n(\R)$, $T_n(\R)$ and $K_{n,\infty}^0$ respectively. For example, 
$\g_n = \prod_{v \in S_r} \gl_n(\C)  \times \prod_{v \in S_c}( \gl_n(\C) \otimes_\R \C).$

\subsubsection{\bf Finite-dimensional representations} 
\label{sec:fin-dim-repn}
Consider $T_n(\R) = \underline{T}_n(F \otimes {\mathbb R}) \cong \prod_{v \in S_\infty} \underline{T}_n(F_v)$. 
We let $X^*(T_n)$ stand for the group of all algebraic characters of $T_n$, and let 
$X^+(T_n)$ stand for all those characters in $X^*(T_n)$ 
which are dominant with respect to $B_n$. A weight $\mu \in 
X^+(T_n)$ may be described as follows: 
$\mu \ = \ (\mu^\iota)_{\iota \in \cE_F},$ where 
\begin{itemize}
\item 
For $v \in S_r$, we have $\mu^v = (\mu^v_1,\dots, \mu^v_n)$, $\mu^v_i \in {\mathbb Z}$, 
$\mu^v_1 \geq \cdots \geq \mu^v_n$, 
and the character $\mu^v$ sends $t = {\rm diag}(t_1,\dots,t_n) \in \underline{T}_n(F_v)$ to 
$\prod_i t_i^{\mu^v_i}.$ 
\item If $v \in S_c$ then $\mu^v$ is the pair $(\mu^{\iota_v}, \mu^{\bar{\iota}_v})$, with $\mu^{\iota_v} = (\mu^{\iota_v}_1,\dots, \mu^{\iota_v}_n)$, 
 $\mu^{\iota_v}_i \in {\mathbb Z}$, $\mu^{\iota_v}_1 \geq \cdots \geq \mu^{\iota_v}_n$; likewise
$\mu^{\bar{\iota}_v} = (\mu^{\bar{\iota}_v}_1,\dots, \mu^{\bar{\iota}_v}_n)$ and  
$\mu^{\bar{\iota}_v}_1 \geq \cdots \geq \mu^{\bar{\iota}_v}_n;$  
the character $\mu^v$ is given by sending $t = {\rm diag}(z_1,\dots,z_n) \in \underline{T}_n(F_v)$ to 
$\prod_{i=1}^n z_i^{\mu^{\bar{\iota}_v}_i} \bar{z}_i^{\mu^{\bar{\iota}_v}_i},$ where $\bar{z}_i$ is the complex  conjugate of $z_i$.  
\end{itemize}
We also write $\mu = (\mu^v)_{v \in S_{\infty}},$ where $\mu^v = (\mu^{\iota_v}, \mu^{\bar{\iota}_v})$ for $v \in S_c.$ 

\smallskip

For $\mu \in X^+(T_n)$, we define a finite-dimensional complex 
representation $(\rho_{\mu}, \M_{\mu, \C})$ of $G_n(\R)$ as follows.
For $v \in S_r$, let $(\rho_{\mu^v}, \M_{\mu^v, \C})$ 
be the irreducible complex representation of $G_n(F_v) = \GL_n({\mathbb R})$ with 
highest weight $\mu_v$.
For $v \in S_c$, let $(\rho_{\mu^v}, \M_{\mu^v, \C})$ be the complex representation of the real algebraic group 
$G(F_v) = \GL_n({\mathbb C})$ defined as 
$\rho_{\mu^v}(g) = \rho_{\mu^{\iota_v}}(g) \otimes \rho_{\mu^{\bar{\iota}_v}}(\overline{g});$ here
$\rho_{\mu^{\iota_v}}$ (resp., $\rho_{\mu^{\bar{\iota}_v}}$)
is the irreducible representation of the complex group $\GL_n(\C)$ with highest weight 
$\mu^{\iota_v}$ (resp., $\mu^{\bar{\iota}_v}$). 
Now we let $\rho_{\mu} = \otimes_{v \in S_{\infty}} \rho_{\mu^v}.$

\smallskip

Note that $\autc$ acts on $X^*(T_n)$ as follows: if $\sigma \in \autc$ and $\mu \in X^*(T_n)$ then 
${}^\sigma\!\mu \in X^*(T_n)$ is defined as: 
${}^\sigma\!\mu \ = \ ({}^\sigma\!\mu^\iota)_{\iota \in \cE_F}$ where 
${}^\sigma\!\mu^\iota := \mu^{\sigma^{-1} \circ \iota}.$ Define the rationality field $\Q(\mu)$ as the fixed field in $\C$ under all those automorphisms $\sigma$ which fix $\mu.$ 
Consider the representation $(\rho_{{}^\sigma\!\mu}, \M_{{}^\sigma\!\mu, \C})$ of $G_n(\R)$ of highest weight 
${}^\sigma\!\mu.$ 
Consider $\C$ as a $(\C,\C)$-bimodule, where the left module structure is via $\sigma$ and the right module structure is the usual multiplication in $\C$; denote this bimodule as ${}_\sigma\C$. Then the canonical map $t : \M \to \M \otimes_{\C} {}_\sigma\C$ defined by $t(w) = w \otimes 1$ is a $\sigma$-linear isomorphism.  Take $\M = \M_{\mu, \C}$ and denote by $({}^\sigma\!\rho_{\mu}, {}^\sigma\!\M_{\mu, \C})$ the representation of $G_n(\R)$ where 
a $g \in G_n(\R)$ acts on 
${}^\sigma\!\M_{\mu, \C} = \M_{\mu, \C} \otimes {}_\sigma\C$ by 
${}^\sigma\!\rho_{\mu}(g) = t \circ \rho_\mu(g) \circ t^{-1}.$ Then ${}^\sigma\!\rho_{\mu} \simeq \rho_{{}^\sigma\!\mu}$ 
as a representation of $G_n(F)$ (see \cite[Lem.\,7.1]{grobner-raghuram-ijnt}) and this representation is defined over 
$\Q(\mu)$ which may be seen exactly as in Waldspurger \cite[Prop.\,I.3]{waldy}.  For any extension $E/\Q(\mu)$ we will let $\M_{\mu,E} = \M_{\mu, \Q(\mu)} \otimes_{\Q(\mu)} E$ on which $G_n(F)$ acts via its action on the first factor.

\subsubsection{\bf Automorphic representations}
Following Borel--Jacquet \cite[\S4.6]{borel-jacquet}, we say an irreducible 
representation of $G_n(\A) = \GL_n(\A_F)$ is automorphic if it is isomorphic to an 
irreducible subquotient of the representation of $G_n(\A)$ on its
space of automorphic forms. We say an automorphic representation is cuspidal 
if it is a subrepresentation of the representation of $G_n(\A)$ on 
the space of cusp forms $\mathcal{A}_{\rm cusp}(G_n(\Q) \backslash G_n(\A)) = \mathcal{A}_{\rm cusp}(\GL_n(F) \backslash \GL_n(\A_F))$. 
The subspace of cusp forms realizing a cuspidal automorphic representation $\pi$ will be denoted $V_{\pi}$. 
For an automorphic representation $\pi$ of $G_n(\A)$, we have
$\pi = \pi_{\infty} \otimes \pi_f$, where $\pi_{\infty}$ is a representation of $G_n({\mathbb R})$,
and $\pi_f = \otimes_{v \notin S_{\infty}} \pi_v$
is a representation of $G_n({\mathbb A}_f)$. The central character of  $\pi$ will be denoted $\omega_{\pi}$.

\subsubsection{\bf Algebraic Hecke characters} 
\label{sec:algebraic-hecke}
(References: Deligne \cite[\S 5]{deligne-sga4.5}, Schappacher \cite[Chapter 0]{schappacher}, Waldspurger \cite[I.5]{waldy} or 
Weil  \cite{weil-a0}.)
A continuous homomorphism $\omega : F^\times\backslash \A_F^\times \to \C^\times$ is called a Hecke character of $F.$ 
An element $\alpha = \sum_{\iota \in \cE_F} a_\iota \iota$, with $a_\iota \in \Z$ is called an {\it infinity type}. 
A Hecke character $\omega$ is called an algebraic Hecke character of infinity type $\alpha$ if 
\begin{itemize}
\item for $v \in S_r$, $\omega_v(x) = x^{a_{\iota_v}}$ for all $x \in \R^\times_+$; 
\item for $v \in S_c$, $\omega_v(z) = z^{a_{\iota_v}} \bar{z}^{a_{\bar{\iota}_v}}$ for all $z \in \C^\times.$
\end{itemize}
Weil  gave the appellation `characters of type (A$_0$)' for such algebraic Hecke characters. 
The existence of an algebraic Hecke character $\omega$ with infinity type $\alpha$ implies the following purity constraint on $\alpha$:
\begin{enumerate}
\item if $S_r$ is not empty, i.e., if $F$ has at least one real place, then the map from $\cE_F \to \Z$ given by $\iota \mapsto a_\iota$ is constant; in this case, let ${\sf w}(\omega) := a_\iota$ for any $\iota$. 
\item if $S_r$ is empty, i.e., if $F$ is a totally imaginary field, then the map from $\cE_F \times {\rm Aut}(\C) \to \Z$ given by 
$(\iota, \sigma) \mapsto a_{\sigma\iota} + a_{\sigma\bar\iota}$ is  constant; 
in this case, let ${\sf w}(\omega) := a_\iota + a_{\bar\iota}$ for any $\iota$
\end{enumerate}
In either case, we call ${\sf w}(\omega)$ the {\it purity weight of $\omega$}. 

\smallskip

Suppose that $F$ has at least one real place, then we define the {\it signature} $\epsilon_\omega$ of an algebraic Hecke character $\omega$ as follows: 
By the purity constraint, the character $\omega^\circ := \omega |\ |^{-{\sf w}(\omega)}$ is a character of finite order.  For $v \in S_r$, define 
$$
\epsilon_{\omega_v} = (-1)^{{\sf w}(\omega)} \omega^\circ_v(-1).
$$
Now put $\epsilon_\omega = (\epsilon_{\omega_v})_{v \in S_r}.$ The signature is an $r_1$-tuple of signs indexed by real embeddings of $F$.

\smallskip

For each finite place $v$, and any smooth character $\chi_v : F_v^\times \to \C^\times$, define the rationality field 
$\Q(\chi_v)$ of $\chi_v$ as the field obtained by adjoining 
the values of $\omega_v$ to $\Q$. For an algebraic Hecke character $\omega$, we define its {\it rationality field} 
$\Q(\omega)$ as the compositum of the fields $\Q(\omega_v)$ for 
all finite places $v$ that are unramified for $\omega.$ It is a standard fact that $\Q(\omega)$ is a number field, and as Weil notes in \cite{weil-a0}, the field $\Q(\omega)$ need not contain the field $F$.

\subsubsection{\bf Additive characters and Gauss sums}
\label{sec:gauss}
We fix an additive character $\psi_{\Q}$ of ${\mathbb Q} \backslash {\mathbb A}$, as in Tate's thesis, namely,
$\psi_{\Q}(x) = e^{2\pi i \lambda(x)}$ with the $\lambda$ as defined in \cite[Sect.\,2.2]{tate-thesis}. 
Next, we define a character $\psi$ of $F\backslash \A_F$ by composing $\psi_{\Q}$ with the
trace map from $F$ to $\Q$: $\psi = \psi_{\Q} \circ Tr_{F/\Q}$. If $\mathfrak{D}_F = \prod_{\wp} \wp^{r_{\wp}}$ with the product running over all prime ideals $\wp$, and $\psi = \otimes_v \psi_v$, then the conductor of 
$\psi_{v}$ at a finite place $v$ is $\wp_v^{-r_{\wp}}$, i.e., $\psi_{v}$ is trivial on $\wp_v^{-r_{\wp_v}}$ and nontrivial 
on $\wp_v^{-r_{\wp_v}-1}.$ For any Hecke character $\chi$ of $F$, we define the Gau\ss ~sum $\G(\chi_f)$ 
of $\chi_f$ exactly as in \cite[Sect.\,2]{raghuram-imrn}; this depends on choice of $\psi$.

\subsection{\bf Rankin--Selberg $L$-functions: analytic aspects}
\label{sec:rankin-selberg}

This subsection is a brief summary of \cite[Sect.\,3.1]{raghuram-imrn}; see references therein for all the assertions made below. 

\subsubsection{\bf Rankin--Selberg zeta integrals for $G_n \times G_{n-1}$}
Let $\Pi$ (resp., $\Sigma$) be a cuspidal automorphic representation of $G_n(\A)$
(resp., $G_{n-1}(\A)$). Let $\phi \in V_{\Pi}$ and $\phi' \in V_{\Sigma}$ be cusp forms. Consider 
$$
I(s, \phi, \phi') = \int_{G_{n-1}({\mathbb Q})\backslash G_{n-1}({\mathbb A})}
\phi(\iota(g))\phi'(g)|{\rm det}(g)|^{s - 1/2}\, dg.
$$
The above integral converges for all 
$s \in {\mathbb C}$. Suppose that $w \in \Whit(\Pi,\psi)$ and $w' \in \Whit(\Sigma, \psi^{-1})$ are global Whittaker functions corresponding to $\phi$ and $\phi'$, respectively. We have 
$$
I(s,\phi,\phi') = \Psi(s,w,w') := \int_{U_{n-1}({\mathbb A})\backslash G_{n-1}({\mathbb A})}
w(\iota(g))w'(g)|{\rm det}(g)|^{s - 1/2}\, dg.
$$
The integral $\Psi(s,w,w')$ converges for ${\rm Re}(s) \gg 0$. 
Let $w = \otimes w_v$ and $w' = \otimes w'_v$, then $\Psi(s,w,w') := \otimes \Psi_v(s,w_v,w'_v)$ for 
${\rm Re}(s) \gg 0$, where the local 
integral $\Psi_v$ is given by a similar formula. Recall that the local integral $\Psi_v(s,w_v,w'_v)$ converges 
for ${\rm Re}(s) \gg 0$ and has a meromorphic continuation to all of ${\mathbb C}$. 
We will choose the local Whittaker functions carefully so that
the integral $I(\tfrac12,\phi,\phi')$ computes the special value $L_f(\tfrac12, \Pi\times\Sigma)$ up to quantities which are 
${\rm Aut}({\mathbb C})$-equivariant.

\subsubsection{\bf Action of ${\rm Aut}({\mathbb C})$ on Whittaker models} 
Any $\sigma \in \autc$ gives an element 
$t_\sigma \in \A_f^\times.$ Given $w \in \Whit(\Pi_f, \psi_f)$, define ${}^\sigma\!w \in \Whit({}^\sigma\Pi_f, \psi_f)$ by 
${}^\sigma\!w(g_f) = \sigma(w(t_{\sigma,n}g_f)),$ $g_f \in G(\A_f),$
where $t_{\sigma,n} = {\rm diag}(t_\sigma^{-(n-1)}, t_\sigma^{-(n-2)}, \dots, t_\sigma^{-1}, 1).$ For more details, see 
\cite[Sect.\,3.1]{raghuram-imrn} and \cite[Sect.\,3.2]{raghuram-shahidi-imrn}.

\subsubsection{\bf Normalized new vectors}
\label{sec:newvectors}
Just for this paragraph, let $F$ be a non-archimedean local field, $\mathcal{O}_F$ the 
ring of integers of $F$, and $\mathcal{P}_F$ the maximal ideal of $\mathcal{O}_F$. 
Let $(\pi,V)$ be an irreducible admissible generic representation of ${\rm GL}_n(F)$. Let $K_n(m)$ be the 
`mirahoric subgroup' of ${\rm GL}_n(\mathcal{O}_F)$ consisting of all matrices whose last row 
is congruent to $(0,\dots,0,*)$ modulo $\mathcal{P}_F^m$.  
Let $V_m := \{v \in V \ | \ \pi(k)v = 
\omega_{\pi}(k_{n,n})v, \forall k \in K_n(m) \}$. Let $\mathfrak{f}(\pi)$ be the least non-negative integer $m$ for which 
$V_m \neq (0)$. One knows that $\mathfrak{f}(\pi)$ is the conductor of $\pi$ and that 
$V_{\mathfrak{f}(\pi)}$ is one-dimensional. Any vector in $V_{\mathfrak{f}(\pi)}$ is called a {\it new vector} of 
$\pi$. 
Fix a nontrivial additive character $\psi$ of $F$, and assume that $V = W(\pi,\psi)$ is the Whittaker model for $\pi$.
If $\pi$ is unramified, i.e., $\mathfrak{f}(\pi) = 0$, 
then we fix a specific new vector called the {\it spherical vector}, denoted $w_{\pi}^{\rm sp}$, and 
normalized as $w_{\pi}^{\rm sp}(1_n) = 1.$ More generally, for 
any $\pi$, amongst all new vectors, there is a distinguished vector, called the {\it essential vector}, denoted 
$w_{\pi}^{\rm ess}$, and characterized by the property that for any irreducible unramified generic representation $\rho$ of ${\rm GL}_{n-1}(F)$ one has 
$$
\Psi(s, w_{\pi}^{\rm ess}, w_{\rho}^{\rm sp}) = 
\int_{U_{n-1}(F)\backslash G_{n-1}(F)}
w_{\pi}^{\rm ess}(\iota(g))w_{\rho}^{\rm sp}(g)|{\rm det}(g)|^{s - 1/2}\, dg = L(s, \pi \times \rho).
$$
If $\pi$ is unramified then $w_{\pi}^{\rm ess} = w_{\pi}^{\rm sp}$. In general, given $\pi$ 
there exists $t_{\pi} \in T_n(F)$ such that a new vector for $\pi$ is nonvanishing on $t_{\pi}$. Note that 
necessarily $t_{\pi} \in T_n^+(F)$, i.e., if $t_{\pi} = {\rm diag}(t_1,t_2,\dots,t_n)$ then $t_it_{i+1}^{-1} \in \mathcal{O}_F$ for all $1 \leq i \leq n-1$. 
We let $w_{\pi}^0$ be the new vector normalized such that 
$w_{\pi}^0(t_{\pi}) = 1$. If $\pi$ is unramified then we may and will take $t_{\pi} = 1_n$, and so
$w_{\pi}^0 = w_{\pi}^{\rm ess} = w_{\pi}^{\rm sp}$. For any $\sigma \in {\rm Aut}({\mathbb C})$ we may and will take 
$t_{\pi^{\sigma}} = t_{\pi}$. Then ${}^{\sigma}\!w_{\pi}^0 = w_{\pi^{\sigma}}^0$. Define $c_{\pi} \in {\mathbb C}^\times$ by $w_{\pi}^0 = c_{\pi}w_{\pi}^{\rm ess}$, i.e., 
$c_{\pi} = w_{\pi}^{\rm ess}(t_{\pi})^{-1}.$ 
For more details, see \cite[Sect.\,3.1.3]{raghuram-imrn}.

\subsubsection{\bf Choice of Whittaker vectors and cusp forms}
\label{sec:whittaker}
We now go back to global notation and choose global Whittaker vectors 
$w_{\Pi} = \otimes_v w_{\Pi,v} \in \Whit(\Pi,\psi)$ and $w_{\Sigma} = \otimes_v w_{\Sigma,v} \in \Whit(\Sigma, \psi^{-1})$ as follows. 
Let $S_{\Sigma}$ be the set of finite places $v$ where $\Sigma_v$ is unramified.  

\smallskip

\begin{enumerate}
\item If $v \notin S_{\Sigma}\cup S_\infty$,  we let $w_{\Pi, v} = w_{\Pi_v}^0$, and 
$w_{\Sigma, v} = w_{\Sigma_v}^{\rm sp}$. 

\smallskip

\item If $v \in S_{\Sigma}$, we let $w_{\Sigma, v} = w_{\Sigma_v}^0$, and let $w_{\Pi, v}$ be the unique Whittaker 
function whose restriction to $\GL_{n-1}(F_v)$ is supported on 
$\underline{U}_{n-1}(F_v)t_{\Sigma_v}K_{n-1}(\mathfrak{f}(\Sigma_v))$, and on this double coset it is given by 
$w_{\Pi, v}(ut_{\Sigma_v}k) = \psi(u)\omega_{\Sigma_v}^{-1}(k_{n-1,n-1})$, 
for all $u \in \underline{U}_{n-1}(F_v)$ and for all $k \in K_{n-1}(\mathfrak{f}(\Sigma_v))$. 

\smallskip

\item If $v \in S_\infty$, we let $w_{\Pi, v}$ and $w_{\Sigma, v}$ be arbitrary nonzero vectors. 
(Later, these will be cohomological vectors.) 
\end{enumerate}

Let $w_{\Pi_f} = \otimes_{v \notin S_\infty} w_{\Pi,v}$, $w_{\Pi_\infty} = \otimes_{v \in S_\infty} w_{\Pi,v}$, 
and $w_{\Pi} = w_{\Pi_{\infty}} \otimes w_{\Pi_f}$. 
Similarly, let $w_{\Sigma_f}$, $w_{\Sigma_{\infty}}$ and  
$w_{\Sigma}$. 
Let $\phi_{\Pi}$ (resp., $\phi_{\Sigma}$) be the cusp form corresponding to $w_{\Pi}$ (resp., $w_{\Sigma}$).

\subsubsection{\bf Integral representation of the central $L$-value}
For $\Re(s) \gg 0$, define $\Psi_{\infty}(s, w_{\Pi_{\infty}}, w_{\Sigma_{\infty}})$ to be  
$\prod_{v \in S_\infty} \Psi_v(s, w_{\Pi,v}, w_{\Sigma, v})$; 
this admits a meromorphic continuation to all of $\C.$  One identifies local $L$-factors $L_v(s, \Pi_v \times \Sigma_v)$, and proves that 
$\Psi_v(s, w_{\Pi,v}, w_{\Sigma, v})/L_v(s, \Pi_v \times \Sigma_v)$ is entire. Later we will be taking $s = 1/2$ to be a critical point, which says that $L_v(\tfrac12, \Pi_v \times \Sigma_v)$ is regular for all $v \in S_\infty$; hence, criticality of 
$s = 1/2$ gives that $\Psi_{\infty}(\tfrac12, w_{\Pi_{\infty}}, w_{\Sigma_{\infty}})$ is finite.

\begin{prop}
\label{prop:rankin-selberg}
We have 
{\small
$$
I(\tfrac12, \phi_{\Pi}, \phi_{\Sigma}) \ = \  
\frac{\Psi_{\infty}(\tfrac12, w_{\Pi_{\infty}}, w_{\Sigma_{\infty}})\, {\rm vol}(\Sigma)\, 
\prod_{v \notin S_{\Sigma} \cup \{\infty\}} c_{\Pi_v}}
{\prod_{v \in S_{\Sigma}} L(\tfrac12, \Pi_v \times \Sigma_v)} 
L_f(\tfrac12, \Pi \times \Sigma),
$$}
where ${\rm vol}(\Sigma) = \prod_{v \in S_{\Sigma}} {\rm vol}(K_{n-1}(\mathfrak{f}(\Sigma_v)) \in {\mathbb Q}^*$.
\end{prop}

\begin{proof}
See \cite[Prop.\,3.1]{raghuram-imrn}. 
\end{proof}

The main theorem on critical values of Rankin--Selberg $L$-functions follows by interpreting the above proposition in cohomology.

\subsection{\bf Automorphic cohomology}

\subsubsection{\bf Locally symmetric spaces} (See Harder \cite[1.1]{harder-inventiones}.) 
\label{sec:locally-symmetric}
Let $K_f$ be an open-compact subgroup of $G_n(\A_f) = \GL_n(\A_{F,f})$. Let us write $K_f = \prod_p K_p$ where each $K_p$ is an open compact subgroup of $G_n(\Q_p)$ and for almost all $p$ we have $K_p = \prod_{v | p} \GL_n(\O_v)$. 
Define the double-coset space 
$$
S^{G_n}_{K_f} \ = \ G_n(\Q) \backslash G_n(\A) /K_{n,\infty}^0 K_f 
\ = \ \GL_n(F) \backslash \GL_n(\A_F) /K_{n,\infty}^0 K_f.
$$
For brevity, let $K = K_{n, \infty}^0 K_f$, and define $X = G_n(\A)/K = G_n(\R)/K_{n,\infty}^0 \times G_n(\A_f)/K_f,$ i.e., $X$ is the product of the symmetric space 
$G_n(\R)/K_{n,\infty}^0$ with a totally disconnected space; any connected component of $X$ is of the form 
$X_g = G_n(\R)^0(g_\infty, g_f) K_f/K$ where $g = (g_\infty, g_f) \in G_n(\A)$ with $g_\infty \in \pi_0(G_n(\R)) \subset G_n(\R)$; for the last inclusion see \ref{sec:group-infinity}. The stabilizer of $X_g$ inside $G_n(\Q)$ is 
$
\Gamma_g := \{ \gamma \in G_n(\Q) \ : \gamma \in G_n(\R)^0 \cap g_fK_fg_f^{-1} \}. 
$
Any connected component of $S^{G_n}_{K_f}$ is of the form $\Gamma_g \backslash X_g \cong \Gamma_g \backslash G_n(\R)^0 / K_{n,\infty}^0$. 
However, $\Gamma_g$ does not act freely on $X_g$ since $S_{n,\infty} \subset K_{n,\infty}$. Indeed,  the stabilizer of every point in 
$X_g$ contains a congruence subgroup $\Delta$ of $S_n(\O_F)$; this $\Delta$ is independent of the point in $X_g$, but the congruence conditions on $\Delta$ depend on $K_f$. 
The group $\bar\Gamma_g = \Gamma_g/\Delta$ acts freely on $X_g$ and the quotient $\bar\Gamma_g\backslash X_g$ is a locally symmetric space. 
We will abuse terminology and sometimes refer to $S^{G_n}_{K_f}$ as  a locally symmetric space of $G_n$ with level structure $K_f$.

\subsubsection{\bf Sheaves on locally symmetric spaces, and their cohomology} 
\label{sec:sheaves} 
(Reference: see Harder--Raghuram~\cite{harder-raghuram}.) 
Given a dominant-integral weight $\mu \in X^+(T_n)$ and the associated representation $\M_{\mu,E}$, where 
$E$ is an extension of $\Q(\mu),$  we get a sheaf $\tM_{\mu, E}$ of $E$-vector spaces on $S^{G_n}_{K_f}$ as follows: 
Let $\pi : G_n(\A) /K_{n,\infty}^0 K_f \to S^{G_n}_{K_f}$ be the canonical projection. For any open subset $U$ of 
$S^{G_n}_{K_f}$ define the sections over $U$ by: 
$$
\tM_{\mu}(U) := \left\{ s : \pi^{-1}(U) \to \M_{\mu,E} \ 
\Bigg| \  
\begin{array}{l}
\mbox{$s$ is locally constant, and } \\
s(\gamma u) = \rho_{\mu}(\gamma) s(u),\ \mbox{for all $\gamma \in G_n(\Q)$ and $u \in \pi^{-1}(U)$} 
\end{array} \right\}.  
$$
This defines a sheaf of complex vector spaces on $S^{G_n}_{K_f}.$ Note  that even if $\M_{\mu, E} \neq 0$ it is possible that the sheaf $\tM_{\mu,E} = 0$. (See Harder \cite[1.1.3]{harder-inventiones}.)  Indeed, 
$\tM_{\mu,E} = 0$ unless the central character of $\rho_\mu$ has the infinity type of an algebraic Hecke character of $F$. 
Suppose $\mu = (\mu^\iota)_{\iota \in \cE_F}$, then define $a_\iota(\mu) := \sum_{i=1}^n \mu^\iota_i.$ The central character of $\rho_\mu$ has the infinity type of an algebraic Hecke character of $F$ if and only if the map $\iota \mapsto a_\iota(\mu)$ satisfies either (1) or (2) of \ref{sec:algebraic-hecke}. Henceforth, we will assume that $\mu$ satisfies this condition.

We are interested in the 
sheaf cohomology groups 
$$
H^{\bullet}(S^{G_n}_{K_f} , \tM_{\mu, E}^{\sf v}).
$$
Here $\tM_{\mu, E}^{\sf v}$ is the sheaf attached to the contragredient representation $\M_{\mu}^{\sf v}$  of 
$\M_\mu$. If $\mu^{\sf v} = -w_0(\mu)$, where $w_0$ is the element of the Weyl group of longest length, then 
$\M_{\mu}^{\sf v} = \M_{\mu^{\sf v}}.$
(This dualizing is only for convenience and is dictated by personal tastes. Dualizing here, avoids some negative signs elsewhere.) It is convenient to pass to the limit over all open-compact subgroups $K_f$ and let 
$H^{\bullet}(S^{G_n}, \tM_{\mu,E}^{\sf v}) := \varinjlim_{K_f} H^{\bullet}(S^{G_n}_{K_f} , \tM_{\mu,E}^{\sf v}). $
There is an action of $\pi_0(G_{n,\infty}) \times G_n(\A_f)$ on $H^{\bullet}(S^{G_n}, \tM_{\mu,E}^{\sf v})$, and 
the cohomology of $S^{G_n}_{K_f}$ is obtained by taking invariants under $K_f$, i.e., 
$H^{\bullet}(S^{G_n}_{K_f} , \tM_{\mu,E}^{\sf v}) = H^{\bullet}(S^{G_n}, \tM_{\mu,E}^{\sf v})^{K_f}.$ 
See \ref{sec:hecke-galois} below. 

\smallskip

Working at a transcendental level, i.e., taking $E = \C$, we can compute the above sheaf cohomology via the de~Rham complex, and then reinterpreting the de~Rham complex in terms of the complex computing relative Lie algebra cohomology, we get the isomorphism: 
$$
H^{\bullet}(S^{G_n}, \tM_{\mu,\C}^{\sf v})  \ \simeq \ 
H^{\bullet}(\g_n, K_{n,\infty}^0; \  C^{\infty}(G_n(\Q)\backslash G_n(\A)) \otimes \M_{\mu,\C}^{\sf v}) .
$$
With level structure $K_f$ we have:  
$
H^{\bullet}(S^{G_n}_{K_f}, \tM_{\mu,\C}^{\sf v})   \simeq  
H^{\bullet}(\g_n, K_{n,\infty}^0; \  C^{\infty}(G_n(\Q)\backslash G_n(\A))^{K_f} \otimes \M_{\mu,\C}^{\sf v}) .
$

\medskip
\subsubsection{\bf Cuspidal cohomology}

The inclusion $C^{\infty}_{\rm cusp} (G_n(\Q)\backslash G_n(\A)) \hookrightarrow C^{\infty}(G_n(\Q)\backslash G_n(\A))$ of the space of smooth cusp forms  in the space of all smooth functions induces, via results of Borel \cite{borel-duke},  an injection in cohomology; this defines cuspidal cohomology: 
\begin{equation}
\label{eqn:cusp-coh-defn}
\xymatrix{
H^{\bullet}(S^{G_n}, \tM_{\mu,\C}^{\sf v}) 
\ar[rr] & &
H^{\bullet}(\g_n, K_{n,\infty}^0; C^{\infty}(G_n(\Q)\backslash G_n(\A)) \otimes \M_{\mu,\C}^{\sf v})  \\
H^{\bullet}_{\rm cusp}(S^{G_n}, \tM_{\mu,\C}^{\sf v}) \ar@{^{(}->}[u]
\ar[rr] 
& & 
H^{\bullet}(\g_n, K_{n, \infty}^0; C^{\infty}_{\rm cusp}(G_n(\Q)\backslash G_n(\A)) \otimes \M_{\mu,\C}^{\sf v}) 
\ar@{^{(}->}[u]
}
\end{equation}
Using the usual decomposition of the space of cusp forms into a direct sum of cuspidal automorphic representations, we get the following fundamental decomposition of 
$\pi_0(G_n(\R)) \times G_n(\A_f)$-modules: 
\begin{equation}
\label{eqn:cusp-coh}
H^{\bullet}_{\rm cusp}(S^{G_n}, \tM_{\mu,\C}^{\sf v}) \ = \ 
\bigoplus_{\Pi} H^{\bullet}(\g_n, K_{n,\infty}^0;  \Pi_{\infty} \otimes  \M_{\mu,\C}^{\sf v}) \otimes \Pi_f.
\end{equation}
We say that {\it $\Pi$ contributes to the cuspidal cohomology of $G_n$ with coefficients in $\M_{\mu,\C}^{\sf v}$}, and 
we write $\Pi \in {\rm Coh}(G_n, \mu^{\sf v})$,  if $\Pi$ has a nonzero contribution to the above decomposition. Equivalently, if $\Pi$ is a cuspidal automorphic representation whose representation at infinity $\Pi_{\infty}$ after twisting by $\M_{\mu,\C}^{\sf v}$ has nontrivial relative Lie algebra cohomology. With a level structure $K_f$, (\ref{eqn:cusp-coh}) takes the form: 
\begin{equation}
\label{eqn:cusp-coh-level}
H^{\bullet}_{\rm cusp}(S^{G_n}_{K_f}, \tM_{\mu,\C}^{\sf v}) \ = \ 
\bigoplus_{\Pi} H^{\bullet}(\g_n, K_{n,\infty}^0;  \Pi_{\infty} \otimes  \M_{\mu,\C}^{\sf v}) \otimes \Pi_f^{K_f}
\end{equation}
We write $\Pi \in {\rm Coh}(G_n, \mu^{\sf v}, K_f)$ if a cuspidal automorphic representation $\Pi$ contributes nontrivially to (\ref{eqn:cusp-coh-level}). Note that ${\rm Coh}(G_n, \mu^{\sf v}, K_f)$ is a finite set, and 
${\rm Coh}(G_n, \mu^{\sf v}) = \cup_{K_f} {\rm Coh}(G_n, \mu^{\sf v}, K_f).$

\subsubsection{\bf Purity} 
\label{sec:purity}
Let $\mu \in X^+(T_n)$ be a dominant integral weight satisfying the condition in (\ref{sec:sheaves}). 
Suppose $\Pi \in \Coh(G_n, \mu^{\sf v}).$ The fact that $\mu^{\sf v}$ supports cuspidal cohomology places some restrictions on $\mu.$ First of all, essential unitarity of $\Pi$, and in particular of $\Pi_\infty$ gives, via Wigner's Lemma, essential self-duality of $\mu$: there is an integer ${\sf w}(\mu)$ such that 
\begin{enumerate}
\item For $v \in S_r$ and $1 \leq i \leq n$ we have $\mu^{\iota_v}_i + \mu^{\iota_v}_{n-i+1} = {\sf w}(\mu);$ 
\smallskip
\item For $v \in S_c$ and $1 \leq i \leq n$ we have $\mu^{\bar{\iota}_v}_i + \mu^{\iota_v}_{n-i+1} = {\sf w}(\mu).$ 
\end{enumerate}
We will call such a weight $\mu$ as a {\it pure weight} and call ${\sf w}(\mu)$ the {\it purity weight of $\mu.$} Let 
$X^+_0(T_n)$ denote the set of dominant integral pure weights.

\smallskip

Furthemore, any $\Pi \in \Coh(G_n, \mu^{\sf v})$ satisfies a purity condition (Clozel \cite[Lem.\,4.9]{clozel}) 
which is a translation to the automorphic side of the phenomenon that the associated motive $M(\Pi)$ is 
pure--a condition on the Hodge types of $M(\Pi)$. 
For $v \in S_\infty$, let $r_L(\Pi_v)$ stand for the Langlands parameter of $\Pi_v$; it is an $n$-dimensional 
semi-simple representation of the Weil group $W_{F_v}$ of $F_v$; $\C^\times \subset W_{F_v}$ as a subgroup of 
index at most $2$. The representation 
$r_L(\Pi_v) |\ |_\C^{(1-n)/2}$ of $\C^\times$ is a sum of $n$ characters $z \mapsto z^p \overline{z}^q$ then 
$p, q \in \Z.$ Purity says that there is an integer ${\sf w}(\Pi)$ such that for any $v \in S_\infty$ all the exponents in 
$r_L(\Pi_v) |\ |_\C^{(1-n)/2}$ satisfy $p+q = {\sf w}(\Pi).$ We will call ${\sf w}(\Pi)$ the {\it purity weight of $\Pi$}, and it is related to the purity weight of $\mu$ via
${\sf w}(\mu) = n - 1 + {\sf w}(\Pi).$

\smallskip

Finally, by Clozel's theorem that cuspidal cohomology has a rational structure, we get 
$$
\Pi \in \Coh(G_n, \mu^{\sf v}) \ \Longrightarrow \  {}^\sigma\Pi \in \Coh(G_n, {}^\sigma\!\mu^{\sf v}), \quad \forall 
\sigma \in \autc.
$$
In particular, ${}^\sigma\!\mu$ also satisfies the purity conditions (1) and (2) above. Note that 
${\sf w}(\mu) = {\sf w}({}^\sigma\!\mu).$ Since this purity weight is going to appear frequently, we will often denote: 
${\sf w} := {\sf w}(\mu) = {\sf w}({}^\sigma\!\mu).$

\begin{defn}\label{def:strongly-pure}
Let $\mu \in X^+(T_n)$ be a dominant integral weight satisfying the condition in (\ref{sec:sheaves}). 
We say $\mu$ is {\it strongly pure} if 
${}^\sigma \!\mu$ is pure for all $\sigma \in \autc.$ Let $X_{00}^+(T_n)$ stand for the set of dominant integral strongly-pure weights. (Note that if  a dominant integral weight $\mu$, is such that for all $\sigma \in \autc$, the weight 
${}^\sigma\!\mu$ satisfies the purity conditions (1) and (2) above, then necessarily, $\mu$ satisfies the the condition in (\ref{sec:sheaves}).)
\end{defn}

For any $F$, we have the following inclusions $X^+_{00}(T_n) \subset X^+_0(T_n) \subset X^+(T_n) \subset X^*(T_n)$ and in general they are all strict inclusions. 
If $F$ is a totally real field or a CM field (totally imaginary quadratic extension of a totally real field) 
then $\mu$ is pure if and only if $\mu$ is strongly pure. However, this is not true in general; it is easy to give an example of  a weight $\mu$ which is pure but not strongly pure when the base field is $F = \Q(2^{1/3})$ or it's Galois closure. 
For any number field, one may see that there are strongly pure weights $\mu.$ Take an integer $b$ and integers 
$a_1 \geq a_2 \geq \cdots \geq a_n$ such that $a_j + a_{n-j+1} = b;$ now for 
each $\iota \in \cE_F$ put $\mu^\iota = (a_1,\dots, a_n)$;  then $\mu$ is strongly pure with ${\sf w}(\mu) = b.$
Such a weight may be called a {\it parallel weight}.

\medskip
\subsubsection{\bf Hecke action Vs Galois action}
\label{sec:hecke-galois} 
We record here a fact, well-known to experts, but seemingly hard to find in the literature, on the relation between the action of $G_n(\A_f) \times \pi_0(G_{n,\infty})$ and the action of $\autc$ on 
$H^{\bullet}_{\rm cusp}(S^{G_n}, \tM_{\mu,\C}).$ Given $\mu \in X^+_{00}(T_n)$ and $\sigma \in \autc$, 
let's denote $T_\sigma : \M_{\mu} \to \M_{{}^\sigma\!\mu} = \M_{\mu} \otimes_{\C} {}_\sigma\C$ 
for the $\sigma$-linear $G_n(F)$-isomorphism
as in \ref{sec:fin-dim-repn}. Let $T_\sigma^* : \tM_{\mu} \to \tM_{{}^\sigma\!\mu}$ denote the corresponding $\sigma$-linear isomorphism of sheaves on $S^{G_n}_{K_f},$ and 
$T_\sigma^\bullet : H^{\bullet}(S^{G_n}_{K_f}, \tM_{\mu}) \to 
H^{\bullet}(S^{G_n}_{K_f}, \tM_{{}^\sigma\!\mu})$
the $\sigma$-linear isomorphism in cohomology. Then $T_\sigma^\bullet$ preserves  inner cohomology, and by Clozel~\cite{clozel}, it also preserves cuspidal cohomology. 
Next, given $\ul x = (\ul x_f, \ul x_\infty) \in G_n(\A_f) \times G_{n, \infty} ,$  the map 
$\ul g \mapsto \ul g \ul x$ induces a homeomorphism 
$m_{\ul x} : S^{G_n}_{K_f} \to S^{G_n}_{\ul x_f^{-1}K_f \ul x_f},$ 
whose dependence on $x_\infty$ is only via its class in $\pi_0(G_{n,\infty}).$ 
As sheaves on $S^{G_n}_{K_f},$  we have
$ m_{\ul x}^*\tM_\mu \simeq  \tM_\mu,$
where the left hand side is the pull-back via $m_{\ul x}$ of the sheaf $\tM_\mu$ on $S^{G_n}_{\ul x_f^{-1}K_f \ul x_f}.$
Furthermore, we have $T_\sigma^*(m_{\ul x}^*(\tM_{\mu, \C})) = 
m_{\ul x}^* (T_\sigma^*(\tM_{\mu, \C}))$ as sheaves on $S^{G_n}_{K_f}.$ 
Taking cohomology, and passing to the limit over all $K_f$, gives an equality 
$T_\sigma^\bullet \circ m_{\ul x}^\bullet  \ = \ 
m_{\ul x}^\bullet \circ T_\sigma^\bullet$
of maps from 
$H^{\bullet}(S^{G_n}, \tM_{\mu, \C})$ to $H^{\bullet}(S^{G_n}, \tM_{{}^\sigma\!\mu, \C}).$ The same holds for cohomology with compact supports, and hence for inner and also for cuspidal cohomology via \cite{clozel}. Therefore, if 
$\Pi_f \otimes \epsilon$ is a $G_n(\A_f) \times \pi_0(G_{n,\infty})$-module appearing in inner or cuspidal cohomology, then:  
\begin{equation}
\label{eqn:t-sigma-m-x-repn}
T_\sigma^\bullet( \Pi_f \otimes \epsilon) \ = \   {}^\sigma\Pi_f \otimes \epsilon,  
\end{equation}
since the image of $\epsilon$ is $\pm 1$ and by the definition of ${}^\sigma\Pi_f.$ 
See \ref{sec:errata} below.

\medskip
\subsection{\bf Archimedean considerations} 
\label{sec:archimedean}

Let $\mu \in X^+_{00}(T_n)$ be a strongly pure weight, and let $\Pi \in {\rm Coh}(G_n,\mu^{\sf v})$. The purpose of this section is to write down explicitly the representation 
$\Pi_\infty$ at infinity in terms of $\mu$; this is possible up to a sign; see Propositions~\ref{prop:local-rep-n-even}, \ref{prop:local-rep-n-odd} and \ref{prop:local-rep-c} below. We record a well-known conclusion--due to Clozel--on the possible degrees in which one has nontrivial cuspidal cohomology; see Prop.\,\ref{prop:cuspidal-range}. This gives rise to an interesting numerology
that ultimately permits us to give a cohomological interpretation to Rankin--Selberg $L$-values. 
Also, with local representations at hand, we compute the set of critical points of Rankin--Selberg $L$-functions. 
Since $\Pi_\infty = \otimes_{v \in S_\infty} \Pi_v$, the problem of describing $\Pi_\infty$  is a purely local one. 
We begin by taking up real and complex places separately.

\subsubsection{\bf Cohomological representations of $\GL_n(\R)$}
\label{sec:repn-glnR}
Fix a place $v \in S_r$, and since $v$ is fixed, we drop it from our notations just for this subsection. 
For example, $\mu_v$ is just $\mu = (\mu_1,\dots,\mu_n),$  an $n$-tuple of integers with 
$\mu_1 \geq \cdots \geq \mu_n$ and $\mu_i + \mu_{n-i+1} = {\sf w}.$ We will now define the representation $J_\mu$ if $n$ is even; and two representations $J_\mu^{\pm}$ if $n$ is odd. For this we need to fix some notation for discrete series representations of $\GL_2(\R)$. 

\medskip
\paragraph{\bf Discrete series for $\GL_2(\R)$.} 
For any integer $l \geq 1$, let $D_l$ stand for the discrete series representation with lowest non-negative $K$-type being the character 
$\left(\begin{smallmatrix} \cos{\theta} & -\sin{\theta} \\ 
\sin{\theta}& \cos{\theta} \end{smallmatrix}\right)  \mapsto  e^{ - i (l+1) \theta}$, and central character $a \mapsto {\rm sgn}(a)^{l+1}$. 
Note the shift from $l$ to $l+1$. The representation at infinity for a holomorphic  elliptic modular cusp form of weight $k$ is $D_{k-1}$. It is well-known that discrete series representations of $\GL_2(\R)$, possibly twisted by a half-integral power of absolute value, have nontrivial cohomology. For brevity, let $(\g_2,K_2^0) := (\gl_2, {\rm SO}(2)Z_2(\R)^0)$. 
For a dominant integral weight $\nu = (a,b)$, with integers $a \geq b$, the basic fact here is that  
there is a non-split exact sequence of $(\g_2,K_2^0)$-modules: 
\begin{equation}
\label{eqn:exact-seq-dsr}
0 \to 
D_{a-b +1} \otimes |\ |_{\R}^{(a+b)/2}  \to 
{\rm Ind}_{B_2(\R)}^{\GL_2(\R)}(\chi_{(a,a)}|\ |^{1/2} \otimes \chi_{(b,b)}|\ |^{-1/2}) \to
\M_{\nu,\C}
\to 0, 
\end{equation}
Moreover, $H^{\bullet}(\g_2, K_2^0; (D_{a-b +1} \otimes |\ |_{\R}^{(a+b)/2}) \otimes \M_{\nu,\C}^{\sf v}) \neq 0$ if and only if
${\bullet} = 1,$ and that dimension of $H^1(\g_2, K_2^0; (D_{a-b +1} \otimes |\ |_{\R}^{(a+b)/2}) \otimes \M_{\nu,\C}^{\sf v})$ is two, with both the characters of ${\rm O}(2)/{\rm SO}(2)$ appearing exactly once. (For more details see Raghuram-Tanabe \cite[Sect.\,3.1]{raghuram-tanabe}.)

\medskip
\paragraph{\bf Definition of $J_\mu$ when $n$ is even}
Given $\mu = (\mu_1,\dots,\mu_n)$ define an $n$-tuple $\ell = \ell(\mu) = (\ell_1,\dots,\ell_n)$ by 
$\ell = 2\mu + 2\bfgreek{rho}_n  -{\sf w},$ i.e., we have 
$$
\ell_i \ = \ 2\mu_i + n - 2i + 1 - {\sf w} \ = \ \mu_i - \mu_{n-i+1} + n - 2i  + 1, \quad 1\leq i\leq n. 
$$
Observe that $\ell_1> \ell_2> \cdots >\ell_{n/2} \geq 1$ and $\ell_i = - \ell_{n-i+1}.$  Let $P$ be the $(2,\dots,2)$-parabolic subgroup of $\GL_n(\R),$ i.e., $P$ has the Levi quotient $L=\prod_{i=1}^{n/2} \GL_2(\R).$ Define the parabolically induced representation: 
\begin{equation}
\label{eq:j-mu-even}
J_\mu := \textrm{Ind}^{\GL_n(\R)}_{P(\R)}\left( D(\ell_1)|\!\det\!|^{{\sf w}/2} \otimes \cdots \otimes D(\ell_{n/2}) |\!\det\!|^{{\sf w}/2} \right).
\end{equation}
(This is a small change in notation from some of my earlier papers (\cite{raghuram-imrn}, \cite{raghuram-shahidi-aim}) where $J_\mu$, following Mahnkopf \cite{mahnkopf-jussieu}, was denoted $J({\sf w}, \ell)$.) 
We will refer to the integers in $\ell$ as the cuspidal parameters of $J_\mu$. It is well-known that 
$J_\mu$ is irreducible, essentially tempered and generic (being fully induced from essentially discrete series), and 
$H^\bullet(\gl_n, \SO(n); J_\mu \otimes \M_\mu^{\sf v}) = H^\bullet(\gl_n, \SO(n)\R^\times_+; J_\mu \otimes \M_\mu^{\sf v}) \neq 0.$ The  following proposition describes the local component for a real place of a global cohomological cuspidal representation; it says that when $n$ is even, the highest weight $\mu_v$  determines the isomorphism class of $\Pi_v$.

\begin{prop}
\label{prop:local-rep-n-even}
Let $\mu \in X^+_{00}(T_n)$ and $\Pi \in {\rm Coh}(G_n,\mu^{\sf v})$. Suppose $n$ is even. Let $v \in S_r$ be a real place. Then $\Pi_v \cong J_{\mu_v}.$
\end{prop}

\medskip
\paragraph{\bf Definition of $J_\mu^{\epsilon}$ when $n$ is odd}

When $n$ is odd, for any sign $\epsilon$, we define a representation $J_\mu^{\epsilon}$ as follows. The cuspidal parameter $\ell$ is again defined by 
$\ell = 2\mu + 2\bfgreek{rho}_n  -{\sf w}.$ This time, let $P$ to be the $(2,\dots,2,1)$-parabolic subgroup. Define 
\begin{equation}
\label{eq:j-mu-odd}
J_\mu^\epsilon := \textrm{Ind}^{\GL_n(\R)}_{P(\R)}\left( D(\ell_1)|\!\det\!|^{{\sf w}/2} \otimes \cdots \otimes D(\ell_{n/2}) |\!\det\!|^{{\sf w}/2} \otimes \epsilon |\!\det\!|^{{\sf w}/2} \right).
\end{equation}
It is well-known that $J_\mu^\epsilon$ is irreducible, essentially tempered, generic and that the relative cohomology group 
$H^\bullet(\gl_n, \SO(n); J_\mu^\epsilon \otimes \M_\mu^{\sf v}) = 
H^\bullet(\gl_n, \SO(n)\R^\times_+; J_\mu^\epsilon \otimes \M_\mu^{\sf v})$ 
is one-dimensional. Reverting to global notation, we have the following proposition which says that 
when $n$ is odd, the highest weight $\mu_v$ \underline{and} the sign of the central character of $\Pi_v$ determine the isomorphism class of $\Pi_v$ as a representation of $\GL_n(\R)$. 

\begin{prop}
\label{prop:local-rep-n-odd}
Let $\mu \in X^+_{00}(T_n)$ and $\Pi \in {\rm Coh}(G_n,\mu^{\sf v}).$ Suppose $n$ is odd. Let $v \in S_r$ be a real place. Then
$\Pi_v \cong J_{\mu_v}^{\epsilon_v},$
where $\epsilon_v(-1) = \omega_{\Pi_v}(-1)\cdot(-1)^{(n-1)/2}.$
\end{prop}
The reader, who wishes to verify the above equality of signs, should note the following consequences of $n$ being odd: 
(1) The purity weight ${\sf w}$ is even since ${\sf w} = 2 \mu_{(n+1)/2}$, and (2) The cuspidal parameters are even since
$l_i = \ 2\mu_i + n - 2i + 1 - {\sf w}$.

\subsubsection{\bf Cohomological representations of $\GL_n(\C)$} 
\label{sec:glnc}
Let $\mu \in X^+_{00}(T_n)$ and $\Pi \in {\rm Coh}(G_n,\mu^{\sf v}).$  For a complex place $v$, 
$\mu_v$ is a pair of $n$-tuples $(\mu^{\iota_v}, \mu^{\bar{\iota}_v})$, where $\iota_v$ is a complex embedding corresponding to $v$ that has been noncanonically chosen and fixed; and $\bar{\iota}_v$ is the conjugate embedding. Since $v \in S_c$ is fixed, we will drop it from our notations. Hence, 
$\mu = (\mu^{\iota}, \mu^{\bar{\iota}})$; we will further simplify our notation and write 
$\mu^\iota = (\mu_1, \dots, \mu_n)$ and $\mu^{\bar{\iota}} = (\mu_1^* \dots, \mu_n^*)$. 
Recall from Sect.\,\ref{sec:purity} that the integers in the pair $(\mu^{\iota}, \mu^{\bar{\iota}})$ are related by: 
$\mu_i^* + \mu_{n-i+1} = {\sf w}.$  Hence, we have 
$$
\mu^\iota = (\mu_1, \dots, \mu_n), \ \ {\rm and} \ \ \mu^{\bar{\iota}} = ({\sf w}-\mu_n, \dots, {\sf w} - \mu_1).
$$
Define the cuspidal parameters as: 
{\small 
\begin{equation}
\label{eqn:cuspidal-parameters-c}
\begin{array}{llllll}
{\sf a} & := & \mu + \bfgreek{rho}_n; & {\sf a} = (a_1,\dots,a_n) & := & \left(\mu_1+\tfrac{n-1}{2}, \mu_2+\tfrac{n-3}{2},\dots, \mu_n - \tfrac{(n-1)}{2}\right) \\ 
& & & & & \\
{\sf b} & := & {\sf w} - \mu - \bfgreek{rho}_n;& {\sf b} = (b_1,\dots,b_n) & := & \left({\sf w} - \mu_1 - \tfrac{n-1}{2}, {\sf w} - \mu_2 - \tfrac{n-3}{2},\dots, {\sf w} - \mu_n + \tfrac{(n-1)}{2} \right). 
\end{array}
\end{equation}}

\noindent Now define the representation $J_\mu$ to be induced from the Borel subgroup $B(\C)$ of $\GL_n(\C)$ as: 
\begin{equation}
\label{eqn:j-mu-c}
J_\mu := \textrm{Ind}^{\GL_n(\C)}_{B(\C)}\left( z^{a_1}\bar{z}^{b_1} \otimes \dots \otimes z^{a_n}\bar{z}^{b_n}\right).
\end{equation}
where, for any half-integers $a,b$, by $z^a \bar{z}^b$ we mean the character of $\C^\times$ which sends $z$ to $z^a \bar{z}^b$. It is well-known that $J_\mu$ is irreducible, essentially tempered, generic and 
$H^\bullet(\gl_n, \Unit(n); J_\mu  \otimes \M_\mu^{\sf v})  \neq 0.$ Reverting to global notation, we have: 

\begin{prop}
\label{prop:local-rep-c}
Let $\mu \in X^+_{00}(T_n)$ and $\Pi \in {\rm Coh}(G_n,\mu^{\sf v}).$ Let $v \in S_c,$ i.e., $v$ is a complex place. Then $\Pi_v \cong J_{\mu_v}.$
\end{prop}

\subsubsection{\bf The cuspidal range}
\label{sec:cuspidal-range}
In this subsection we record well-known bounds for the possible degrees in which there can be nonzero cuspidal cohomology. These bounds depend only on the rank $n$ of $\GL_n$, and the  numbers $r_1$ and $r_2$ of real and complex embeddings of $F$.  

\begin{prop}
\label{prop:cuspidal-range}
Define the following numbers: 
\begin{enumerate}
\item[] The bottom degrees: 
$$b_n^\R := \left[\frac{n^2}{4}\right],  \ \ \ b_n^\C := n(n-1)/2.$$
\item[] The top degrees: 
$$t_n^\R = b_n^\R + \left[\frac{n-1}{2}\right], \ \ \  t_n^\C = b_n^\C + n-1.$$
\end{enumerate} 
Now define the bottom degree and top degree for the group $G_n = R_{F/\Q}(\GL_n/F)$ as: 
$$
b_n^F = r_1b_n^\R + r_2b_n^\C, \ \ \ t_n^F = r_1t_n^\R + r_2t_n^\C, \ \ \ 
\tilde{t}_n^F = t_n^F + [F:\Q]-1 
$$
Let $\mu \in X^+_{00}(T_n).$ Then 
$
H^{\bullet}_{\rm cusp}(S^{G_n}, \tM_{\mu}^{\sf v}) \neq 0  \iff 
b_n^F \leq \bullet \leq \tilde{t}_n^F.
$
\end{prop}

\begin{proof}
Given $\mu \in  X^+_{00}(T_n)$, and a $\Pi \in {\rm Coh}(G_n, \mu^{\sf v}),$ we know that 
$\Pi_\infty = \otimes_{v \in S_\infty} J_{\mu_v}^{\epsilon_v}$. (The symbol $\epsilon_v$ is a nonempty condition only for $n$ odd and $v \in S_r$.) For $v \in S_\infty$, let $\tilde{G}_{n,v} = \{g_v \in \GL_n(F_v) \ : \ |\det(g_v)| = 1\}$ and let 
$\tilde{\g}_{n,v}$ be the Lie algebra of $\tilde{G}_{n,v}.$ Note that $\g_{n,v} = \tilde{\g}_{n,v} \oplus \z_{n,v}$ and summing over $v \in S_\infty$, we have $\g_n = \tilde{\g}_n + \z_n.$ By Wigner's Lemma we have: 
$$
H^\bullet(\g_n, K_{n,\infty}^0; \Pi_{\infty} \otimes  \M_{\mu, \C}^{\sf v}) \ = \ 
H^\bullet(\tilde{\g}_n, C_{n,\infty}^0; \Pi_{\infty} \otimes  \M_{\mu, \C}^{\sf v}) \otimes 
\wedge^\bullet(\z_n/\s_n).
$$
The term $\wedge^\bullet(\z_n/\s_n)$ accounts for the difference between $\tilde{t}_n^F$ and $t_n^F$. 
By the K\"unneth formula we get 
$$
H^{q}(\tilde{\g}_n, C_{n,\infty}^0;  \Pi_{\infty} \otimes  \M_{\mu, \C}^{\sf v}) = \bigotimes_{\sum_v q_v = q} 
H^{q_v}(\tilde{\g}_n, C_{n,v}^0;  \Pi_v \otimes  \M_{\mu_v, \C}^{\sf v}). 
$$
From Clozel \cite[Lemme 3.14]{clozel}, we get 
$H^{q_v}(\tilde{\g}_n, C_{n,v}^0;  \Pi_v \otimes  \M_{\mu_v, \C}^{\sf v}) \neq 0$ if and only if 
$b_n^{F_v} \leq q_v \leq t_n^{F_v}.$
\end{proof}

\smallskip

Motivated by the numerical coincidence to be discussed below, we will focus our attention exclusively on cohomology in degree $\bullet = b_n^F$. (In contrast, see my paper with Grobner \cite{grobner-raghuram-ajm} where we considered top-degree cuspidal cohomology.) 

\begin{cor}
Let $\Pi \in {\rm Coh}(G_n, \mu^{\sf v})$ for $\mu \in X^+_{00}(T_n).$ 
\begin{enumerate}
\item If $n$ is even then every character $\epsilon = (\epsilon_v)_{v \in S_r}$ of $\pi_0(K_{n,\infty}) = K_{n,\infty}/K_{n,\infty}^0$ 
appears  once in 
$H^{b_n^F}(\g_n, K_{n,\infty}^0;  \Pi_{\infty} \otimes  \M_{\mu,\C}^{\sf v});$ hence cuspidal cohomology in degree 
$b_n^F$ decomposes as
$$
H^{b_n^F}_{\rm cusp}(S^{G_n}, \tM_{\mu}^{\sf v}) = \bigoplus_\Pi
\bigoplus_\epsilon  \epsilon \otimes \Pi_f, 
$$
where $\Pi$ runs over ${\rm Coh}(G_n, \mu^{\sf v})$ and $\epsilon$ runs over all possible characters of $\pi_0(K_{n,\infty})$. Hence, each $\Pi_f$ appears $2^{r_1}$ times in cuspidal cohomology in degree $b_n^F.$ 

\item If $n$ is odd then $H^{b_n^F}(\g_n, K_{n,\infty}^0;  \Pi_{\infty} \otimes  \M_{\mu,\C}^{\sf v})$ is one-dimensional and the character of $\pi_0(K_{n,\infty})$ which appears is denoted as $\epsilon_{\Pi_\infty}$. In this case, cuspidal cohomology decomposes as 
$$
H^{b_n^F}_{\rm cusp}(S^{G_n}, \tM_{\mu}^{\sf v}) = \bigoplus_\Pi \epsilon_{\Pi_\infty} \otimes \Pi_f, 
$$
where $\Pi$ runs over ${\rm Coh}(G_n, \mu^{\sf v})$; and each $\Pi_f$ appears once.  
\end{enumerate}
\end{cor}

As in \cite{raghuram-shahidi-imrn}, we say that an $r_1$-tuple of signs $\epsilon = (\epsilon_v)_{v\in S_r}$ is 
{\it permissible} for $\mu$ if $\epsilon = \epsilon_{\Pi_\infty}$ for some 
$\Pi \in {\rm Coh}(G_n, \mu^{\sf v})$ when $n$ is odd, and is any of the possible $2^{r_1}$ signatures when $n$ is even. 
For any $n$, given $\mu \in X^+_{00}(T_n),$ $\Pi \in {\rm Coh}(G_n, \mu^{\sf v})$, and given a permissible signature 
$\epsilon$ for $\mu$, the $\epsilon \otimes \Pi_f$ isotypic component in cuspidal cohomology in degree $b_n^F$, via the isomorphism in (\ref{eqn:cusp-coh-defn}), can be expressed in terms of the space 
$V_\Pi$ of cusp forms realizing $\Pi$ as 
$$
H^{b_n^F}_{\rm cusp}(S^{G_n}, \tM_{\mu,\C}^{\sf v})( \epsilon \otimes \Pi_f) \ = \ 
H^{b_n^F}(\g_n, K_{n,\infty}^0; V_\Pi \otimes  \M_{\mu,\C}^{\sf v})(\epsilon), 
$$
where the right hand side is the $\epsilon$-isotypic component in 
$H^{b_n^F}(\g_n, K_{n,\infty}^0; V_\Pi \otimes  \M_{\mu, \C}^{\sf v})$ for the action of $K_{n,\infty}/K_{n,\infty}^0$.

\subsubsection{\bf A numerical coincidence}
\label{sec:numerical}
We record a relation between the numbers $b_n^F$ and $b_{n-1}^F$ and the dimension of a locally symmetric space for $G_{n-1}$ which is crucial for giving a cohomological interpretation to the Rankin--Selberg theory for $\GL_n \times \GL_{n-1}$. Define
$$
\tilde{S}^{G_n}_{K_f} := G_n(\Q) \backslash G_n(\A)/ C_{n,\infty}^0 K_f \ = \  \GL_n(F) \backslash \GL_n(\A) / C_{n,\infty}^0 K_f
$$
where $K_f$ is an open compact subgroup of $G_n(\A_f)$, and 
$C_{n,\infty}^0 =   \prod_{v \in S_r} \SO(n) \times \prod_{v \in S_c} \Unit(n)$ 
is the connected component of the identity in the maximal compact subgroup $C_{n,\infty}$ of $G_n(\R).$ 
Since $K_{n, \infty}^0 = S_n({\mathbb R})^0C_{n,\infty}^0,$ see Sect.\,\ref{sec:group-infinity} for our notations, 
we get a canonical fibration $\phi$ given by: 
$$
\xymatrix{
\tilde{S}^{G_n}_{K_f} &  = &  G_n(\Q) \backslash G_n(\A)/ C_{n,\infty}^0 K_f \ar[d]^{\phi} \\ 
S^{G_n}_{K_f} & =  &  G_n(\Q) \backslash G_n(\A)/ K_{n,\infty}^0 K_f
}
$$

\begin{prop}
\label{prop:numerical}
Let $n \geq 2.$ 
Let $b_n^F$ be the bottom degree for $G_n$ as defined in Prop.\,\ref{prop:cuspidal-range}.  Then 
$$
b_n^F + b_{n-1}^F = {\rm dim}(\tilde{S}^{G_{n-1}}_{R_f}), 
$$
for any open compact subgroup $R_f$ of $G_{n-1}(\A_f)$. 
\end{prop}

\begin{proof}
Note that 
\begin{eqnarray*}
{\rm dim}(\tilde{S}^{G_{n-1}}_{R_f}) & = & {\rm dim}(G_{n-1}(\R)^0/C_{n,\infty}^0) \\ 
& = & r_1 \cdot {\rm dim}(\GL_{n-1}(\R)^0/\SO(n-1)) \ + \ r_2 \cdot {\rm dim}(\GL_{n-1}(\C)/U(n-1)). 
\end{eqnarray*}
The proof follows if we check that 
\begin{eqnarray*}
b_n^\R + b_{n-1}^\R & = & {\rm dim}(\GL_{n-1}(\R)^0/\SO(n-1)), \\ 
b_n^\C + b_{n-1}^\C & = & {\rm dim}(\GL_{n-1}(\C)/U(n-1)).
\end{eqnarray*}
These are easy to verify using the definitions of $b_n^\R$ and $b_n^\C;$ for example, $b_n^\C + b_{n-1}^\C = n(n-1)/2 + (n-1)(n-2)/2 = (n-1)^2 = 
{\rm dim}(\GL_{n-1}(\C)/U(n-1)).$
\end{proof}

For $F = \Q$ and $n=3,$ this kind of numerology is apparent in Schmidt \cite{schmidt} and 
Mahnkopf \cite{mahnkopf-crelle}. 
For $F = \Q$ and general $n$, see Kazhdan, Mazur and Schmidt \cite[Table on p.\,100]{kms}; in this situation, the numerology was cleverly exploited by Mahnkopf \cite{mahnkopf-jussieu}, and so was used in my paper \cite{raghuram-imrn}. For general $n$ and $F$, this was recently used by Januszewski \cite{januszewski} in his study of modular symbols.


\medskip
\subsubsection{\bf Critical points and compatibility of coefficient systems} 
\label{sec:critical-compatible}

\medskip
\paragraph{\bf Definition of the critical set} 

Consider the Rankin--Selberg $L$-function $L(s, \Pi \times \Sigma)$ where $\Pi$ (resp., $\Sigma$) is a cuspidal automorphic representation 
$\GL_n(\A_F)$ (resp., $\GL_{n-1}(\A_f)$). 

\begin{defn}
\label{defn:critical} 
We say that $s_0 = \tfrac12 + m \in \tfrac12 + \Z$ is critical for  $L(s, \Pi \times \Sigma)$ if both $L_\infty(s, \Pi_\infty \times \Sigma_\infty)$ and 
$L_\infty(1-s, \Pi_\infty^{\sf v} \times \Sigma_\infty^{\sf v})$ are regular at $s = s_0,$ i.e., both the $L$-factors at infinity on either side of the functional equation are holomorphic 
at $s=s_0.$ 
\end{defn}

By definition, we only look at half-integers, i.e, the critical set is a subset of $ \tfrac12 + \Z.$ This has to do with the so-called motivic normalization: that if a cuspidal representation $\Theta$ of $\GL_r$ corresponds to a motive $M$ then under this correspondence, $L(s + \tfrac{1-r}{2}, \Theta) = L(s, M)$. On the motivic side, one always looks for critical points amongst integers; see Deligne~\cite{deligne}. Transcribing to the automorphic side, one looks for critical points amongst $ \tfrac{r-1}2 + \Z;$ in particular, if $r$ is even then we look for critical points in $ \tfrac12 + \Z.$ In our situation, assuming Langlands's functoriality, we have $\Theta = \Pi \boxtimes \Sigma$, which is (usually) a cuspidal representation of $\GL_r(\A_F)$ with $r = n(n-1);$ in particular $r$ is even; hence the critical set for $L(s, \Pi \times \Sigma)$ consists of only half-integers. Another easy point to note is that given a particular half-integer $s_0 = \tfrac12 + m,$ to check whether $s_0$ is critical or not is an entirely local calculation because  
$L_\infty(s_0, \Pi_\infty \times \Sigma_\infty) = \prod_{v \in S_\infty} L_v(s_0, \Pi_v \times \Sigma_v)$ and local $L$-factors are nonvanishing everywhere.

\medskip
\paragraph{\bf Branching rule for the pair $(\GL_n(\C), \GL_{n-1}(\C))$}
Working with local notations, let $\mu = (\mu_1,\dots,\mu_n)$ be a dominant integral weight for 
$\GL_n(\C)$ and $\M_\mu$ 
be the irreducible finite-dimensional representation of the algebraic Lie group $\GL_n(\C)$ with highest weight $\mu.$ 
Similarly, we have $\lambda = (\lambda_1,\dots,\lambda_{n-1})$ and $\M_{\lambda}$ for $\GL_{n-1}(\C)$. The following branching rule is well-known (see, for example, Goodman-Wallach \cite[Thm.\,8.1.1]{goodman-wallach}): 

\begin{prop}
\label{prop:branching}
The representation $\M_{\lambda}$ appears in the restriction to $\GL_{n-1}(\C)$ of the representation $\M_{\mu},$ i.e., 
${\rm Hom}_{\GL_{n-1}(\C)}(\M_\lambda, \M_\mu) \neq 0$  
if and only if 
$$
\mu_1 \ \geq \ \lambda_1 \ \geq \ \mu_2 \ \geq \ \lambda_2 \ \geq \ \cdots \ \geq \ \lambda_{n-1} \ \geq \ \mu_n. 
$$
In this situation, $\M_\lambda$ appears with multiplicity one in $\M_\mu.$ The conditions on the weights $\mu$ and 
$\lambda$ captured by the above inequalities will be denoted 
$\mu \succ \lambda$, and we say $\mu$ interlaces $\lambda$. 
\end{prop}

Remember that we dualized the coefficient systems. Let us restate the above proposition in the form that we will need it later on.  Given $\mu = (\mu_1,\dots,\mu_n)$ recall that $\mu^{\sf v} =  (-\mu_n,\dots,-\mu_1),$ and that the contragredient representation $\M_\mu^{\sf v}$ of $\M_\mu$ is  $\M_{\mu^{\sf v}}.$ 
Similarly, for $\lambda$ and  $\M_{\lambda}$.

\begin{cor}
\label{cor:branching}
${\rm Hom}_{\GL_{n-1}(\C)}(\M_\mu^{\sf v} \otimes \M_\lambda^{\sf v},  \C) \neq 0$  
if and only if $\mu^{\sf v} \succ \lambda$, i.e., 
$$
-\mu_n \ \geq \ \lambda_1 \ \geq \ -\mu_{n-1} \ \geq \ \lambda_2 \ \geq \ \cdots \ \geq \ \lambda_{n-1} \ \geq \ -\mu_1. 
$$
In this situation, ${\rm Hom}_{\GL_{n-1}(\C)}(\M_\mu^{\sf v} \otimes \M_\lambda^{\sf v},  \C)$ is a one-dimensional space. 
\end{cor}

\begin{proof}
${\rm Hom}_{\GL_{n-1}(\C)}(\M_\mu^{\sf v} \otimes \M_\lambda^{\sf v},  \C) = 
{\rm Hom}_{\GL_{n-1}(\C)}(\M_\mu^{\sf v}, \M_\lambda).$
\end{proof}

\smallskip

We will also write $\M_{\lambda} \hookrightarrow \M_\mu$ to say that  $\M_\lambda$ appears (with multiplicity one) 
in $\M_\mu.$
Given $\lambda$ and an $m \in \Z$, the representation $\M_{\lambda} \otimes {\rm det}^m$ corresponds to the 
weight $\lambda + m$. In global notations; given 
a weight $\lambda = (\lambda^\iota)_{\iota \in \cE_F}$, the weight $\lambda + m$ is simply 
$(\lambda^\iota+m)_{\iota \in \cE_F}.$

\medskip
\paragraph{\bf Critical set and compatibility} 

The main result of this section is the following theorem which relates the crticial set to a condition on the coefficient systems. 

\begin{thm}
\label{thm:critical-compatible}
Let $\mu \in X^+_{00}(T_n)$ and $\Pi \in {\rm Coh}(G_n, \mu^{\sf v})$.  Similarly, let $\lambda \in X^+_{00}(T_{n-1})$ and $\Sigma \in {\rm Coh}(G_{n-1}, \lambda^{\sf v}).$
Assume that there is an integer $m_0$ such that 
$$
M_{\lambda + m_0} \hookrightarrow M_\mu^{\sf v}, \quad \mbox{that is} \quad \mu^{\sf v} \succ \lambda + m_0.
$$ 
Then we have  
$$
\left\{ m \in \Z \ : \mu^{\sf v} \succ \lambda + m \right\} \ \ = \ \ 
\left\{ m \in \Z \ : \ \tfrac12 + m \ \mbox{is critical for} \  L(s, \Pi \times \Sigma) \right\}. 
$$
(In particular, the critical set is a finite set.) 
\end{thm}

\begin{proof}
Let us begin by noting that the assertion is of a purely local nature. If $m \in \Z$ then $\mu^{\sf v} \succ \lambda + m$ if and only if for every $v \in S_\infty$ we have 
$\mu_v^{\sf v} \succ \lambda_v + m.$ On the other hand, $\tfrac12 + m$ is critical for $L(s, \Pi \times \Sigma)$ if and only if both 
$L_v(s, \Pi_v \times \Sigma_v)$ and $L_v(1-s, \Pi_v^{\sf v} \times \Sigma_v^{\sf v})$ are holomorphic at $s = \tfrac12 + m$ for  every $v \in S_\infty.$
We will consider the real and complex places separately. 

\smallskip

For $v \in S_r$, this has been observed by Kasten and Schmidt \cite[Theorem 2.3]{kasten-schmidt}. 
What is remarkable is that their observation for a real place in fact goes through for a complex place also; see also 
Grobner--Harris \cite[Lem.\,4.7]{grobner-harris}. 
We very briefly summarize the proof of \cite[Theorem 2.3]{kasten-schmidt}, so that the reader may compare the similarities and the differences in the combinatorics in the real and complex cases. Fix a real place $v \in S_r$ and we drop it from the notations. 
The hypothesis $\mu^{\sf v} \succ \lambda$ means 
\begin{equation}
\label{eqn:mu>lambda-real-case} 
-\mu_n \ \geq \ \lambda_1 \ \geq \ \mu_{n-1} \ \geq \ \dots \ \geq \ \lambda_{n-1} \ \geq \ -\mu_1.
\end{equation}
If $\ell$ and $\ell'$ are the cuspidal parameters of $\mu$ and $\lambda$ (see \ref{sec:repn-glnR}) then define the cuspidal width between $\Pi$ and $\Sigma$ as: 
$$
c_{\Pi, \Sigma} \ := \ {\rm min} \{|\ell_i - \ell_j'| \ : \ 1\leq i \leq n, 1\leq j \leq n-1\}.
$$
Verify that $s = 1/2$ is critical if and only if 
\begin{equation}
\label{eqn:cuspidal-width}
-(c_{\Pi, \Sigma} - 1) \ \leq \ ({\sf w} + {\sf w}') \ \leq \ c_{\Pi, \Sigma} - 1. 
\end{equation} 
Verify that (\ref{eqn:mu>lambda-real-case}) $\implies$ (\ref{eqn:cuspidal-width}), i.e., $\mu^{\sf v} \succ \lambda$ implies that $s=1/2$ is critical. During this verification one observes that the condition $\mu^{\sf v} \succ \lambda$ in fact implies 
$$
\ell_1 \ > \ \ell'_1 \ > \ \ell_2 \ > \ \ell_2' \ > \ \cdots  \ > \ \ell_{n-1}' \ > \ \ell_n'.
$$
These inequalities are then used to verify that suppose $\mu^{\sf v} \succ \lambda+m_0$ for some $m_0$ and $\tfrac12 + m$ is critical for $L(s, \Pi \times \Sigma)$ then $\mu^{\sf v} \succ \lambda+m.$ For more details the reader is referred to \cite{kasten-schmidt}. Now we take the complex case, which is combinatorially more complicated.

\medskip

For the rest of this proof, we assume $v \in S_c$, i.e, that $v$ is a complex place.
Since we have now fixed $v \in S_c$, we drop it and use local notations as in \ref{sec:glnc}.
Recall: $\mu = (\mu^\iota, \mu^{\bar{\iota}})$ with 
$\mu^\iota = (\mu_1, \dots, \mu_n)$ and $\mu^{\bar{\iota}} = (\mu_1^* \dots, \mu_n^*);$ the purity condition gives 
$\mu_i^* + \mu_{n-i+1} = {\sf w},$ where ${\sf w}$ is the purity weight of the global $\mu$; the cuspidal parameters are:  
$$
(a_1,\dots,a_n) = \left(\mu_1+\tfrac{n-1}{2},\dots, \mu_n - \tfrac{(n-1)}{2}\right) \ \ {\rm and} \ \ 
(b_1,\dots,b_n) = \left({\sf w} - a_1, \dots, {\sf w} - a_n \right);
$$ 
the local representation 
$\Pi_v = J_\mu := \textrm{Ind}\left( z^{a_1}\bar{z}^{b_1} \otimes \dots  \otimes z^{a_n}\bar{z}^{b_n}\right).$
Note that  $a_i + b_i = {\sf w}$   and $a_i - b_i = 2\mu_i - {\sf w} + n -2i + 1.$ 
Similarly, we have $\lambda = (\lambda^\iota, \lambda^{\bar{\iota}})$ with 
$\lambda^\iota = (\lambda_1, \dots, \lambda_{n-1})$ and $\lambda^{\bar{\iota}} = (\lambda_1^* \dots, \lambda_{n-1}^*);$ purity gives 
$\lambda_j^* + \lambda_{n-j} = {\sf w}',$ where ${\sf w}' = {\sf w}(\lambda)$ is the purity weight of $\lambda$; the cuspidal parameters are:  
$$
(a_1',\dots,a_{n-1}') = \left(\lambda_1+\tfrac{n-2}{2}, \dots, \lambda_{n-1} - \tfrac{(n-2)}{2}\right) \ \ {\rm and} \ \ 
(b_1',\dots,b_{n-1}') = ({\sf w}' - a_1', \dots, {\sf w}' - a_{n-1}');
$$ 
the local representation $\Sigma_v = J_\lambda := 
\textrm{Ind}\left( z^{a_1'}\bar{z}^{b_1'}  \otimes \dots  \otimes z^{a_{n-1}'}\bar{z}^{b_{n-1}'}\right).$
Note that  $a_j' + b_j' = {\sf w}'$   and $a_j' - b_j' = 2\lambda_j' - {\sf w}' + n -2j.$ 

\smallskip

It suffices to prove the theorem for a pair $(\Pi, \Sigma)$ such that $\mu^{\sf v} \succ \lambda.$ This is because, suppose we have a pair $(\Pi_1, \Sigma_1)$ such that 
$\mu_1^{\sf v} \succ \lambda_1 + m_0$ for an integer $m_0$, then we can consider $\Pi = \Pi_1$ and $\Sigma = \Sigma_1 \otimes |\ |^{m_0};$ for these representations the weights are $\mu = \mu_1$ and $\lambda = \lambda_1 + m_0.$  It is easy to see that if the theorem holds for $(\Pi, \Sigma)$ then it holds for $(\Pi_1, \Sigma_1)$. Henceforth we assume therefore that $\mu^{\sf v} \succ \lambda.$ We have: 

\begin{lem}
\label{lem:1-2-crit}
If $\mu^{\sf v} \succ \lambda$ then $s = \tfrac12$ is critical for $L(s, \Pi \times \Sigma)$. 
\end{lem}

\begin{proof}
The hypothesis $\mu^{\sf v} \succ \lambda$, means $\mu^{\iota \sf v} \succ \lambda^\iota$ and 
$\mu^{\bar{\iota}{\sf v}} \succ \lambda^{\bar{\iota}}.$ By Cor.\,\ref{cor:branching} and the formulae for 
$\mu_i^* = {\sf w} -\mu_{n-i+1}$ and $\lambda_j^* = {\sf w}' -\lambda_{n-j}$ we have: 
\begin{equation}\label{eqn:mu>lambda}
\begin{array}{ccccccccccccc}
-\mu_n & \geq & \lambda_1 & \geq & -\mu_{n-1} & \geq & \lambda_2 & \geq & \cdots & \geq & \lambda_{n-1} & \geq & 
-\mu_1, \\
 &&&&&&&&&&&& \\
{\sf w} -\mu_n & \geq & \lambda_1- {\sf w}' & \geq & {\sf w} -\mu_{n-1} & \geq & \lambda_2 - {\sf w}' & \geq & \cdots 
& \geq & \lambda_{n-1} - {\sf w}' & \geq & {\sf w} -\mu_1. 
\end{array}
\end{equation}

Using the formulae in Knapp \cite[Sect.\,4]{knapp}, for any complex place $v$ as above, we have
$$
L_v \left(s, \Pi_v \times \Sigma_v\right) \ = \ \prod_{i,j} L\left(s, z^{a_i+a_j'}\bar{z}^{b_i+b_j'} \right) \ \sim \ 
\prod_{i,j} \Gamma\left(s + \frac{a_i+a_j'+b_i+b_j'}{2} + \frac{|a_i - b_i + a_j'- b_j'|}{2} \right), 
$$
where, by $\sim$, we mean up to a nonzero exponential factor which is irrelevant in computing the critical set. 
In terms of the data going in to $\mu$ and $\lambda$, we have: 
\begin{gather}\label{eqn:l-fn-crit-1}
L_v \left(s, \Pi_v \times \Sigma_v\right) 
= \prod_{1\leq i,j \leq n} \Gamma\left(s + \frac{{\sf w} + {\sf w}'}{2} + 
\frac{|2\mu_i + 2\lambda_j+2n-2i-2j+1-({\sf w} + {\sf w}')|}{2} \right), \\
L_v \left(s, \Pi_v^{\sf v}\times \Sigma_v^{\sf v} \right) 
= \prod_{1 \leq i,j \leq n} \Gamma\left(s - \frac{{\sf w} + {\sf w}'}{2} + 
\frac{|2\mu_i + 2\lambda_j+2n-2i-2j+1-({\sf w} + {\sf w}')|}{2} \right). \label{eqn:l-fn-crit-2}
\end{gather}
It follows that $s = 1/2$ is critical if every factor on the right hand side above is regular at $s=1/2,$ i.e., if both the inequalities 
\begin{gather}\label{eqn:crit-1}
\frac{1+ ({\sf w} + {\sf w}')+ |2\mu_i + 2\lambda_j+2n-2i-2j+1-({\sf w} + {\sf w}')|}{2} \geq 1, \\ \smallskip
\frac{1 - ({\sf w} + {\sf w}') + |2\mu_i + 2\lambda_j+2n-2i-2j+1-({\sf w} + {\sf w}')|}{2}  \geq 1.
\label{eqn:crit-2}
\end{gather}
are satisfied. Consider two cases: \\

\noindent
{\bf Case 1:} $i+j \geq n+1$. By (\ref{eqn:mu>lambda}), we have 
$\mu_i+\lambda_j \leq 0$ and $\mu_i+\lambda_j \leq {\sf w}+{\sf w}'.$ Hence
$$
|2\mu_i + 2\lambda_j+2n-2i-2j+1-({\sf w} + {\sf w}')| \ = \ 
-2\mu_i - 2\lambda_j-2n+2i+2j-1+({\sf w} + {\sf w}').
$$
For a given $j$, this quantity is minimized if $i = n+1-j.$ Hence (\ref{eqn:crit-1}) is satisfied since 
$$
\frac{1+ ({\sf w} + {\sf w}') -2\mu_{n+1-j} - 2\lambda_j + 1 + ({\sf w} + {\sf w}')}{2} \  = \  
-\mu_{n+1-j}-\lambda_j + ({\sf w} + {\sf w}') + 1 \ \geq \ 1. 
$$
Similarly, (\ref{eqn:crit-2}) is satisfied since 
$$
\frac{1- ({\sf w} + {\sf w}') -2\mu_{n+1-j} - 2\lambda_j + 1 + ({\sf w} + {\sf w}')}{2} \  = \  
1-\mu_{n+1-j}-\lambda_j \ \geq \ 1. 
$$

\noindent
{\bf Case 2:} $i+j \leq n$. By (\ref{eqn:mu>lambda}), we have 
$\mu_i+\lambda_j \geq 0$ and $\mu_i+\lambda_j \geq {\sf w}+{\sf w}'.$ Hence
$$
|2\mu_i + 2\lambda_j+2n-2i-2j+1-({\sf w} + {\sf w}')| \ = \ 
2\mu_i + 2\lambda_j+2n-2i-2j+1-({\sf w} + {\sf w}'). 
$$
For a given $j$, this quantity is minimized if $i = n-j.$ Hence (\ref{eqn:crit-1}) is satisfied since 
$$
\frac{1+ ({\sf w} + {\sf w}') +2\mu_{n-j} + 2\lambda_j + 1 - ({\sf w} + {\sf w}')}{2} \  = \  
1+\mu_{n-j} + \lambda_j \geq \ 1. 
$$
Similarly, (\ref{eqn:crit-2}) is satisfied since 
$$
\frac{1- ({\sf w} + {\sf w}') + 2\mu_{n-j} + 2\lambda_j + 1 - ({\sf w} + {\sf w}')}{2} \  = \  
1+\mu_{n-j}+\lambda_j - ({\sf w} + {\sf w}')\ \geq \ 1. 
$$
\end{proof}

\begin{cor}
\label{cor:>implies-crit}
$$
\left\{ m \in \Z \ : \mu^{\sf v} \succ \lambda + m \right\} \ \ \subset \ \ 
\left\{ m \in \Z \ : \ \tfrac12 + m \ \mbox{is critical for} \  L(s, \Pi \times \Sigma) \right\}. 
$$
\end{cor}

\begin{proof}
If $m$ is such that $\mu^{\sf v} \succ \lambda + m$ then by Lem.\,\ref{lem:1-2-crit}, $s = 1/2$ is critical for 
$L(s, \Pi \times \Sigma(m))$ where $\Sigma(m) = \Sigma \otimes |\ |^m.$ But, 
$L(\tfrac12, \Pi \times \Sigma(m)) = L(\tfrac12 +m , \Pi \times \Sigma).$ 
\end{proof}

\medskip

\begin{prop}
\label{prop:crit-glnc}
Assume $\mu^{\sf v} \succ \lambda.$ Define the numbers: 
\begin{eqnarray*}
m^+_{\mu,\lambda} & := & 
{\rm min}\left\{ \{-\mu_{n+1-j} - \lambda_j : 1 \leq j \leq n-1\} \ \cup\  
\{\mu_{n-j} + \lambda_j - ({\sf w} + {\sf w}') :  1 \leq j \leq n-1 \}  \right\},  \\ 
m^-_{\mu,\lambda} & := & 
{\rm max}\left\{ \{-\mu_{n-j} - \lambda_j : 1 \leq j \leq n-1\} \ \cup \ 
\{\mu_{n+1-j} + \lambda_j - ({\sf w} + {\sf w}') :  1 \leq j \leq n-1 \}  \right\}.
\end{eqnarray*} 
Observe that $m^-_{\mu,\lambda} \leq 0 \leq m^+_{\mu,\lambda}.$ We have 
\begin{enumerate}
\item $\left\{ m \in \Z \ : \mu^{\sf v} \succ \lambda + m \right\} = \left\{ m \in \Z \ :  m^-_{\mu,\lambda} \leq m \leq m^+_{\mu,\lambda} \right\}$; and 

\item $\left\{ m \in \Z \ : \ \tfrac12 + m \ \mbox{is critical for} \  L(s, \Pi \times \Sigma) \right\} = \left\{ m \in \Z \ :  m^-_{\mu,\lambda} \leq m \leq m^+_{\mu,\lambda} \right\}$.
\end{enumerate}
\end{prop}

\begin{proof}
That $m^-_{\mu,\lambda} \leq 0 \leq m^+_{\mu,\lambda}$ follows  from (\ref{eqn:mu>lambda}).  To prove (1), suppose $\mu^{\sf v} \succ \lambda + m$, then applying (\ref{eqn:mu>lambda}) to the case of 
$\lambda + m$ we get: 
{\small
$$\label{eqn:mu>lambda+m}
\begin{array}{ccccccccccccc}
-\mu_n & \geq & \lambda_1+m & \geq & -\mu_{n-1} & \geq & \lambda_2 +m& \geq & \cdots & \geq & 
\lambda_{n-1} + m& \geq & 
-\mu_1, \\
 &&&&&&&&&&&& \\
{\sf w} -\mu_n & \geq & \lambda_1- {\sf w}' -m& \geq & {\sf w} -\mu_{n-1} & \geq & \lambda_2 - {\sf w}' -m & \geq & 
\cdots & \geq & \lambda_{n-1} - {\sf w}' -m & \geq & {\sf w} -\mu_1. 
\end{array}
$$}
From these inequalities, we deduce for all $1 \leq j \leq n-1$: 
\begin{equation}
\begin{array}{ccccc}
\mu_{n+1-j}+\lambda_j - {\sf w}-{\sf w}' & \leq & m & \leq & -\mu_{n+1-j} - \lambda_j,   \\
& & & & \\
-\mu_{n-j} - \lambda_j & \leq & m & \leq & \mu_{n-j}+\lambda_j - {\sf w}-{\sf w}'. 
\end{array}
\end{equation}
Hence $m^-_{\mu,\lambda} \leq m \leq m^+_{\mu,\lambda}$. This entire chain of reasoning is reversible, which proves (1). 

\smallskip

To prove (2), from Cor.\,\ref{cor:>implies-crit} and (1), it suffices to show that if $\tfrac12 + m$ is critical then the integer 
$m$ satisfies: $m^-_{\mu,\lambda} \leq m \leq m^+_{\mu,\lambda}$. Suppose then that $\tfrac12 + m$ is critical, then 
from (\ref{eqn:l-fn-crit-1}) and (\ref{eqn:l-fn-crit-2}) it follows that 
\begin{gather}\label{eqn:crit-1/2+m-1}
m+ \frac{1+ ({\sf w} + {\sf w}')+ |2\mu_i + 2\lambda_j+2n-2i-2j+1-({\sf w} + {\sf w}')|}{2} \geq 1, \\ \smallskip
-m+\frac{1 - ({\sf w} + {\sf w}') + |2\mu_i + 2\lambda_j+2n-2i-2j+1-({\sf w} + {\sf w}')|}{2}  \geq 1.
\label{eqn:crit-1/2+m-2}
\end{gather}
As in the proof of Lem.\,\ref{lem:1-2-crit} consider two cases: \\

\noindent
{\bf Case 1:} $i+j \geq n+1$. The hypothesis $\mu^{\sf v} \succ \lambda$ gives the inequalities 
(\ref{eqn:mu>lambda}) from which we get as before 
$\mu_i+\lambda_j \leq 0$ and $\mu_i+\lambda_j \leq {\sf w}+{\sf w}'.$ Hence
$$
|2\mu_i + 2\lambda_j+2n-2i-2j+1-({\sf w} + {\sf w}')| \ = \ 
-2\mu_i - 2\lambda_j-2n+2i+2j-1+({\sf w} + {\sf w}').
$$
In this case, (\ref{eqn:crit-1/2+m-1}) holds if and only if $m \geq \mu_{n+1-j} + \lambda_j - ({\sf w}+{\sf w}')$, and 
(\ref{eqn:crit-1/2+m-2}) holds if and only if $m \leq -\mu_{n+1-j}-\lambda_j.$ 

\medskip

\noindent
{\bf Case 2:} $i+j \leq n$. In this case, exactly as above, 
(\ref{eqn:crit-1/2+m-1}) and (\ref{eqn:crit-1/2+m-2}) hold if and only if 
$-\mu_{n-j}-\lambda_j \leq m \leq \mu_{n-j}+\lambda_j-({\sf w}+{\sf w}')$. 

\medskip

Putting both cases together, we see that $\tfrac12+m$ is critical implies that 
$m^-_{\mu,\lambda} \leq m \leq m^+_{\mu,\lambda},$ which concludes the proof of the proposition.   
\end{proof}
This also concludes the proof of Thm.\,\ref{thm:critical-compatible}.
\end{proof}

Let's revert to global notations, and for future reference, record the set of all critical points: 
\begin{cor} 
Let 
$\mu \in X^+_{00}(T_n)$ (resp., $\lambda \in X^+_{00}(T_{n-1})$) and $\Pi \in {\rm Coh}(G_n, \mu^{\sf v})$  
(resp., $\Sigma \in {\rm Coh}(G_{n-1}, \lambda^{\sf v})$). Assume $\mu^{\sf v} \succ \lambda.$
Define the numbers: 
\begin{enumerate}
\item For $v \in S_r$ which corresponds to $\iota_v \in \cE_F$: 
  \begin{itemize}
   \item $m^+_{\mu^{\iota_v},\lambda^{\iota_v}} = 
  {\rm min}_{1 \leq j \leq n-1} \left\{-\mu_{n+1-j}^{\iota_v} - \lambda_j^{\iota_v}\right\}$; 
  
  \item $m^-_{\mu^{\iota_v},\lambda^{\iota_v}} := 
   {\rm max}_{1 \leq j \leq n-1} \left\{-\mu_{n-j}^{\iota_v} - \lambda_j^{\iota_v}  \right\}$. 
 \end{itemize}

\item For $v \in S_c$ which corresponds to a pair of embeddings $\{\iota_v \bar{\iota}_v\}$: 
  
  \begin{itemize}
   \item $m^+_{\mu^{\iota_v},\lambda^{\iota_v}} = 
  {\rm min} \left\{ \{-\mu_{n+1-j}^{\iota_v} - \lambda_j^{\iota_v} \}_{1 \leq j \leq n-1} \ \cup\  
  \{\mu_{n-j}^{\iota_v} + \lambda_j^{\iota_v} - ({\sf w} + {\sf w}')  \}_{1 \leq j \leq n-1} \right\}$; 
  
  \item $m^-_{\mu^{\iota_v},\lambda^{\iota_v}} := 
   {\rm max} \left\{ \{-\mu_{n-j}^{\iota_v} - \lambda_j^{\iota_v}\}_{1 \leq j \leq n-1} \ \cup \ 
  \{\mu_{n+1-j}^{\iota_v} + \lambda_j^{\iota_v} - ({\sf w} + {\sf w}') \}_{1 \leq j \leq n-1}  \right\}$. 
  \end{itemize}

\item Now define
  \begin{itemize}
  \item $m^+(\mu,\lambda) := {\rm min}\{v \in S_\infty : m^+_{\mu^{\iota_v},\lambda^{\iota_v}} \}$; 
  \item $m^-(\mu,\lambda) := {\rm max}\{v \in S_\infty : m^-_{\mu^{\iota_v},\lambda^{\iota_v}} \}$. 
\end{itemize}
\end{enumerate}
Then, $\tfrac12 + m$ is critical for $L(s, \Pi \times \Sigma)$ if and only if 
$m^-(\mu,\lambda) \leq m \leq m^+(\mu,\lambda).$ 
\end{cor}

\smallskip
\paragraph{\bf Consequences of compatibility of coefficient systems - I} 

Let us consider the hypothesis that there is an integer $m_0$ such that 
$M_{\lambda + m_0} \hookrightarrow M_\mu^{\sf v},$ in the statement of Thm.\,\ref{thm:critical-compatible}. 
Suppose $F = \Q.$
Observe that this hypothesis always satisfied for $\GL_2 \times \GL_1$; take $\lambda_1 = -\mu_2$. But for $n \geq 3$ without that hypothesis, the conclusion of  
Theorem~\ref{thm:critical-compatible} need not hold. For example, consider $\GL_3 \times \GL_2$ over $\Q$. Take $\mu = (0,0,0)$; then 
${\sf w} = 0$; $\ell = (2,0,-2)$. Take $\lambda = (1,-1)$; then ${\sf w}' = 0$; $\ell' = (3,-3).$ Then $C_{\Pi, \Sigma} = 1;$ hence $\tfrac12$ is critical since 
$1-  C_{\Pi, \Sigma} \leq {\sf w} + {\sf w}' \leq C_{\Pi, \Sigma} - 1.$ However, there is no integer $m$ such that $M_{\lambda + m} \hookrightarrow M_\mu^{\sf v}.$ In other words, there are Rankin--Selberg $L$-funcitons 
$L(s, \Pi \times \Sigma)$ which have critical values, but the underlying sheaves are not compatible; in such situations we are unable to say anything about these critical values.
Furthermore, the methods of this paper say that when the sheaves are compatible, 
we can prove a theorem for every critical value, and only for critical values, but we are unable to say anything about noncritical values.

\smallskip
\paragraph{\bf Consequences of compatibility of coefficient systems - II} \label{sec:subtlety} 
When the base field $F$ has a complex place, there is yet another consequence of the hypothesis that 
$\mu^{\sf v} \succ \lambda + m.$ This issue is already seen by working with $\GL_2 \times \GL_1.$
A consequence of the discussion to follow will be that our Thm.\,\ref{thm:main} in the case $\GL_2 \times \GL_1$ is exactly Hida's \cite[Thm.\,I]{hida-duke} for critical values of type $\emptyset.$ We will follow Hida~\cite{hida-duke} for most of \ref{sec:subtlety}. 

Suppose $M$ is a pure motive of rank $2$ over $F$ with coefficients in a field $E.$ 
Recall that $\cE_F = \Hom(F,\C)$. Let's write $\cE_r = \Hom(F,\R) \subset \cE_F$ for the set of real embeddings, and 
$\cE_c = \cE \setminus \cE_r$ for the set of complex embeddings. The Hodge types of $M$ are of the form 
$$
\cH_\iota(M) = \{(p_\iota, q_{\bar\iota}), (q_\iota, p_{\bar\iota})\}, \quad \forall \iota \in \cE_F.
$$
The set of all Hodge types is denoted $\cH(M) = \cup_\iota \cH_\iota(M).$ 
Make the following assumptions: 
\begin{itemize}
\item (Purity) There exists $w$ such that $p_{\bar\iota} + q_\iota = q_{\bar\iota} + p_\iota = w, \forall \iota \in \cE_F.$ 
\item (No middle Hodge type) $(p,q) \in \cH(M) \Longrightarrow p \neq q.$ 
\item (Regularity) $p_\iota \neq q_\iota$ for all $\iota \in \cE_F,$ hence all nonzero Hodge numbers are $1.$ Without loss of generality we may also assume that $p_\iota > q_\iota$ for all $\iota.$
\end{itemize}
 The motive conjecturally attached to $\pi \in \Coh(\GL_2/F, \mu)$ satisfies the above properties; this motive is proven to exist by Scholl \cite{scholl} for $F = \Q$ and by Blasius--Rogawski \cite{blasius-rogawski}  for $F$ totally real. 

Now let $\chi$ be an algebraic Hecke character of $F$, and suppose the infinity type of $\chi$ is written as 
$\infty(\chi) = j = (j_\iota)_{\iota}.$ Let $M(\chi)$ be the rank $1$ motive attached to $\chi.$ 
The Hodge types of $\bM = M \otimes M(\chi)$ are of the form: 
$$
\cH_\iota(\bM) = \{(p_\iota - j_\iota, q_\iota-j_\iota), (q_\iota-j_\iota, p_\iota-j_\iota)\}, \quad \forall \iota \in \cE_F.
$$ 
Recall the $\Gamma$-factors for $\bM$, defined as: 
$$
L_\infty(s, \bM) = \prod_\iota L_\iota(s, \bM), \ \ {\rm where} \ \ L_\iota(s, \bM) = 
\prod_{(p,q) \in \cH_\iota(\bM), \, p < q} \Gamma_\C(s-p), 
$$
and $\Gamma_\C(s) = 2(2\pi)^{-s}\Gamma(s).$ For any subset $S$ of $\cE_F$, for example $S$ could be 
$\{\iota, \bar\iota\}$ for a fixed complex embedding $\iota$, we let 
$L_S(s, \bM) = \prod_{\iota \in S} L_\iota(s, \bM).$ Let $\bM^{\sf v}$ be the dual motive; note that 
$(p,q) \in \cH(\bM) \iff (-p,-q) \in \cH(\bM^{\sf v}).$ We say $\bM$ is critical if both $L_\infty(s, \bM)$ and 
$L_\infty(1-s, \bM^{\sf v})$ are finite at $s=0.$ 
Suppose, $\bM = M \otimes M(\chi)$ is critical, then we deduce: 
\begin{enumerate}
\item $\iota \in \cE_r$, $q_\iota+1 \leq j_\iota \leq p_\iota.$
\item $\iota \in \cE_c$, then one and only one of the following three cases are possible: 
   \begin{enumerate}
   \item $p_\iota-j_\iota < q_{\bar\iota}-j_{\bar\iota}.$ Then we also have $q_\iota-j_\iota < p_{\bar\iota}-j_{\bar\iota}.$ 
   Looking at the contribution from $L_{\{\iota, \bar\iota\}}(s, \bM)$ towards criticality, we get the inequalities 
   $j_\iota \geq p_\iota+1$ and $j_{\bar\iota} \leq q_{\bar\iota}.$ 
   \item $p_{\bar\iota} -j_{\bar\iota} < q_\iota-j_\iota.$ This is similar to (a) with $\bar\iota$ instead of $\iota.$ 
   \item $p_\iota-j_\iota > q_{\bar\iota}-j_{\bar\iota}$ and $p_{\bar\iota} -j_{\bar\iota} > q_\iota-j_\iota$. In this case we see 
   that criticality gives $q_\iota+1 \leq j_\iota \leq p_\iota,$ which is similar to the case of a real embedding.  
   \end{enumerate}
\end{enumerate}
Therefore, if $\bM$ is critical then there exist sets $A$ and $T$ with $A \subset \cE_c$ and $\cE_r \subset T$ such that 
$\cE_F = T \cup A \cup \bar A $ as a disjoint union; furthermore, 
$$
\iota \in T \iff q_\iota+1 \leq j_\iota \leq p_\iota, \quad \quad
\iota \in A \iff j_\iota \geq p_\iota+1 \ {\rm and} \ j_{\bar\iota} \leq q_{\bar\iota}.
$$
In this situation, following Hida, we say {\it $\bM$ is critical of type $A$}, or that $s=0$ is critical of type $A$ for 
$L(s, \bM).$ The nature of the algebraicity results for a critical value depends on this set $A$; compare the two different parts of Hida~\cite[Theorem I]{hida-duke}. 
We now show that in our situation, the compatibility of coefficient systems implies that we are looking at critical points of type $A = \emptyset,$ the empty set. In the proposition below, $M(1) = M \otimes \Q(1)$ denotes a Tate twist, which is a harmless artifice resulting from the vagaries of the so-called motivic normalization. 

\begin{prop}
Let $\pi \in \Coh(G_2, \mu^{\sf v})$ and $\chi \in \Coh(G_1, \lambda^{\sf v}).$ 
Let $M(\pi)$ (resp., $M(\chi)$) be the motive attached to $\pi$ (resp., $\chi$). 
Suppose $\mu^{\sf v} \succ \lambda,$ then $(M(\pi) \otimes M(\chi))(1)$ is critical of type $\emptyset.$ 
\end{prop}

\begin{proof}
Suppose $\mu = (\mu^\iota)$ with $\mu^\iota = (\mu^\iota_1, \mu^\iota_2)$, then by comparing our coefficient system
$\M_\mu^{\sf v}$ with the coefficient system $L(\kappa; \C)$ in Hida~\cite[p.\,433]{hida-duke} and from the recipe for Hodge types of $M(\pi)$ in terms of the coefficient system as in \cite[Conj.\,0.1, (i)]{hida-duke}, we see that
$p_\iota = -\mu^\iota_2 + 1, q_\iota = -\mu^\iota_1.$
Also, the infinity type of $\chi$ is given by $j := \infty(\chi) = \lambda.$ Applying the preceding discussion, 
and the well-known fact about Tate twists: if $(p,q) \in \cH(M)$ then $(p-1,q-1) \in \cH(M(1)),$
we see that the Hodge types of $\bM := (M(\pi) \otimes M(\chi))(1)$ are given by
$
\cH_\iota(\bM) = \{(-\mu^\iota_2 -\lambda^\iota, -\mu^{\bar\iota}_1-\lambda^{\bar\iota}-1), 
(-\mu^\iota_1-\lambda^\iota-1, -\mu^{\bar\iota}_2-\lambda^{\bar\iota}) \}.
$
Compatibility of the coefficient systems ($\mu^{\sf v} \succ \lambda$) translates to:  
$
-\mu_2^\iota \geq \lambda^\iota \geq -\mu^\iota_1, \quad \forall \iota \in \cE_F. 
$
It is easy to see that these inequalities preclude the possibilities in (2)(a) and (2)(b) above. For example, for (2)(a), we will need $-\mu^\iota_2 -\lambda^\iota < -\mu^{\bar\iota}_1-\lambda^{\bar\iota}-1$, but from the above inequalities under compatibility we will then have: $0 \leq -\mu^\iota_2 -\lambda^\iota < -\mu^{\bar\iota}_1-\lambda^{\bar\iota}-1 \leq -1,$
which is absurd. Similarly, (2)(b) is not possible. It is easy now to see that $\bM$ is critical of type $\emptyset.$
\end{proof}

\medskip
\subsection{\bf The proof of Theorem~\ref{thm:main}}

\smallskip
\subsubsection{\bf The main idea behind a cohomological interpretation of $L(\tfrac12, \Pi \times \Sigma)$}
\label{sec:main-idea}
We interpret the Rankin--Selberg integral $I(\tfrac12, \phi_{\Pi}, \phi_{\Sigma})$ 
in terms of Poincar\'e duality. More precisely, the vector $w_{\Pi_f}$ will correspond to a cohomology class 
$\vartheta_{\Pi}$ in degree $b_n^F$, and similarly $w_{\Sigma_f}$ will correspond to a class $\vartheta_{\Sigma}$ in degree $b_{n-1}^F$. These classes, after dividing by certain periods, have good rationality properties. 
Pull back $\iota^*\vartheta_{\Pi}$ along 
the proper map $\iota: \tilde{S}^{G_{n-1}} \to S^{G_n}$, and wedge (or cup) with $\vartheta_{\Sigma}$, to give a top degree class (by Prop.\,\ref{prop:numerical}) on $\tilde{S}^{G_{n-1}}$ with coefficients in a tensor product sheaf. Now if the constituent sheaves are compatible ($\mu^{\sf v} \succ \lambda$), which by \ref{sec:critical-compatible} is the same as saying $s=1/2$ is critical, then we get a top-degree class on $\tilde{S}^{G_{n-1}}$ with constant coefficients. 
Apply Poincar\'e duality, i.e., fix an orientation on $\tilde{S}^{G_{n-1}}$ and integrate. One realizes then that this is 
essentially the Rankin--Selberg integral in Prop.\,\ref{prop:rankin-selberg}. Interpreting that integral, and hence the critical $L$-value $L_f(\tfrac12, \Pi \times \Sigma)$ in cohomology, permits us to study it's arithmetic properties, since 
Poincar\'e duality is Galois equivariant. All this may be summarized in the diagram: 

{\small
$$
\xymatrix{
\Whit(\Pi_f) \times \Whit(\Sigma_f) \ar[rr] \ar[dd]&   &   
H^{b_n^F}(\g_n, K_{n,\infty}^0; V_\Pi \otimes  \M_{\mu,\C}^{\sf v})(\epsilon) \times 
H^{b_{n-1}^F}(\g_{n-1}, K_{n-1,\infty}^0; V_\Sigma \otimes  \M_{\lambda,\C}^{\sf v})(\eta) 
\ar@{^{(}->}[d] \\ 
&  & H^{b_n^F}_{\rm cusp}(S^{G_n}, \tM_{\mu,\C}^{\sf v}) \times 
H^{b_{n-1}^F}_{\rm cusp}(S^{G_{n-1}}, \tM_{\lambda,\C}^{\sf v}) 
\ar@{^{(}->}[d] \\
\C &  & H^{b_n^F}_!(S^{G_n}, \tM_{\mu,\C}^{\sf v}) \times 
H^{b_{n-1}^F}_!(S^{G_{n-1}}, \tM_{\lambda,\C}^{\sf v}) 
\ar[d]^{\iota^* \times \phi^*} \\
&  & H^{b_n^F}_!(\tilde{S}^{G_{n-1}},  \iota^*\tM_{\mu,\C}^{\sf v}) \times 
H^{b_{n-1}^F}(\tilde{S}^{G_{n-1}}, \phi^*\tM_{\lambda,\C}^{\sf v}) 
\ar[d]^{ \wedge} \\
H^{{\rm top}}_!(\tilde{S}^{G_{n-1}}, \C) \ar[uu]_{\int}
& & H^{d_{n-1}^F}_!(\tilde{S}^{G_{n-1}}, \iota^*\tM_{\mu,\C}^{\sf v} \otimes  \phi^*\tM_{\lambda,\C}^{\sf v}) \ar[ll]_{\T^*}
}
$$}

\subsubsection{\bf Review of certain periods and period relations} 
We briefly review the definition of certain periods attached to $\Pi \in {\rm Coh}(G_n, \mu)$, and their behavior 
upon twisting by $\Pi$ by algebraic Hecke characters. The reader is referred to my paper with Shahidi \cite{raghuram-shahidi-imrn} for all the details.

\medskip
\paragraph{\bf Comparing Whittaker models and cohomological representations}
\label{sec:comparison}
Given any $\Pi \in {\rm Coh}(G_n, \mu^{\sf v})$ and a permissible signature $\epsilon$ for $\Pi$, 
fix  a generator $[\Pi_{\infty}]^{\epsilon}$ of the one-dimensional space 
$H^{b_n^F}(\g_n ,K_{n, \infty}^0; \Whit(\Pi_{\infty}) \otimes \M_{\mu,\C}^{\sf v})(\epsilon).$
We have the following comparison isomorphism: 
\begin{equation*}
\mathcal{F}^{\epsilon}_{\Pi} \ : \ \Whit(\Pi_f) \ \longrightarrow \ 
H^{b_n^F}(\g_n ,K_{n, \infty}^0; V_{\Pi} \otimes \M_{\mu,\C}^{\sf v})(\epsilon). 
\end{equation*}
See \cite[Sect.\,3.3]{raghuram-shahidi-imrn}. 

\medskip
\paragraph{\bf Rationality fields and rational structures} 
Given $\mu \in X^+_{00}(T_n)$ and $\Pi \in \Coh(G_n, \mu^{\sf v})$, define the number field 
$\Q(\mu)$ as in \ref{sec:fin-dim-repn}, and the rationality field $\Q(\Pi_f)$ as in \cite[Sect.\,7.1]{grobner-raghuram-ijnt}. 
Now define $\Q(\Pi)$ as the compositum of $\Q(\mu)$ and $\Q(\Pi_f)$. 
The Whittaker model $\Whit(\Pi_f)$ of the finite part of $\Pi$ admits a $\Q(\Pi_f)$-structure, and hence a 
$\Q(\Pi)$-structure; see \cite[Sect.\,3.2]{raghuram-shahidi-imrn}. This rational structure is generated by normalized new vectors. 
The cohomological model $H^{b_n^F}(\g_n ,K_{n, \infty}^0; V_{\Pi} \otimes \M_{\mu,\C}^{\sf v})(\epsilon)$ admits a 
$\Q(\Pi)$-structure arising from purely geometric considerations, since the sheaf $\tM_{\mu, \C}$ has a 
$\Q(\mu)$-structure $\tM_{\mu, \Q(\mu)};$ see \cite[Sect.\,3.3]{raghuram-shahidi-imrn}.

\medskip
\paragraph{\bf Definition of the periods}
The isomorphism $\mathcal{F}^{\epsilon}_{\Pi}$ need not preserve rational structures on either side. Each side is an irreducible representation space for the action of $G_n(\A_f)$ and rational structures being unique up to homotheties, we can adjust the isomorphism $\mathcal{F}^{\epsilon}_{\Pi}$ by a scalar--which is the period--so as to preserve rational structures. There is a nonzero complex number
$p^{\epsilon}(\Pi)$ attached to the datum $(\Pi_f, \epsilon, [\Pi_{\infty}]^{\epsilon})$ such that the normalized map
$
\mathcal{F}^{\epsilon}_{\Pi,0}:= 
p^{\epsilon}(\Pi)^{-1}
\mathcal{F}^{\epsilon}_{\Pi}$
is ${\rm Aut}({\mathbb C})$-equivariant, i.e., the following diagram commutes:
\begin{equation}
\label{eqn:period-defn-diag}
\xymatrix{
\Whit(\Pi_f) \ar[rrr]^-{\mathcal{F}^{\epsilon}_{\Pi,0}} \ar[d]_{\sigma} & & &
H^{b_n^F}(\g_n ,K_{n, \infty}^0; V_{\Pi}\otimes \M_{\mu,\C}^{\sf v})(\epsilon)
\ar[d]^{\sigma} \\
\Whit({}^{\sigma}\Pi_f) 
\ar[rrr]^-{\mathcal{F}^{{}^{\sigma}\!\epsilon}_{{}^{\sigma}\Pi,0}} 
& & &
H^{b_n^F}(\g_n ,K_{n, \infty}^0; V_{{}^{\sigma}\Pi}\otimes \M_{{}^{\sigma}\!\mu, \C}^{\sf v})
(\epsilon)
}
\end{equation}
The complex number $p^{\epsilon}(\Pi)$ is well-defined up to multiplication by elements of ${\mathbb Q}(\Pi)^\times$. 
The collection $\{p^{\epsilon}({}^\sigma\Pi) : \sigma \in \autc \}$ is well-defined in 
$(\Q(\Pi) \otimes \C)^\times/\Q(\Pi)^\times.$ In terms of the un-normalized maps, we can write the above commutative diagram as
\begin{equation}
\label{eqn:unnormalized}
\sigma \circ \mathcal{F}_{\Pi}^{\epsilon} = 
\left(\frac{\sigma(p^{\epsilon}(\Pi))}
{p^{\epsilon}({}^{\sigma}\Pi)} \right) 
\mathcal{F}_{{}^{\sigma}\!\Pi}^{\epsilon} \circ \sigma.
\end{equation}
See \cite[Def./Prop.\,3.3]{raghuram-shahidi-imrn}.

\medskip
\paragraph{\bf Behaviour of the periods under twisting}\label{sec:period-relaitons}
Let $\xi$ be an algebraic Hecke character with signature $\epsilon_\xi$ as in \ref{sec:algebraic-hecke}, 
and Gau\ss~sum $\G(\xi_f)$ as in \ref{sec:gauss}. Then \cite[Thm.\,4.1]{raghuram-shahidi-imrn} says that
\begin{equation}
\label{eqn:period-reln}
\sigma\left(
\frac{p^{\epsilon \cdot \epsilon_{\xi}}(\Pi_f\otimes\xi_f)}
{\mathcal{G}(\xi_f)^{n(n-1)/2}\,p^{\epsilon}(\Pi_f)} \right) 
\ = \ 
\left(\frac{p^{\epsilon\cdot\epsilon_{\xi}}({}^{\sigma}\Pi_f\otimes {}^{\sigma}\!\xi_f)}
{\mathcal{G}({}^{\sigma}\!\xi_f)^{n(n-1)/2}\,p^{\epsilon}({}^{\sigma}\Pi_f)} \right). 
\end{equation}

\medskip
\paragraph{\bf Remarks on \cite{raghuram-shahidi-imrn}}
\label{sec:errata}
We take this opportunity to correct a small mistake in the definition of periods given in 
\cite[Def./Prop.\,3.3]{raghuram-shahidi-imrn}, which is that 
the right vertical arrow in 
(\ref{eqn:period-defn-diag}) above was in \cite{raghuram-shahidi-imrn} erroneously taken as mapping into the ${}^\sigma\!\epsilon$-isotypic component; this arrow is exactly the $T_\sigma^\bullet$ in (\ref{eqn:t-sigma-m-x-repn}) which maps 
$\Pi_f \otimes \epsilon$ to ${}^\sigma\Pi_f \otimes \epsilon.$ It should be noted that, in general, ${}^\sigma\!\epsilon$ is not even defined, since $\sigma$ need not stabilize the set of real embeddings of $F$. 

\smallskip

Next, in \cite[Thm.\,4.1]{raghuram-shahidi-imrn}, we used the signature $\epsilon_{{}^\sigma\!\xi}.$ This is fine, but is potentially misleading, unless the reader realizes that when the base field $F$ has a real place (it is only then that we have these signatures to worry about), for any algebraic Hecke character such as $\xi$, we have  
$$
\epsilon_\xi \ = \ \epsilon_{{}^\sigma\!\xi}.
$$ 
This may be seen by showing, exactly as in Gan--Raghuram~\cite[(3.7)]{gan-raghuram}, that for any archimedean place $v$ we have $({}^\sigma\!\xi)_v = \xi_v$; in {\it loc.\,cit.\,}the base field was totally real, but the same argument works, 
if $F$ has at least one real place: by purity (see \ref{sec:algebraic-hecke}) 
we have $\xi = |\ |^{\sf w}\cdot \xi^0$ for an integer ${\sf w}$ 
and a finite order character $\xi^0$, which implies that 
${}^\sigma\!\xi = |\ |^{\sf w} \cdot \sigma \circ \xi^0.$ Since $\xi^0_v$ is a finite order character of $\R^\times$ or 
$\C^\times$, we conclude that it is quadratic, hence $\sigma$-invariant, whence 
$({}^\sigma\!\xi)_v = \xi_v.$

\medskip
\subsubsection{\bf The cohomology classes and the global pairing}


\paragraph{\bf The cohomology classes}
Recall our choice of Whittaker vectors $w_\Pi = w_{\Pi_f} \otimes w_{\Pi_\infty} \in \Whit(\Pi, \psi)$ and 
$w_\Sigma = w_{\Sigma_f} \otimes w_{\Sigma_\infty} \in \Whit(\Sigma, \psi^{-1})$ from \ref{sec:whittaker}. Define
\begin{equation}
\label{eqn:coh-classes}
\vartheta^\epsilon_{\Pi} \ := \  \mathcal{F}_{\Pi}^{\epsilon}(w_{\Pi_f}) \quad {\rm and} \quad 
\vartheta^\eta_{\Sigma} \ := \  \mathcal{F}_{\Sigma}^{\eta}(w_{\Sigma_f}). 
\end{equation}
These classes are transcendental, whereas the normalized classes 
\begin{equation}
\label{eqn:coh-classes-0}
\vartheta^\epsilon_{\Pi, 0} \ := \  \mathcal{F}_{\Pi,0}^{\epsilon}(w_{\Pi_f}) = p^{\epsilon}(\Pi_f)^{-1}\vartheta^\epsilon_{\Pi}
\quad {\rm and} \quad 
\vartheta^\eta_{\Sigma, 0} \ := \  \mathcal{F}_{\Sigma,0}^{\eta}(w_{\Sigma_f}) = 
p^\eta(\Sigma_f)^{-1}\vartheta^\eta_{\Sigma}. 
\end{equation}
are rational, i.e., 
$\vartheta^\epsilon_{\Pi, 0} \in H^{b_n^F}_!(S^{G_n}, \tM_{\mu}^{\sf v})$ and 
$\vartheta^\eta_{\Sigma, 0} \in H^{b_{n-1}^F}_!(S^{G_{n-1}}, \tM_{\lambda}^{\sf v}).$

\medskip
\paragraph{\bf Compatibility of the sheaves and the map $\cT$.}
The hypothesis  (\ref{eqn:comp}) on the weights $\mu$ and $\lambda$, implies via 
Thm.\,\ref{thm:critical-compatible} that the critical set is non-empty. In particular, $s=1/2$ is critical, and 
hence $\mu^{\sf v} \succ \lambda.$ Apply the branching rule Cor.\,\ref{cor:branching} at every archimdean place; for $v \in S_\infty$, fix a nonzero 
$
\cT_v \ \in \ {\rm Hom}_{\GL_{n-1}(F_v)}(\M_{\mu_v}^{\sf v} \otimes \M_{\lambda_v}^{\sf v}, 1\!\!1), 
$
where $1\!\!1$ is the trivial representation. Define
\begin{equation}
\otimes_{v \in S_\infty} \cT_v \ =: \ \cT  \in 
{\rm Hom}_{G_{n-1}}(\M_{\mu}^{\sf v} \otimes \M_{\lambda}^{\sf v}, 1\!\!1).
\end{equation}

\medskip
\paragraph{\bf The orientation class} 
Consider $\tilde{S}^{G_{n-1}}_{R_f},$ for an open-compact subgroup $R_f \in G_{n-1}(\A_f),$ as defined in 
\ref{sec:numerical}. From the description in \ref{sec:locally-symmetric}, the connected components of 
$\tilde{S}^{G_{n-1}}_{R_f}$ are of the form $\Gamma_i \backslash G_{n-1}(\R)^0 / C_{n-1,\infty}^0$. 
We fix an orientation on $G_{n-1}(\R)^0 / C_{n-1,\infty}^0$, push it down to all such connected components, and then take the sum over the index $i$ (as in $\Gamma_i$) to get the orientation class on $\tilde{S}^{G_{n-1}}_{R_f}$; 
which we denote as $[\tilde{S}^{G_{n-1}}_{R_f}]$. For each $v \in S_r$, define $\delta_v \in \pi_0(K_{n-1,\infty})$ to be the element which is trivial at all places other than 
$v$, and at $v$ it is $\delta_n = {\rm diag}(-1,1,1,\dots,1)$; see \ref{sec:group-infinity}. 
Then $\pi_0(K_{n-1,\infty})$ is generated by $\{\delta_v : v \in S_r\}.$

\begin{lem}
\label{lem:delta-action}
The group $\pi_0(K_{n-1,\infty})$ acts on the orientation class $[\tilde{S}^{G_{n-1}}_{R_f}]$ via: 
$$
\delta_v^* \cdot [\tilde{S}^{G_{n-1}}_{R_f}] \ = \ (-1)^n \, [\tilde{S}^{G_{n-1}}_{R_f}], \quad \forall v \in S_r.
$$
\end{lem}

\begin{proof}[Brief sketch of proof]
Let $\{X_1,\dots,X_m\}$ be an ordered basis for $\gl_{n-1}/\so(n-1)$ consisting of:
$$
\{H_i = E_{ii} : 1 \leq i \leq n-1 \} \cup \{X_{ij} = E_{ij}+E_{ji} : 1 \leq i < j \leq n-1\}, 
$$
where $E_{ij}$ is the matrix with $1$ in the $(i,j)$-th entry and $0$ elsewhere. Then  
$\delta = {\rm diag}(-1,1,1,\dots,1)$ fixes each $H_i$ and each $X_{ij}$ with $i \neq 1$, and takes 
$X_{1j}$ to $-X_{1j}.$  Hence 
$\delta^*(X_1\wedge \cdots \wedge X_m) = (-1)^{n-2}(X_1\wedge \cdots \wedge X_m).$ The proof consists in applying this remark to $\delta$ as a given $\delta_v$, which has no effect on the basis vectors coming from places other than $v.$
(See also \cite[Lem.\,3.4]{raghuram-imrn}.)
\end{proof}

\medskip
\paragraph{\bf The global pairing}
Referring to the diagram in \ref{sec:main-idea}, we wish to compute the pairing
\begin{equation}
\label{eqn:global-pairing}
\langle \vartheta^\epsilon_{\Pi}, \vartheta^\eta_{\Sigma} \rangle \ := \ 
\int_{[\tilde{S}^{G_{n-1}}_{R_f}]} 
\cT^*\left(\iota^*\vartheta^\epsilon_{\Pi} \, \wedge \, \phi^*\vartheta^\eta_{\Sigma} \right), 
\quad {\rm or} \quad 
\langle \vartheta^\epsilon_{\Pi,0}, \vartheta^\eta_{\Sigma,0} \rangle \ = \ 
\frac{\langle \vartheta^\epsilon_{\Pi}, \vartheta^\eta_{\Sigma} \rangle}{p^{\epsilon}(\Pi)\, p^{\eta}(\Sigma)}.
\end{equation}

\medskip
\paragraph{\bf The choice of signs $\epsilon$ and $\eta$}
\label{sec:signs}
Recall that $\epsilon = (\epsilon_v)_{v \in S_r}$ and $\eta = (\eta_v)_{v \in S_r}$. It is clear from 
Lem.\,\ref{lem:delta-action} that we can expect to have a nonzero pairing in (\ref{eqn:global-pairing}) only when 
$$
\epsilon_v \ = \ (-1)^n \eta_v, \quad \forall v \in S_r.
$$
Further, by parity constraints as in \ref{sec:cuspidal-range}, we know that 
$\epsilon$ or $\eta$ is uniquely determined by $\Pi$ or $\Sigma$, depending on whether $n$ is odd or even. To summarize, for every $v \in S_r$, we have
\begin{enumerate}
\item $\epsilon_v = (-1)^n\eta_v$. 
\item 
   \begin{itemize}
     \item If $n$ is odd then let $\epsilon_v = \omega_{\Pi_v}(-1) \cdot (-1)^{{\sf w}(\mu)/2}$; 
     \item if $n$ is even then let $\eta_v = \omega_{\Sigma_v}(-1) \cdot (-1)^{{\sf w}(\lambda)/2}.$ 
   \end{itemize}
\end{enumerate}

\medskip
\paragraph{\bf The pairing at infinity}
As in \cite[Sect.\,3.2.5]{raghuram-imrn}, the computation of the pairing in (\ref{eqn:global-pairing}) involves a certain archimedean contribution. Recall from \ref{sec:comparison}, the generator $[\Pi_{\infty}]^{\epsilon}$ of the one-dimensional space $H^{b_n^F}(\g_n ,K_{n, \infty}^0; \Whit(\Pi_{\infty}) \otimes \M_{\mu,\C}^{\sf v})(\epsilon),$ and 
similarly, we have $[\Sigma_{\infty}]^{\eta}$ generating the one-dimensional 
$H^{b_{n-1}^F}(\g_{n-1} ,K_{n-1, \infty}^0; \Whit(\Sigma_{\infty}) \otimes \M_{\lambda,\C}^{\sf v})(\eta).$
To compute the pairing at infinity, fix a basis 
$\{{\bf y}_j : 1 \leq j \leq d_{n-1}^F\}$ for $(\mathfrak{g}_{n-1,\infty}/\c_{n-1,\infty})^*$ such that 
$\{{\bf y}_j : 1 \leq j \leq d_{n-1}^F-1\}$ is a basis for $(\mathfrak{g}_{n-1,\infty}/\k_{n-1,\infty})^*$.  
Next, fix a basis $\{{\bf x}_i : 1 \leq i \leq d_n^F-1\}$ for 
$(\mathfrak{g}_{n,\infty}/\mathfrak{k}_{n,\infty})^*$, 
such that 
$\iota^*{\bf x}_j = {\bf y}_j$ for all 
$1 \leq j \leq {\rm dim}(\mathfrak{g}_{n-1,\infty}/\mathfrak{k}_{n-1,\infty})^* = d_{n-1}^F-1$, and 
$\iota^*{\bf x}_i = 0$ if $i \geq d_{n-1}^F$. We further note that 
${\bf y}_1\wedge {\bf y}_2 \wedge \cdots \wedge {\bf y}_{d_{n-1}^F}$ corresponds to 
a $G_{n-1}({\mathbb R})^0$-invariant measure on $\tilde{S}^{G_{n-1}}_{R_f}$. 
Let $\{m_{\alpha}^{\sf v}\}$ (resp., $\{m_{\beta}^{\sf v}\}$) be a $\Q$-basis for $\M_{\mu}^{\sf v}$ (resp., 
$\M_{\lambda}^{\sf v}$). 
The class $[\Pi_{\infty}]^\epsilon$ is represented by a $K_{n,\infty}^0$-invariant element in 
$\wedge^{b_n}(\mathfrak{g}_{n,\infty}/\mathfrak{k}_{n,\infty})^* \otimes \Whit(\Pi_{\infty}) \otimes \M_{\mu,\C}^{\sf v}$
which we write as 
\begin{equation}
\label{eqn:pi-infty}
[\Pi_{\infty}]^\epsilon = 
\sum_{{\bf i} = i_1 < \cdots < i_{b_n^F}}
\sum_{\alpha} \ {\bf x}_{\bf i} \otimes w_{\infty,\bf{i},\alpha}  \otimes m_{\alpha}^{\sf v},
\end{equation}
where $w_{\infty,\bf{i},\alpha} \in \Whit(\Pi_{\infty},\psi_{\infty})$. Thinking in terms of the K\"unneth theorem for relative Lie algebra cohomology, the choice of basis $\{{\bf x}_i\}$ and 
$\{m_{\alpha}^{\sf v}\}$ can be made so that we have: 
$$
[\Pi_{\infty}]^\epsilon \ = \ \otimes_{v \in S_\infty} [\Pi_v]^{\epsilon_v}. 
$$
(It is understood that for $v \in S_c$ there is no sign $\epsilon_v.$) 
Similarly, $[\Sigma_{\infty}]^\eta$ is represented by a $K_{n-1,\infty}^0$-invariant element in 
$\wedge^{b_{n-1}^F}(\mathfrak{g}_{n-1,\infty}/\mathfrak{k}_{n-1,\infty})^* \otimes \Whit(\Sigma_{\infty}) \otimes 
\M_{\lambda}^{\sf v}$ which we write as:
\begin{equation}
\label{eqn:sigma-infty}
[\Sigma_{\infty}]^\eta = 
\sum_{{\bf j} = j_1 < \cdots < j_{b_{n-1}^F}} \sum_{\beta} \ 
{\bf y}_{\bf j} \otimes w_{\infty,\bf{j},\beta}  \otimes m_{\beta}^{\sf v},
\end{equation}
with $w_{\infty,\bf{j},\beta} \in \Whit(\Sigma_{\infty},\psi^{-1}_{\infty})$. Again, we have 
$$
[\Sigma_{\infty}]^\eta \ = \ \otimes_{v \in S_\infty} [\Sigma_v]^{\eta_v}. 
$$

We now define a pairing at infinity by
\begin{equation}
\label{eqn:pairing-infinity}
\langle [\Pi_{\infty}]^\epsilon, [\Sigma_{\infty}]^\eta \rangle = 
\sum_{{\bf i}, {\bf j}} s({\bf i}, {\bf j}) 
\sum_{\alpha, \beta} 
\cT(m_{\beta}^{\sf v} \otimes m_{\alpha}^{\sf v})  
\Psi_{\infty}(1/2, w_{\infty,{\bf i},\alpha}, w_{\infty,{\bf j},\beta}), 
\end{equation}
where $s({\bf i},{\bf j}) \in \{0,-1,1\}$ is defined by 
$\iota^*{\bf x}_{\bf i} \wedge {\bf y}_{\bf j} = s({\bf i},{\bf j}) {\bf y}_1 \wedge {\bf y}_2 \wedge \cdots 
\wedge {\bf y}_{d_{n-1}^F}$. Recall that 
$\Psi_{\infty}(1/2, w_{\infty,{\bf i},\alpha}, w_{\infty,{\bf j},\beta})$ is defined only after 
meromorphic continuation. However, the assumption on $\mu$ and $\lambda$ in (\ref{eqn:comp}) guarantees that 
$s = 1/2$ is critical which ensures that the integrals $\Psi_{\infty}(1/2, w_{\infty,{\bf i},\alpha}, w_{\infty,{\bf j},\beta})$ are all finite, hence $\langle [\Pi_{\infty}], [\Sigma_{\infty}] \rangle$ is finite. Furthermore, computing this pairing is a purely local problem since 
$$
\langle [\Pi_{\infty}]^\epsilon, [\Sigma_{\infty}]^\eta \rangle \ = \ \prod_{v \in S_\infty} 
\langle [\Pi_v]^{\epsilon_v}, [\Sigma_v]^{\eta_v} \rangle. 
$$ 
 Binyong Sun \cite{sun} has recently proved that the local pairings are all nonzero giving us the following  
\begin{thm}
\label{thm:nonvanishing}
$
\langle [\Pi_{\infty}]^\epsilon, [\Sigma_{\infty}]^\eta \rangle \neq 0.
$
\end{thm}
The quantity $\langle [\Pi_{\infty}]^\epsilon, [\Sigma_{\infty}]^\eta \rangle$ depends only on the weights $\mu$ and $\lambda$, and the signs 
$\epsilon$ and $\eta$ which are determined as in 
\ref{sec:signs}. Now define:
\begin{equation}
\label{eqn:p-infinity}
p_{\infty}^{\epsilon, \eta}(\mu,\lambda) := \frac{1}{\langle [\Pi_{\infty}]^\epsilon, [\Sigma_{\infty}]^\eta \rangle}.
\end{equation}
Ultimately, one should expect $p_{\infty}^{\epsilon, \eta}(\mu,\lambda)$
to be a power of $(2\pi i)$.

\medskip
\subsubsection{\bf The main identity}
The following theorem is a generalization of \cite[Thm.\,3.12]{raghuram-imrn}. 

\begin{thm}[Main Identity]
\label{thm:main-identity}
Let $\mu \in X^+_{00}(T_n)$ and $\Pi \in \Coh(G_n,\mu^{\sf v})$. Let $\lambda \in X^+_{00}(T_{n-1})$ and 
$\Sigma \in \Coh(G_{n-1},\lambda^{\sf v}).$ 
Assume that $\mu$ and $\lambda$ satisfy $\mu^{\sf v} \succ \lambda$; in particular, $s=1/2$ is critical for 
$L_f(s, \Pi \times \Sigma)$. Take the signs $\epsilon, \eta$ as in \ref{sec:signs}. 
Let $\vartheta^\epsilon_{\Pi,0}$ and $\vartheta^\eta_{\Sigma,0}$ be the normalized classes defined in 
(\ref{eqn:coh-classes-0}). Then
$$
\frac{L_f(\tfrac12, \Pi \times \Sigma)}
{p^{\epsilon}(\Pi)\, p^{\eta}(\Sigma)\, p_{\infty}^{\epsilon,\eta}(\mu,\lambda)} \ = \  
\frac{\prod_{v \in S_{\Sigma}} L(\tfrac12, \Pi_v \times \Sigma_v)}
{{\rm vol}(\Sigma)\, \prod_{v \notin S_{\Sigma} \cup S_\infty} c_{\Pi_v}}\, 
\langle \vartheta^\epsilon_{\Pi,0},  \vartheta^\eta_{\Sigma,0} \rangle,
$$
where the pairing on the right hand side is defined in (\ref{eqn:global-pairing}), the nonzero rational number  
${\rm vol}(\Sigma)$ is as in Prop.\,\ref{prop:rankin-selberg}, 
and $c_{\Pi_v}$ is defined in \ref{sec:newvectors}.
\end{thm}
Since $c_{\Pi_v} = 1$ for an unramified places $v$, the infinite product in the denominator of the right hand side is in fact a finite product.  

\begin{proof}
We will only adumbrate the proof, since the proof of \cite[Thm.\,3.12]{raghuram-imrn} goes through {\it mutatis mutandis}. 
Using the definition of normalized classes in (\ref{eqn:coh-classes-0}), we see that it is enough to show: 
$$
\langle \vartheta^\epsilon_{\Pi},  \vartheta^\eta_{\Sigma} \rangle \ = \ 
\frac{L_f(\tfrac12, \Pi \times \Sigma) \cdot 
{\rm vol}(\Sigma)\cdot \prod_{v \notin S_{\Sigma} \cup S_\infty} c_{\Pi_v}}
{\prod_{v \in S_{\Sigma}} L(\tfrac12, \Pi_v \times \Sigma_v) \cdot p_{\infty}^{\epsilon,\eta}(\mu,\lambda)}.
$$
Using (\ref{eqn:coh-classes}) and (\ref{eqn:pi-infty}) we may write the cohomology classes as: 
$$
\vartheta^\epsilon_{\Pi}  = 
\sum_{{\bf i}}
\sum_{\alpha} \ {\bf x}_{\bf i} \otimes \phi_{\bf{i},\alpha}  \otimes m_{\alpha}^{\sf v},
$$
where the cusp form $\phi_{{\bf i}, \alpha}$, in the $\psi$-Whittaker model of $\Pi$, looks like 
$w_{\Pi_f} \otimes w_{\infty, {\bf i}, \alpha}.$ Similarly, 
$$
\vartheta^\eta_{\Sigma}  = 
\sum_{{\bf j}}
\sum_{\beta} \ {\bf y}_{\bf j} \otimes \phi_{\bf{j},\beta}  \otimes m_{\beta}^{\sf v},
$$
where the cusp form $\phi_{{\bf j}, \beta}$, in the $\psi^{-1}$-Whittaker model of $\Sigma$, looks like 
$w_{\Sigma_f} \otimes w_{\infty, {\bf j}, \beta}.$ 
Computing the global pairing as defined in (\ref{eqn:global-pairing}), in the situation where 
the signatures are as prescribed in 
\ref{sec:signs}, while using the above expressions for $\vartheta^\epsilon_{\Pi}$ and $\vartheta^\eta_{\Sigma}$, we get
$$
\langle \vartheta^\epsilon_{\Pi},  \vartheta^\eta_{\Sigma} \rangle \ = \ 
\sum_{{\bf i}, {\bf j}} 
\sum_{\alpha, \beta} 
s({\bf i}, {\bf j})  
\cT(m_{\beta}^{\sf v} \otimes m_{\alpha}^{\sf v})  
I(\tfrac12, \phi_{{\bf i}, \alpha}, \phi_{{\bf j}, \beta}).
$$
Using Prop.\,\ref{prop:rankin-selberg} and the pairing at infinity (\ref{eqn:pairing-infinity}) this may be re-written as  
$$
\langle \vartheta^\epsilon_{\Pi},  \vartheta^\eta_{\Sigma} \rangle \ = \ 
\frac{L_f(\tfrac12, \Pi \times \Sigma) \cdot 
{\rm vol}(\Sigma)\cdot \prod_{v \notin S_{\Sigma} \cup S_\infty} c_{\Pi_v}}
{\prod_{v \in S_{\Sigma}} L(\tfrac12, \Pi_v \times \Sigma_v)} 
\langle [\Pi_{\infty}]^\epsilon, [\Sigma_{\infty}]^\eta \rangle .
$$
The main identity now follows from (\ref{eqn:p-infinity}). 
\end{proof}

\medskip
\subsubsection{\bf The proof of Theorem~\ref{thm:main}} 

\paragraph{\bf Central critical value}
Suppose $\mu^{\sf v} \succ \lambda$ then $s = 1/2$ is critical. The proof of \cite[Thm.\,1.1]{raghuram-imrn} goes through {\it mutatis mutandis} giving an algebraicity result 
for the central critical value; the proof entails verifying that the right hand side of the main identity in 
Thm.\,\ref{thm:main-identity} above is Galois equivariant, i.e., well-behaved under the action of $\sigma \in \autc.$ 
This is the essence of \cite[Sect.\,3.3]{raghuram-imrn}, all of which appropriately generalizes to the situation of this paper and so we merely state the necessary observations to make while omitting the details.  


\begin{itemize}

\item The Poincar\'e duality pairing  $\langle \, ,  \,  \rangle$
is Galois-equivariant; see \cite[Prop.\,3.14]{raghuram-imrn}.

\item The Galois equivariance of the classes $\vartheta^\epsilon_{\Pi,0}$ and $\vartheta^\eta_{\Sigma,0}$, is somewhat more subtle, because of our specific choice of finite Whittaker vectors $w_{\Pi_f}$ and $w_{\Sigma_f}$. Using 
\cite[Prop.\,3.15]{raghuram-imrn} and \cite[Cor.\,3.14]{raghuram-imrn}, both of which easily go through in our setting, we get:
$$
{}^{\sigma}\!\vartheta^\epsilon_{\Pi, 0}  =  
\frac{\sigma(\mathcal{G}(\omega_{\Sigma_f}))}
{\mathcal{G}(\omega_{\Sigma_f^{\sigma}})}
\vartheta^\epsilon_{\Pi^{\sigma},0}, \quad {\rm and} \quad 
{}^{\sigma}\!\vartheta^\eta_{\Sigma,0} =  \vartheta^\eta_{\Sigma^{\sigma}, 0}. 
$$

\item Noting that \cite[Prop.\,3.17]{raghuram-imrn} works for any local field $F_v$ we get that local $L$-values are suitably rational, and moreover, we have: 
$$
\sigma(L(1/2, \Pi_v \times \Sigma_v)) = L(1/2, \Pi_v^{\sigma} \times \Sigma_v^{\sigma}).
$$

\item Finally, exactly as in \cite[Prop.\,3.18]{raghuram-imrn}, we have $\sigma(c_{\Pi_v}) = c_{{}^{\sigma}\!\Pi_v}.$

\end{itemize}

\medskip
\paragraph{\bf All critical values}
Now suppose $\mu$ and $\lambda$ satisfy (\ref{eqn:comp}), and suppose $\tfrac12 + m$ is critical. Then we take a suitable Tate twist and consider a situation when $\tfrac12$ is critical, and then apply the period relations in  \ref{sec:period-relaitons}. The reader should bear in mind that Thm.\,\ref{thm:critical-compatible} says that the set of possible Tate-twists we can take, subject to the restriction imposed by the compatibility condition (\ref{eqn:comp}), 
is exactly the set of all critical points for $L(s, \Pi \times \Sigma).$ However, the parity of $n$ will play a role, because this will affect the recipe for signs $\epsilon$ and $\eta$ as in \ref{sec:signs}. We argue as follows: 

Suppose $n$ is even, then we absorb the $m$ into the representation of $\GL_n$ as:
$$
L_f(\tfrac12 +m, \Pi \times \Sigma) 
\ = \ L_f(\tfrac12, \Pi(m) \times \Sigma) \ \sim \  
p^{\epsilon}(\Pi(m))\, p^{\eta}(\Sigma)\, \G(\omega_{\Sigma_f})\, p^{\epsilon, \eta}_{\infty}(\mu+m,\lambda).  
$$
Using (\ref{eqn:period-reln}) 
we can write this as 
$p^{\epsilon_m \epsilon}(\Pi)\, p^{\eta}(\Sigma)\, \G(\omega_{\Sigma_f})\, p^{\epsilon, \eta}_{\infty}(\mu+m,\lambda). $
Suppose $n$ is odd, then we absorb the $m$ into the representation of $\GL_{n-1}$, and argue similarly.

\bigskip
\section{\bf Symmetric power $L$-functions}

\subsection{Symmetric power transfers are cohomological}\label{sec:symm-powers}

\subsubsection{\bf Definition of ${\rm Sym}^r(\pi)$ and $L(s, \Sym^r(\pi) \otimes \chi)$} 
Let $\pi$ be a cohomological cuspidal automorphic representation of $\GL_2$ over $F$. 
Let $\Sym^r : \GL_2(\C) \to \GL_{r+1}(\C)$ be the $r$-th symmetric power of the standard representation of $\GL_2(\C)$. By the local Langlands correspondence, see Harris--Taylor \cite{harris-taylor} and Henniart \cite{henniart} for the non-archimedean places and Langlands \cite{lang2} for the archimedean places, the local transfer 
$\Sym^r(\pi_v)$ is defined as an irreducible admissible representation of $\GL_{r+1}(F_v)$. 
Now define $\Sym^r(\pi):=\bigotimes'_v {\rm Sym}^r(\pi_v)$ which is a well-defined irreducible admissible representation of $\GL_{r+1}(\A_F)$. Langlands's functoriality predicts that $\Sym^r(\pi)$ is an isobaric automorphic representation of $\GL_{r+1}(\A_F)$; this is known for $r \leq 4$ by Gelbart--Jacquet~\cite{gelbart-jacquet},  
Kim--Shahidi~\cite{kim-shahidi-annals}, and Kim~\cite{kim-jams} for a general $\pi$, and it is known for all $r$ if $\pi$ is dihedral. Define the $r$-th symmetric power $L$-function of $\pi$ as the standard $L$-function of the $r$-th symmetric power transfer of $\pi$; we may also introduce a twisting Hecke character $\chi$. We are interested in the special values of such twisted symmetric power $L$-functions of $\pi$: $L(s, \Sym^r(\pi) \otimes \chi).$ As in \cite{raghuram-imrn} we approach odd symmetric power $L$-functions inductively. To get started, consider the case when $r=1$.

\medskip

\subsubsection{\bf $L$-functions for $\GL_2 \times \GL_1$} 
\label{sec:explicate-gl2-gl1}
Let's explicate the $n=2$ case of Thm.\,\ref{thm:main}. The reader should note that this particular case is not new; see, for example, Shimura~\cite{shimura-duke} when $F$ is totally real, and Hida \cite{hida-duke} in the general case, although our notations are rather different. 

\smallskip

Let $\pi \in \Coh(G_2, \mu^{\sf v})$ with $\mu \in X^+_{00}(T_2).$ Let $\mu = (\mu^\iota)_{\iota \in \cE_F}$, and 
$\mu^\iota = (a^\iota, b^\iota) \in \Z^2$ with $a^\iota \geq b^\iota.$ Let ${\sf w} = {\sf w}(\mu)$ be the purity weight of 
$\mu.$ Recall: for any $v \in S_r$, ${\sf w} = a^\iota + b^\iota$, and for any $v \in S_c$, if 
$\mu^{\iota_v} = (a^{\iota_v}, b^{\iota_v})$ then $\mu^{\bar{\iota}_v} = ({\sf w}- b^{\iota_v}, {\sf w}- a^{\iota_v}).$

\smallskip

Let $\lambda \in \Z$ be thought of as an element $\lambda = (\lambda^\iota)_{\iota \in \cE_F} \in X^+_{00}(T_1),$ where each $\lambda^\iota = \lambda.$ The corresponding representation $\M_\lambda$ of $G_1$ is  
${\rm det}^\lambda.$ The purity weight is ${\sf w}(\lambda) = 2\lambda.$ Let $\chi \in \Coh(G_1, \lambda^{\sf v})$, then 
$\chi = \chi^\circ \otimes |\ |^{\lambda}$, where $\chi^\circ : F^\times \backslash \A_F^\times \to \C^\times$ is 
a character of finite order. For any $v \in S_r$, define $\epsilon_{\chi,v} = (-1)^\lambda \cdot \chi^\circ_v(-1)$; the signature of $\chi$ is $\epsilon_\chi = (\epsilon_{\chi, v})_{v \in S_r}.$  
 
\smallskip 
 
The compatibility condition (\ref{eqn:comp}) for the pair $(\mu, \lambda)$ 
is satisfied if there is an integer $j$ such that $-b^\iota \geq \lambda +j \geq -a^\iota$ for all $\iota$; this is best considered in two cases:  
\begin{enumerate}
\item $v \in S_r$: here we have $a^\iota - {\sf w} \geq \lambda +j \geq -a^\iota.$ Note a consequence of these inequalities is that $a^\iota - {\sf w} \geq -[{\sf w}/2] \geq -a^\iota.$
\item $v \in S_c$: here we have $-b^\iota \geq \lambda +j \geq -a^\iota$ and 
$a^\iota - {\sf w} \geq  \lambda +j \geq  b^\iota - {\sf w}.$ A consequence of both these inequalities are: 
$-b^\iota \geq -[{\sf w}/2]  \geq -a^\iota$ and $a^\iota - {\sf w} \geq  -[{\sf w}/2] \geq  b^\iota - {\sf w}.$ 
\end{enumerate}
It is important to appreciate the possibility that the compatibility condition (\ref{eqn:comp}) need not always be satisfied. For example, take $F$ to be imaginary quadratic extension of $\Q$, take $a^\iota = b^\iota = 0$ and ${\sf w} \neq 0$; then clearly the necessary condition in (2) is not satisfied. Furthermore, the reader can check in this situation that 
$L(s, \pi)$ doesn't have any critical points. 
(As explained in \ref{sec:subtlety}, even when $F$ is an imaginary quadratic extension and say 
$\cE_F = \{\iota, \bar\iota\}$, as in 
\cite{hida-duke} we could conceivably have critical points of type $A = \{\iota\}$, however, such a critical point is outside the purview of the compatibility condition (\ref{eqn:comp}).)
At the other extreme, if $F$ is totally real, then (\ref{eqn:comp}) is satisfied for $j = -\lambda - [{\sf w}/2] =  -[({\sf w}(\lambda) + {\sf w}(\mu))/2]$; and by  Thm.\,\ref{thm:critical-compatible}, there are critical points for $L(s, \pi \otimes \chi).$ Finally, let's note that for any number field with at least one real place, if $\mu$ is a parallel weight then (\ref{eqn:comp}) is satisfied and $L(s, \pi \otimes \chi)$ has critical points. Assume henceforth that  (\ref{eqn:comp}) holds for $(\mu, \lambda).$

\smallskip

Let $\tfrac12 + m \in \tfrac12 + \Z$ be any critical point for $L(s, \pi \times \chi)$. By Thm.\,\ref{thm:critical-compatible} this is the same as $m \in \Z$ such that $\mu^{\sf v} \succ \lambda +m.$  The sign $\eta = \epsilon_\chi$ as given above, and since $n=2$, $\epsilon = \eta.$ Further, $\chi$ being on $\GL_1$, we have  $p^\eta(\chi) \sim 1.$ 
The statement of Thm.\,\ref{thm:main} gives: 
$$
L_f(\tfrac12 + m, \pi \otimes \chi) \ \sim \ 
p^{\epsilon_m \epsilon_\chi}(\pi) \, \G(\chi) \, p_\infty^{\epsilon_\chi}(\mu+m, \lambda), 
$$
and more generally, the ratio of the left hand side by the right hand side is $\autc$-equivariant. When $F$ is totally real, the quantity $p_\infty^{\epsilon_\chi}(\mu+m, \lambda)$ has been explicitly calculated in 
\cite[Prop.\,3.24]{raghuram-tanabe}. In general, one expects, as in the totally real case, that this quantity is a rational multiple of an integral power of $2 \pi i.$

\smallskip

We will apply the above discussion to the situation $L(s, \pi \times \omega_\pi\xi)$ where $\omega_\pi$ is the central character of $\pi$ and $\xi$ is a finite order Hecke character of $F$. Given $\mu$ as above, define 
${\rm det}(\mu)$ by ${\rm det}(\mu)^\iota = a^\iota+b^\iota$; then ${\sf w}({\rm det}(\mu)) = 2{\sf w}(\mu).$ 
We have:
\begin{equation}
\label{eqn:det-coh}
\pi \in \Coh(G_2, \mu^{\sf v}) \implies \omega_\pi \in \Coh(G_1, {\rm det}(\mu)^{\sf v}). 
\end{equation}
The latter also implies that $\omega_\pi\xi \in \Coh(G_1, {\rm det}(\mu)^{\sf v}).$ As above, the compatibility condition 
(\ref{eqn:comp}) for the pair $(\mu, {\rm det}(\mu))$ need not be satisfied in general.

\subsubsection{\bf Transferring the weights $\mu$}
\label{sec:defn-symr-mu}
Let $\mu \in X^+_{00}(T_2)$. As above, suppose $\mu = (\mu^\iota)_{\iota \in \cE_F}$, with  
$\mu^\iota = (a^\iota, b^\iota) \in \Z^2$ and $a^\iota \geq b^\iota.$  Define a weight 
$\Sym^r(\mu) = (\Sym^r(\mu)^\iota)_{\iota \in \cE_F} \in X^+(T_{r+1})$, where for each $\iota$ we have: 
$$
\Sym^r(\mu)^\iota \ := \ \Sym^r(\mu^\iota) \ := \ (ra^\iota, (r-1)a^\iota + b^\iota, \dots, a^\iota + (r-1)b^\iota, rb^\iota).
$$
If ${\sf w} = {\sf w}(\mu)$ be the purity weight of $\mu$, then it is easy to check that $\Sym^r(\mu)$ is also pure and it's purity weight is ${\sf w}(\Sym^r(\mu)) = r {\sf w}.$ Furthermore, one checks that $\Sym^r(\mu)$ is strongly pure, i.e., 
$\Sym^r(\mu) \in X^+_{00}(T_{r+1}).$ Also, if $\mu$ is a parallel weight, then so is $\Sym^r(\mu).$

\subsubsection{\bf Symmetric power transfer preserves the property of being cohomological}

The following theorem is a generalization of our result with Shahidi \cite[Thm.\,5.5]{raghuram-shahidi-aim} and is 
a variation of Labesse--Schwermer \cite[Prop.\,5.4]{labesse-schwermer}.

\begin{thm}\label{thm:sym-coh}
Let $\mu \in X^+_{00}(T_2)$ and $\pi \in \Coh(G_2, \mu^{\sf v})$. Suppose $\Sym^r(\pi)$ is a cuspidal automorphic representation of $G_{r+1}$, then $\Sym^r(\pi) \in \Coh(G_{r+1}, \Sym^r(\mu)^{\sf v}).$
\end{thm}

\begin{proof}
The proof is purely local, as it suffices to check that for $v \in S_\infty$, the representation $\Sym^r(\pi)_v$ has nontrivial relative Lie algebra cohomology after twisting by $\M_{\Sym^r(\mu_v)^{\sf v}}.$ Consider two cases: 

\smallskip

\noindent
$v \in S_r:$ Let's suppress the $\iota_v = \iota$ and write $\mu^\iota = (a,b).$ Then 
$\pi_v = D(a-b+1)|\ |^{{\sf w}/2}$, whose Langlands parameter we denote for brevity as $I(\xi_{l})({\sf w}/2)$ with $l = a-b+1$;  it is, up to twisting by  $|\ |^{{\sf w}/2}$, 
the two-dimensional representation of the Weil group $W_\R$ of $\R$ which is induced from the Weil group 
$W_\C = \C^\times$ of $\C$ by the character $\xi_{l}$ given by $z = re^{i\theta} \mapsto e^{i l \theta}.$ A pleasant exercise gives: 
$$
\Sym^r(I(\xi_{l})) \ = \ \left\{
\begin{array}{ll}
I(\xi_{l}) \oplus I(\xi_{3l}) \oplus \cdots \oplus I(\xi_{rl}), & \mbox{if $r$ is odd,} \\
 & \\
{\rm sgn}^{rl/2} \oplus I(\xi_{2l}) \oplus I(\xi_{4l}) \oplus \cdots \oplus I(\xi_{rl}), & \mbox{if $r$ is even}.
\end{array}\right.
$$
Also, $\Sym^r(I(\xi_{l})({\sf w}/2)) = \Sym^r(I(\xi_{l}))(r{\sf w}/2).$ 
By the local Langlands correspondence for $\GL_{r+1}(\R)$ (see Knapp~\cite{knapp}), we see that 
$\Sym^r(\pi_v)$ is the representation in (\ref{eq:j-mu-even}) or (\ref{eq:j-mu-odd}) exactly when the corresponding highest weight is taken as $\Sym^r(\mu)$ when defined as in Sect.\,\ref{sec:defn-symr-mu}.  

\smallskip

\noindent
$v \in S_c:$ Suppressing again the notations for $\{ \iota_v, \bar{\iota}_v\}$, we write 
$\mu_v = \{(a,b), ({\sf w}-b, {\sf w}-a)\}.$ Then, the Langlands parameter of $\pi_v = J_{\mu_v}$, see 
(\ref{eqn:j-mu-c}), is 
$\xi(a+\tfrac12, {\sf w}-a-\tfrac12) \oplus \xi(b-\tfrac12, {\sf w}-b+\tfrac12)$, where for half-integers $p,q$ by $\xi(p,q)$ is meant the character of $W_\C$ given by $z \mapsto z^p \bar{z}^q.$ Then the Langlands parameter of 
$\Sym^r(\pi_v)$ is given by:
$$
\xi(ra+\tfrac{r}{2}, r{\sf w}-ra-\tfrac{r}{2}) \oplus 
\xi((r-1)a+b+\tfrac{r-2}{2}, r{\sf w}-(r-1)a-b-\tfrac{r-2}{2}) 
\oplus \cdots \oplus 
\xi(rb-\tfrac{r}{2}, r{\sf w}-rb+\tfrac{r}{2}). 
$$ 
Hence, $\Sym^r(\pi_v)$, via (\ref{eqn:cuspidal-parameters-c}) and 
(\ref{eqn:j-mu-c}), corresponds to the highest weight 
$\Sym^r(\mu_v).$
\end{proof}

Thm.\,\ref{thm:sym-coh} and (\ref{eqn:det-coh}) lend further evidence to the discussion in \cite[Sect.\,5.2]{raghuram-shahidi-aim} relating functoriality and the property of being cohomological.

\medskip
\subsection{Special values of  Symmetric power $L$-functions}

\subsubsection{\bf A factorization of $L$-functions}
Let $\pi$ be a cuspidal automorphic representation of $G_2(\A)$, and suppose  $\Sym^r(\pi)$ is automorphic for all $r$, then our main idea behind special values of symmetric power $L$-functions is: 
$$
L_f\left(s, \Sym^r(\pi) \times \Sym^{r-1}(\pi) \right) \ = \ 
\prod_{a=1}^r
L_f\left(s, \Sym^{2a-1}(\pi)\otimes\omega_{\pi}^{r-a}\right); 
$$
see \cite[Cor.\,5.2]{raghuram-imrn}. Now suppose that $\pi \in \Coh(G_2,\mu^{\sf v})$ and suppose that the symmetric power transfers are all cuspidal, then we apply Thm.\,\ref{thm:main} to get algebraicity results for the values of 
$L_f\left(s, \Sym^r(\pi) \times \Sym^{r-1}(\pi) \right)$, and inductively, we get special value results for 
odd symmetric power $L$-functions on the right hand side. As can be seen in \cite[Prop.\,5.4]{raghuram-imrn}, where the base field $F$ was  $\Q$, carrying out this exercise can be combinatorially tedious. In the rest of this article, we carry through the above idea for symmetric cube $L$-functions of $\pi$. Algebraicity results for the critical values of symmetric cube $L$-functions are available in the literature in various special cases; see 
Garrett--Harris \cite[Thm.\,6.2]{garrett-harris}, Kim--Shahidi \cite[Prop.\,4.1]{kim-shahidi-israel}, 
Grobner--Raghuram \cite[Cor.\,8.1.2]{grobner-raghuram-ajm}, and Januszewski  \cite[Sect.\,6]{januszewski}. 
The following results are new when $F$ is a general number field and the representation $\pi$ is cohomological with respect to a general strongly pure coefficient system $\mu.$

\medskip
\subsubsection{\bf Critical points for symmetric cube $L$-functions}\label{sec:critical-sym3}

\begin{prop}
\label{prop:critical-points-sym3-L-fn}
Let $\pi \in \Coh(G_2, \mu^{\sf v})$ with $\mu \in X^+_{00}(T_2).$ Let $\mu = (\mu^\iota)_{\iota \in \cE_F}$, and 
$\mu^\iota = (a^\iota, b^\iota) \in \Z^2$ with $a^\iota \geq b^\iota.$ Let ${\sf w} = {\sf w}(\mu)$ be the purity weight of 
$\mu.$ Let $\chi$ be any Hecke character of $F$ of finite order. Then, $\tfrac12 + m \in \tfrac12 + \Z$ is critical for 
$L\left(s, \Sym^3(\pi) \otimes \chi\right)$ if and only if $m$ satisfies the inequalities in (\ref{eqn:sym3-inequalities-1}) and 
(\ref{eqn:sym3-inequalities-2}) below. For $v \in S_r$ and $\iota_v \in \cE_F$ the corresponding embedding, we have
\begin{equation}
\label{eqn:sym3-inequalities-1}
-2a^{\iota_v}-b^{\iota_v}\  \leq \ m \ \leq \ -2b^{\iota_v}-a^{\iota_v}.
\end{equation}
For $v \in S_c$ and $\{\iota_v, \bar{\iota}_v\}$ the corresponding pair of embeddings, suppose  
$\alpha_v$ stands for the minimum of 
$\{|6a^{\iota_v} - 3{\sf w}+3|, \ 
|4a^{\iota_v}+2b^{\iota_v}-3{\sf w}+1|,  \ 
|2a^{\iota_v}+4b^{\iota_v}-3{\sf w}-1|,\ 
|6b^{\iota_v} - 3{\sf w}-3|\}$, then 
\begin{equation}
\label{eqn:sym3-inequalities-2}
\frac{1 - 3{\sf w} - \alpha_v}{2} \ \leq \ m \ \leq \  \frac{-1-3{\sf w}+ \alpha_v}{2}.
\end{equation}
\end{prop}

\begin{proof}
The proof is a somewhat tedious exercise, and we merely sketch it leaving all the details to the reader. Consider two cases: \\

\noindent
$v \in S_r$: suppressing the superscript $\iota_v$, let $\mu_v = (a,b)$. Then up to nonzero constants and exponential functions, we have
$$
L(s, \Sym^3(\pi_v)) \ = \ 
\Gamma\left(s + \frac{3{\sf w}}{2} + \frac{a-b+1}{2}  \right) 
\Gamma\left(s  + \frac{3{\sf w}}{2} + \frac{3(a-b+1)}{2}  \right). 
$$
We leave it to the reader to check that $L(s, \Sym^3(\pi_v))$ and $L(1-s, \Sym^3(\pi_v)^{\sf v})$ are regular at 
$s = \tfrac12 + m$ if and only if $m$ satisfies (\ref{eqn:sym3-inequalities-1}). \\

\noindent
$v \in S_c$: suppressing the superscripts $\iota_v$ and $\bar{\iota}_v$, up to nonzero constants and exponential functions, we have
\begin{multline*}
L(s, \Sym^3(\pi_v)) \ = \ 
\Gamma\left(s + \frac{3{\sf w}}{2} + \frac{|6a - 3{\sf w}+3|}{2}\right) 
\Gamma\left(s + \frac{3{\sf w}}{2} + \frac{|4a + 2b - 3{\sf w}+1|}{2} \right) \cdot \\
\cdot \Gamma\left(s + \frac{3{\sf w}}{2} + \frac{|2a + 4b - 3{\sf w}-1|}{2} \right)
\Gamma\left(s + \frac{3{\sf w}}{2} + \frac{|6b - 3{\sf w}-3|}{2}\right) . 
\end{multline*}
We leave it to the reader to check that $L(s, \Sym^3(\pi_v))$ and $L(1-s, \Sym^3(\pi_v)^{\sf v})$ are regular at 
$s = \tfrac12 + m$ if and only if $m$ satisfies (\ref{eqn:sym3-inequalities-2}). 
\end{proof}

The following special case of a parallel weight when $F$ has at least one real place is interesting, and the reader can just as well consider only this case: 

\begin{cor}
\label{cor:parallel-weight-critical}
Let $\pi \in \Coh(G_2, \mu^{\sf v})$ with $\mu \in X^+_{00}(T_2)$ being a parallel weight 
$\mu^\iota = (a, b) \in \Z^2$ with $a \geq b.$  Suppose $S_r \neq \emptyset$, i.e., $F$ is not totally imaginary then 
the purity weight is ${\sf w} = {\sf w}(\mu) = a+b$ and furthermore, the set 
$$
\left\{ \tfrac12 + m \in \tfrac12 + \Z \ : \ \ -2a - b  \ \leq \ m \ \leq \ -a-2b \right\}
$$
is the critical set for $L\left(s, \Sym^3(\pi) \otimes \chi\right)$, $L\left(s, \pi \otimes \omega_\pi\chi\right)$ and 
$L\left(s, \Sym^2(\pi) \times \pi \otimes \chi\right)$. 
\end{cor}

\begin{proof}
The inequalities in (\ref{eqn:sym3-inequalities-1}) and 
(\ref{eqn:sym3-inequalities-2}) boil down to $-2a - b  \leq m \leq  -a-2b$ which takes care of 
$L\left(s, \Sym^3(\pi) \otimes \chi\right).$ The critical sets for $L\left(s, \pi \otimes \omega_\pi\chi\right)$ and 
$L\left(s, \Sym^2(\pi) \times \pi \otimes \chi\right)$ follow from $n=3$ and $n=2$ cases respectively of 
Thm.\,\ref{thm:critical-compatible}; we leave the details to the reader. 
\end{proof}

\subsubsection{\bf The compatibility condition (\ref{eqn:comp})}

Let $\pi \in \Coh(G_2, \mu^{\sf v})$ with $\mu \in X^+_{00}(T_2).$ Consider two successive symmetric power transfers of 
$\pi$. We address the question of whether Thm.\,\ref{thm:main} is applicable to 
$L(s, \Sym^r(\pi) \times \Sym^{r-1}(\pi))$, i.e.,  after applying Thm.\,\ref{thm:sym-coh}, 
we ask whether the transferred weights $\Sym^r(\mu)$ and $\Sym^{r-1}(\mu)$ 
satisfy the compatibility condition (\ref{eqn:comp}). For this we are seeking an integer $j$ such that for all 
$\iota \in \cE_F$, we should have $\Sym^r(\mu^\iota)^{\sf v} \succ \Sym^{r-1}(\mu^\iota) + j$, i.e., 
$$
-rb^\iota \geq (r-1)a^\iota+j \geq -(r-1)b^\iota-a^\iota \geq (r-2)a^\iota+b^\iota + j \geq \cdots 
\geq (r-1)b^\iota+j \geq -ra^\iota.
$$
Consider two cases: if $\iota$ corresponds to a real place then we want
$$
-ra^\iota-(r-1)b^\iota \ \leq \ j \ \leq \ -(r-1)a^\iota-rb^\iota, 
$$
and if $\iota$ corresponds to a complex place, then not only do we want the above (at $\iota$) but also at the conjugate embedding, since $(a^{\bar{\iota}}, b^{\bar{\iota}}) = ({\sf w}-b^\iota, {\sf w}-a^\iota)$, we would also want 
$$
-(2r-1){\sf w} + (r-1)a^\iota + rb^\iota \ \leq \ j \ \leq \  -(2r-1){\sf w} +ra^\iota + (r-1)b^\iota. 
$$

Such an integer $j$ may not exist, because, if $j$ exists, then both the above inequalities give the necessary condition: 
$$
2ra^\iota + 2(r-1)b^\iota \geq (2r-1){\sf w} \geq 2(r-1)a^\iota + 2rb^\iota, 
$$
which need not be satisfied. (For example, take $F$ to be an imaginary quadratic extension of $\Q$, and 
take $a^\iota = b^\iota$, ${\sf w} \neq 2a^\iota$, and any $r \geq 1.$) 

On the other hand, suppose $F$ is not totally imaginary, i.e., $S_r \neq \emptyset$, then 
${\sf w} = a^\iota + b^\iota$ and all the above conditions are equivalent to 
$-r{\sf w} + b^\iota \leq j  \leq -r{\sf w} + a^\iota.$
For brevity, let $j' = j + r {\sf w}.$ Then we are seeking $j'$ such that $b^\iota \leq j' \leq a^\iota$ for all $\iota$; this is possible because $b^\iota \leq {\sf w}/2 \leq a^\iota$; take $j' = [{\sf w}/2].$


\subsubsection{\bf Proof of Theorem~\ref{thm:sym3}}
\label{sec:proof-thm-sym}
The proof, as explained in the introduction, is to explicate 
Thm.\,\ref{thm:main} for $L_f(\tfrac12+m, \pi \otimes\omega_{\pi}\xi)$ and 
$L_f(\tfrac12+m, \Sym^2(\pi) \times \pi \otimes \xi )$ which we take up in the following two paragraphs:  

\medskip

\paragraph{\bf Thm.\,\ref{thm:main} for $L_f(\tfrac12+m, \pi \otimes\omega_{\pi}\xi)$}
If $\pi \in \Coh(G_2, \mu^{\sf v})$ then $\omega_\pi\xi \in \Coh(G_1, {\rm det}(\mu)^{\sf v}).$ We are in the situation when $n=2$, hence $\epsilon = \eta$ and  
$$
\eta_v \ = \ \omega_{\pi_v}(-1) \cdot \xi_v(-1) \cdot (-1)^{{\sf w}({\rm det}(\mu))/2} \ = \ \xi_v(-1).
$$
Hence $\eta = \epsilon_\xi.$ Thm.\,\ref{thm:main} takes the form: 
\begin{equation}
\label{eqn:gl2-gl1}
\begin{split}
L_f(\tfrac12+m, \pi \otimes\omega_{\pi}\xi) 
& \sim \ p^{\epsilon_m \epsilon_\xi}(\pi) \ \G(\omega_\pi\xi) \ p^{\epsilon_\xi, \epsilon_\xi}(\mu+m, {\rm det}(\mu)) \\
& \sim \ p^{\epsilon_m \epsilon_\xi}(\pi) \ \G(\omega_\pi) \ \G(\xi) \ p^{\epsilon_\xi, \epsilon_\xi}(\mu+m, {\rm det}(\mu)). 
\end{split}
\end{equation}

\medskip

\paragraph{\bf Thm.\,\ref{thm:main} for $L_f(\tfrac12+m, \Sym^2(\pi) \times \pi \otimes \xi )$}
Here, $n=3$, hence $\eta = -\epsilon$ and 
$$
\epsilon_v = \omega_{\Sym^2(\pi_v)}(-1) \cdot (-1)^{{\sf w}(\Sym^2(\mu))/2} = 
\omega_{\pi_v}^3(-1)\cdot (-1)^{\sf w} = 1.
$$
Hence $\epsilon = \epsilon_+$ and $\eta = -\epsilon_+ =: \epsilon_-.$
{\small 
\begin{equation}\label{eqn:gl3-gl2}
\begin{split}
L_f(\tfrac12+m, \Sym^2(\pi) \times \pi \otimes \xi ) & \sim \ 
p^{\epsilon_+}(\Sym^2(\pi)) \cdot p^{-\epsilon_m}(\pi \otimes \xi) \cdot \G(\omega_{\pi \otimes \xi}) \cdot 
p_\infty^{\epsilon_+, \epsilon_-}(\Sym^2(\mu), \mu+m) \\
& \sim 
p^{\epsilon_+}(\Sym^2(\pi)) \cdot p^{-\epsilon_m \epsilon_\xi}(\pi)\G(\xi) \cdot \G(\omega_\pi)\G(\xi)^2 \cdot
p_\infty^{\epsilon_+, \epsilon_-}(\Sym^2(\mu), \mu+m).  
\end{split}
\end{equation}
}
Proof of Theorem~\ref{thm:sym3} follows from (\ref{eqn:gl2-gl1}) and (\ref{eqn:gl3-gl2}).

\end{document}